\documentclass[11pt, reqno]{amsart}
\usepackage{amsmath, amsthm, amscd, amsfonts, amssymb, latexsym, graphicx, color, stmaryrd}
\usepackage{mathrsfs}
\usepackage{eurosym}
\usepackage{float} 
\usepackage{tikz-cd}
\usepackage{tikz-cd}
\usepackage[all,cmtip]{xy}	
\usepackage[bookmarksnumbered, colorlinks, plainpages]{hyperref}

\def\presuper#1#2%
  {\mathop{}%
   \mathopen{\vphantom{#2}}^{#1}%
   \kern-\scriptspace%
   #2}
   
\tikzset{%
    symbol/.style={%
        ,draw=none
        ,every to/.append style={%
            edge node={node [sloped, allow upside down, auto=false]{$#1$}}}
    }
}

\newcommand{\iso}{\xrightarrow{\sim}}	

\newcommand{\Rr}{\mathbb{R}}
\newcommand{\Cc}{\mathbb{C}}

\newcommand{\CDW}{\mathrm{CDW}}
\newcommand{\Cinfty}{\mathbf{C^{\infty}}} 

\renewcommand{\H}{\mathcal{H}} 
\newcommand{\D}{\mathcal{D}}	
\newcommand{\E}{\mathcal{E}}
\newcommand{\Gr}{\mathrm{Gr}}	
\newcommand{\Emb}{\mathscr{E}}	
\newcommand{\EmbV}{\Emb_{\V}} 
\newcommand{\F}{\mathcal{F}} 
\newcommand{\G}{\mathcal{G}}	
\newcommand{\Gop}{\mathrm{\G^{(0)}}} 
\newcommand{\Gpull}{\mathrm{\G^{(2)}}}	

\newcommand{\Gmod}{\mathrm{\Gamma_{mod}}}
\newcommand{\Hop}{\mathrm{\H^{(0)}}} 
 
\newcommand{\Hpull}{\mathrm{\H^{(2)}}}
\newcommand{\Kop}{\mathrm{\K^{(0)}}}

\newcommand{\K}{\mathcal{K}}	

\renewcommand{\O}{\mathcal{O}}
\newcommand{\Mor}{\mathcal{Mor}} 

\renewcommand{\L}{\mathcal{L}} 
\newcommand{\LA}{\L \A} 
\newcommand{\LG}{\L \G}	
\newcommand{\R}{\mathcal{R}}
\newcommand{\SG}{\mathcal{SG}}	

	% ...

\newcommand{\Zop}{Z^t}
\newcommand{\OmegaG}{\Omega_{\G}}
\newcommand{\OmegaZ}{\Omega_{Z}}
\newcommand{\OmegaGh}{\OmegaG^{\frac{1}{2}}}
\newcommand{\OmegaZh}{\OmegaZ^{\frac{1}{2}}}

\newcommand{\OmegaZop}{\Omega_{\Zop}}
\newcommand{\OmegaZoph}{\OmegaZop^{\frac{1}{2}}}

\newcommand{\SigmaG}{\Sigma_{\G}}
\newcommand{\SigmaZ}{\Sigma_{Z}}

\newcommand{\SigmaGh}{\widehat{\Sigma}_{\G}^{\frac{1}{2}}}
\newcommand{\SigmaZh}{\widehat{\Sigma}_{Z}^{\frac{1}{2}}}

\newcommand{\V}{\mathcal{V}}		
\newcommand{\W}{\mathcal{W}}		
\newcommand{\A}{\mathcal{A}}		
\newcommand{\N}{\mathcal{N}}		
\newcommand{\End}{\mathrm{End}}		
\renewcommand{\Mor}{\mathrm{Mor}}
\newcommand{\U}{\mathcal{U}}		
\newcommand{\B}{\mathcal{B}}
\newcommand{\C}{\mathcal{C}}
\newcommand{\Diff}{\mathrm{Diff}}	
\newcommand{\supp}{\mathrm{supp}}	

\newcommand{\PsiZ}{\mathrm{\Psi(Z)^{\G}}}
\newcommand{\PsiZop}{\mathrm{\presuper{\G}{\Psi(Z^t)}}}
\newcommand{\PsiZm}{\mathrm{\Psi^m(Z)^{\G}}}

\newcommand{\pr}{\mathrm{pr}}		
\newcommand{\prop}{\pr^{(0)}}
\newcommand{\lift}{\mathrm{lift}}
\newcommand{\liftcl}{\widetilde{\lift}}

\newcommand{\scal}[2]{\langle #1, #2 \rangle}	
\newcommand{\id}{\operatorname{id}}	
\newcommand{\fop}{f^{(0)}} 

\newcommand{\codim}{\operatorname{codim}}

\newcommand{\flip}{\operatorname{f}}
\newcommand{\pair}{\operatorname{pair}}

\newcommand{\simM}{\sim_{\mathcal{M}}}

\newcommand{\Sm}{C^\infty}
\newcommand{\Smc}{C^\infty_{c}}

\newcommand{\dd} {{\mathrm{d}}}
\newcommand{\Fu} {{\mathcal{F}}}

\newcommand{\BB}{{\mathbb{B}}}
\newcommand{\BBd}{\BB^d}

\newcommand{\RR}{{\mathbb{R}}}
\newcommand{\RRd}{\RR^d}

\newcommand{\SSS}{{\mathbb{S}}}
\newcommand{\SSSd}{{\mathbb{S}^{d-1}}}

\newcommand{\cl}{{\mathrm{cl}}}

\newtheorem{Thm}{Theorem}[section]
\newtheorem{Lem}[Thm]{Lemma}
\newtheorem{Prop}[Thm]{Proposition}
\newtheorem{Cor}[Thm]{Corollary}
\theoremstyle{definition}
\newtheorem{Def}[Thm]{Definition}
\newtheorem{Exa}[Thm]{Example}

\newtheorem{Rem}[Thm]{Remark}

\begin{document}
\setcounter{page}{1}

\title{Quantization on manifolds with an embedded submanifold}

\author[Karsten Bohlen, Ren\'e Schulz]{Karsten Bohlen, Ren\'e Schulz}

\address{$^{1}$ Universit\"at Regensburg, Germany}
\email{\textcolor[rgb]{0.00,0.00,0.84}{karsten.bohlen@mathematik.uni-regensburg.de}}

\address{$^{2}$ Leibniz Universit\"at Hannover, Germany}
\email{\textcolor[rgb]{0.00,0.00,0.84}{rschulz@math.uni-hannover.de}}





\begin{abstract}
We investigate a quantization problem which asks for the construction of an algebra 
for relative elliptic problems of pseudodifferential type associated to smooth embeddings. 
Specifically, we study the problem for embeddings in the category of compact manifolds with corners. The construction of a calculus for elliptic problems
is achieved using the theory of Fourier integral operators on Lie groupoids. 
We show that our calculus is closed under composition and furnishes a so-called noncommutative completion of the given embedding.
A representation of the algebra is defined and the continuity of the operators in the algebra on suitable Sobolev spaces is established.
\end{abstract} \maketitle


\section{Introduction}

The quantization problem for pseudodifferential operators on a given manifold with singularities is the problem of finding a reasonable class of symbols (the classical side)
and of pseudodifferential operators (the quantum side) and a map from the symbols to the operators (quantization procedure) fulfilling certain physically or mathematically motivated assumptions. 
A natural question is whether a given quantization scheme has functorial properties. It is known that in a certain sense geometric quantization can be realized as a functor, cf. \cite{L}, \cite{H}. 
In more sophisticated problems of quantization the base geometry of the manifold under consideration is of a more singular nature.
There are various types of geometric singularities a given manifold can possess, e.g. conical, cuspidal, fibered cusp and edge type singularities. 
We therefore consider a pseudodifferential calculus on a given non-compact manifold with controlled singular geometry.
Various different pseudodifferential calculi have been introduced in the literature, adapted to each type of singularities.
Examples are the $b$-calculus of Melrose and the cone calculus of Schulze as well was the (fibered-) cusp calculus. 
Recently it has been established that Lie groupoids provide a framework for the study of pseudodifferential calculi on many types of singular manifolds simultaneously.
Lie groupoids provide at the same time a natural language for the resolution of geometric singularities.
For a broad class of singular manifolds (so-called Lie manifolds) one can always associate a Lie groupoid and 
obtain a general pseudodifferential calculus, \cite{NWX}, \cite{ALN}. 

In this work we consider the quantization problem for manifolds relative to an embedded submanifold.
The general definition of a quantization consists of two sides: \emph{i)} the \emph{classical side} are the symbols which are smooth
functions living on the cotangent space, fulfilling certain uniform growth estimates, \emph{ii)} the \emph{quantum side} of operators which form
an algebra with regard to composition and adjoint.
There are two problems one has to face: \emph{i)} finding the right class of symbol spaces and \emph{ii)} finding the right
kind of natural quantization procedure.

In the standard case a calculus for elliptic problems associated in this way can be constructed, an example is Boutet de Monvel's
calculus and another is given by Nazaikinskii and Sternin \cite{NS}.
We found it more convenient to base our constructions on \cite{NS}, due to recent work of Lescure and Vassout \cite{LV}, where
the authors introduce Fourier integral operators on Lie groupoids. 

Connections between noncommutative geometry and boundary value problems have been explored in other works of Aastrup, Schrohe and Nest with a focus on index theory, \cite{ANS}, \cite{ANS2}
and recent work of Debord and Skandalis, \cite{DS}, \cite{DS2}.
They construct deformation groupoids, generalizations of Connes'es tangent groupoid, to define operators of Boutet de Monvel
type in a very general setting. This is based on the observation that an algebra for boundary
value problems, like Boutet de Monvel's algebra, can be viewed as a generalized morphism (or Morita equivalence)
of $C^{\ast}$-algebras. Ours is a different viewpoint, more in line with that of the first author \cite{B}, \cite{B2}, in that we do not use deformation groupoids. 
But we still apply the same basic philosophy of describing a generalized morphism as an abstraction of an algebra for boundary value problems. 
Earlier constructions in the literature focus on particular cases of singularites using tools individually designed for each case.
The program we want to announce in this work is an attempt to close a natural gap by uniting the theory of Fourier integral operators on groupoids and recent approaches to boundary value problems using Lie groupoids
to a general quantization procedure in one common framework. 

\subsection*{Outline}
Let us outline the contents of this paper in more detail. As a base geometry, we consider a \emph{Lie manifold} which is a compact manifold with corners
$M$ endowed with a Lie algebra of vector fields $\V \subset \Gamma(TM)$ tangent
to the boundary strata of $M$, and a vector bundle $\A \to M$ which has 
canonically a Lie algebroid structure such that $\Gamma(\A) \cong \V$.
Such manifolds usually arise as the compactification of a non-compact manifold with geometric singularities. \\
From \cite{AIN} we take the notation of an embedded, codimension $\nu$ \emph{Lie} submanifold $j \colon Y \hookrightarrow M$. In particular $Y$ is \emph{transverse} to the strata of $M$ and it is with regard to $Y$ that we consider
boundary value problems. \\
We further endow $M$ (and hence $Y$) with a Riemannian metric compatible with the Lie structure.\\
The problem now is to describe the elliptic theory associated with a pair $(M, Y)$ as above.
Simple candidates for elliptic operators are Laplacians $\Delta = \Delta_g$ for some compatible metric $g$ on $M$
and $\Delta_{\partial} = \Delta_{g_{\partial}}$ for some compatible metric $g_{\partial}$ on $Y$. 
Here the operators act on suitable Sobolev spaces $H_{\V}^s(M)$ and $H_{\W}^s(Y)$ as introduced in \cite{AIN}.\\
These pairs form the example of a \emph{relative elliptic operator} is the diagonal matrix
\[
\begin{pmatrix} \Delta_g & 0 \\ 0 & \Delta_{\partial} \end{pmatrix} \colon \begin{matrix} H_{\V}^{s_1}(M) \\ \oplus \\ H_{\W}^{s_2}(Y) \end{matrix} \to \begin{matrix} H_{\V}^{t_1}(M) \\ \oplus \\ H_{\W}^{t_2}(Y) \end{matrix}. 
\]
More generally, we would like to describe an operator algebra from the given data $j \colon Y \hookrightarrow M$ that captures
all non-commutative phenomena which arise in the study of elliptic boundary value problems (BVP's). 
We can call this \emph{non-commutative completion} of a manifold with corners and regular boundary.
Inspired by the classical BVP's (e.g. the Dirichlet and Neumann problem) we introduce additional operators.
Let $j^{\ast}$ be the restriction to $Y$ operator. 
By \cite{AIN}, Theorem 4.7 we have $j^{\ast} \colon H_{\V}^{s}(M) \to H_{\W}^{s-\frac{\nu}{2}}(Y)$ continuously for 
$s > \frac{\nu}{2}$. 
Denote by $j_{\ast}$ the formal adjoint of $j^{\ast}$. 
Unfortunately, as can be checked easily already in the standard case, operators of the form
\begin{equation}
\label{eq:opmat}
\begin{pmatrix} \Delta_g & j_{\ast} \\ j^{\ast} & \Delta_{\partial} \end{pmatrix}
\end{equation}
are not closed under composition. 
The reason is that operators in the upper left corner, the singular Green operators, of the form $P j_{\ast} j^{\ast} Q$, for
$P$ and $Q$ (pseudo-)differential operators, are not again (pseudo-)differential operators. \\
The approach of \cite{NS} is to treat each entry in \eqref{eq:opmat} as a Fourier integral operator, meaning operators whose Schwartz kernels are Lagrangian distributions associated to naturally occurring canonical relations.\\
We want to solve the problem of constructing such an algebra for boundary value problems also in the case of more singular geometric data $M$ and $Y$. Because the manifold has corners, the arising Schwartz kernels are more complicated and a resolution of 
the corner structure is required. This \emph{resolution} is fascilitated by the theory of Lie groupoids and notions from noncommutative geometry.
Our strategy for the construction of the algebra is briefly described as follows:\\
We consider from the embedding $j \colon Y \hookrightarrow M$ of Lie manifolds the corresponding Lie groupoids $\H \rightrightarrows Y$
integrating the Lie structure of $Y$ and $\G \rightrightarrows M$ integrating the Lie structure of $M$.
In the next step we lift the embedding morphism $j$ to a \emph{generalized morphism} $\H \dashrightarrow \G$ in the category
of Lie groupoids.
To describe the operators of the algebra, we work with Fourier integral operators which are associated to the 
Coste-Dazord-Weinstein (CDW) groupoids associated to $\H$ and $\G$.
We rely on the theory of Fourier integral operators on Lie groupoids from \cite{LMV} and \cite{LV}.
From the generalized morphism $\H \dashrightarrow \G$ in the category of Lie groupoids we obtain a generalized morphism
$T^{\ast} \H \dashrightarrow T^{\ast} \G$ (CDW-groupoids) in the category of symplectic groupoids.
Then we describe in a functorial way a generalized morphism of $C^{\ast}$-algebras over $\H$ to $\G$. 
This functor (noncommutative completion) has the minimal properties expected of an algebra for boundary value problems.
In this framework the singular Green operators appear as natural objects which are Morita equivalent to pseudodifferential operators on the given submanifold, see also \cite{DS}.\\
Our quantization procedure differs somewhat from that in \cite{NS}. Namely, in \cite{NS} the authors use Maslov quantization based on the standard theory of Fourier integral operators.
Here we face the problem that this type of quantization is not available on singular manifolds, except if we use Lie groupoids. We also define a Kohn-Nirenberg quantization procedure for the representations of operators defined on groupoids.
We describe the representation of a non-commutative completion which is a natural transformation, giving an extension of the representation theorem
of the groupoid algebra of pseudodifferential operators by Ammann-Lauter-Nistor in \cite{ALN}. \\
The paper is organized as follows. We give a brief introduction to the Kohn-Nirenberg quantization for
relative elliptic theory in the standard case of a compact manifold with boundary in the second Section. In the third Section we recall the definition of a manifold with corners and 
define a suitable class of embeddings. 
The fourth Section is in part a short summary of the Muhly-Renault-Williams theorem and we also formally define a noncommutative
completion of a manifold with boundary.
In the fifth Section we give several examples of Lie manifolds and the appropriate groupoids integrating particular Lie structures.
We recall the definition of the Coste-Dazord-Weinstein (CDW) groupoid in the sixth Section and we also consider actions
of CDW-groupoids, the functoriality of symplectic groupoids.
In the seventh Section we discuss Fourier integral operators on Lie groupoids and groupoid actions.
We give the conditions on the orbit foliation needed for composability, the equivariance condition and we define
the twisted product of symbols.
In the eigth Section we construct the groupoid algebra for boundary value problems or embeddings in the
category of Lie manifolds. We prove closedness under composition using the theory of the previous section.
In the final Section we show that our algebra has the properties of a noncommutative completion as defined in the
third Section. We collect some information about radial blowups and symbol spaces for Fourier integral operators in the appendix.

\section{Kohn-Nirenberg quantization and relative elliptic theory}

\label{section:KN}

In this section, we will address the case where the base geometry has no singularity, i.e. we reformulate the calculus of relative elliptic theory \cite{NS} in a form that is suitable for generalization to groupoids. The setup is as follows. Suppose $X\stackrel{i}{\hookrightarrow} M$ is a $\Sm$-embedding of smooth compact manifolds, $\dim(X)=d\leq \dim(M)=n$. The key idea of the calculus is to realize singular operations on pseudo-differential operators as compositions with certain Fourier integral operators associated to Lagrangian submanifolds that occur as natural geometric objects.\\
Then $T^{\ast} M$, the natural space for quantization, contains as natural the spaces 
\[
T^\ast X, \ T^\ast_{i(X)} M, \ N i(X) 
\]
and we are led to consider how these subspaces are (microlocally) mapped into each other under the action of an operator. Therefore we 
study the following five invariantly defined Lagrange manifolds:
\begin{align*}
\Lambda_{\Psi},\Lambda_g&\subset T^*M\times T^*M \qquad  \Lambda_c \subset T^*M\times T^*X\\
\Lambda_b&\subset T^*X\times T^*M \qquad \Lambda_{\partial} \subset T^*X\times T^*X.
\end{align*}
\begin{enumerate}
\item $\Lambda_{\Psi}=(N^*\Delta M)^\prime\subset T^*M\times T^*M$ is the twisted conormal bundle to the diagonal embedding $M\hookrightarrow \Delta_M\subset M\times M$.
\item $\Lambda_g=(N^*\Delta_{\iota(X)})^\prime\subset T^*M\times T^*M$ is the twisted conormal bundle to the composed embedding $X\hookrightarrow M\hookrightarrow \Delta_M\subset M\times M$.
\item $\Lambda_b=(N^*\mathrm{graph}(\iota))^\prime\subset T^*X\times T^*M$ is the twisted conormal bundle to the graph of the embedding $\iota$.
\item $\Lambda_c=\,^t\Lambda_b=(\,^tN^*\mathrm{graph}(\iota))^\prime\subset T^*M\times T^*X$ is the twisted conormal bundle to the transpose of the graph of the embedding $\iota$.
\item $\Lambda_{\partial}=(N^*\Delta X)^\prime\subset T^*X\times T^*X$ is the twisted conormal bundle to the diagonal embedding $X\hookrightarrow \Delta_X\subset X\times X$.
\end{enumerate}
%
In local canonical coordinates adapted to the embedding in which $X$ is given by $\{x^\prime,0\}$ the space $T^{\ast} M \times T^{\ast} M$ is given by $\{(x^{\prime}, x^{\prime \prime}, \xi^{\prime}, \xi^{\prime \prime}), (y^{\prime}, y^{\prime \prime}, \eta^{\prime}, \eta^{\prime \prime})\}$. The Lagrangians then take the form 
\begin{equation}
\label{eq:lagcoord}
\begin{aligned}
\Lambda_{\Psi}&=\left\{\big((x^\prime,x^{\prime\prime},\xi^\prime,\xi^{\prime\prime}),(x^\prime,x^{\prime\prime},\xi^\prime,\xi^{\prime\prime})\big)\right\},\\
\Lambda_g&=\left\{\big((x^\prime,0,\xi^\prime,\xi^{\prime\prime}),(x^\prime,0,\xi^\prime,\eta^{\prime\prime})\big)\right\},\\
\Lambda_b&=\left\{\big((x^\prime,\xi^\prime),(x^\prime,0,\xi^\prime,\eta^{\prime\prime})\big)\right\},\\
\Lambda_c&=\left\{\big((x^\prime,0,\xi^\prime,\xi^{\prime\prime}),(x^\prime,\xi^\prime)\big)\right\},\\
\Lambda_{\partial}&=\left\{\big((x^\prime,\xi^\prime),(x^\prime,\xi^\prime)\big)\right\}.
\end{aligned}
\end{equation}

The operators in the calculus of relative elliptic theory exhibit singularities carried by the above Lagrangians. Indeed, they can be expressed in the form
\begin{equation}
\label{eq:Sterninops}
A=\left(\begin{array}{cc}
A_{\Psi}+A_g & A_c \\ 
A_b & A_{\partial}
\end{array}\right) 
\end{equation}
where the Schwartz kernel of each $A_\bullet$ is a Fourier Integral operator associated to $\Lambda_\bullet$. In particular, $A_{\Psi}$ and $A_{\partial}$ are ordinary pseudodifferential operators on $X$ and $M$ respectively. 
We recall the composability relations between the underlying canonical relations for operators of the form \eqref{eq:Sterninops}. The compositions are summarized in Table \ref{tab:Lagrangians} (see \cite{NS}), where in the sense of matrix entries we have $\Lambda_{ij}=\Lambda_{j1}\circ \Lambda_{ij}$ for $i,j>1$.
\begin{table}[H]
\caption{Lagrangians}
\centering
\begin{tabular}{|c||c|c|c|c|c|}
\hline 
$\circ$  & $\Lambda_{\partial}$ & $\Lambda_{\Psi}$ & $\Lambda_b$ & $\Lambda_c$ & $\Lambda_g$ \\ 
\hline 
\hline
$\Lambda_{\partial}$  & $\Lambda_{\partial}$ & - & $\Lambda_b$ & - & - \\ 
\hline 
$\Lambda_{\Psi}$  & - & $\Lambda_{\Psi}$ & - & $\Lambda_c$ & $\Lambda_g$ \\ 
\hline
$\Lambda_g$  & - & $\Lambda_b$ & - & $\Lambda_{\partial}$ & $\Lambda_b$ \\ 
\hline 
$\Lambda_c$  & $\Lambda_g$ & - & $\Lambda_g$ & - & - \\  
\hline 
$\Lambda_g$  & - & $\Lambda_g$ & - & $\Lambda_c$ & $\Lambda_g$ \\  
\hline 
\end{tabular} 
\label{tab:Lagrangians}
\end{table}

\subsection{Quantization of symbols}

In \cite{NS}, the quantization procedure adopted to produce operators from symbols is as follows: to the given Lagrangian submanifolds, they associate a (``quantized'') measure and construct operators associated to symbols via Maslov's canonical operator. Here, we use instead a modified Kohn-Nierenberg quantization: since the submanifolds at hand are conormals to certain embedded submanifolds, we may use a normal decomposition and the exponential map to produce a more direct approach. \\
We first locally realize $X\subset M$ in special coordinates on $M$, $x=(x^\prime,x^{\prime\prime})$ as the set where $x^\prime$ vanishes. We then equip $M$ with a metric of product type, meaning it splits in a neighbourhood of $X$, where $X=\{(x^\prime,0)\}$ in coordinates, into $g=g_1(x)(dx^\prime)^2+g_2(x)(dx^{\prime\prime})^2$.\\
Following \cite{Mel, ALN, S}, we consider the map 
$$\Phi: TM\rightarrow M\times M,\ \Phi(x,\xi)=(x,\exp_x(-\xi)).$$
Since $X$ and $M$ are assumed compact and without boundary, the injectivity radius of their respective exponential functions $\exp_X$ and $\exp_M$ is bounded from below. Hence $\Phi$ is, when restricted to a small open neighbourhood $U$ of the zero-secton $M\subset TM$ a diffeomorphism onto an open neighbourhood of the diagonal $\Delta_M$. We may assume that $U=(TM)_r$ are the vectors of length $\leq r$ for some $r>0$, and $\Phi|_{(TM)_r}$ is called the \textit{Riemann-Weyl-fibration}. We identify the exponential map on $X$ equipped with the pulled-back metric under $\iota$ with that on $M$ restricted to $TX\subset TM$. \\
Restriction of $\Phi$ yields the following isomorphisms:
\begin{enumerate}
\item $\Phi_{\Psi}=\Phi|_{(TM)_r}$: maps $(TM)_r$ onto an open neighbourhood of the diagonal $\Delta_M\subset M\times M$. 
\item $\Phi_g=\Phi|_{(T\iota(X))_r}$: maps the tangent $T\iota(X)\subset TM$ onto an open neighbourhood of the diagonal $\Delta_{\iota(X)}\subset M\times M$.
\item $\Phi_b=\Phi|_{(T_XM)_r}$: maps the restriction of $(TM)_r$ to $X$ onto an open neighbourhood of the graph of $\iota(X)$ in $X\times M$.
\item We set $\Phi_c=\Phi_b^t$.
\item $\Phi_{\partial}=\Phi|_{(TX)_r}$: maps $(TX)_r$ onto an open neighbourhood of the diagonal $\Delta_X\subset X\times X$. 
\end{enumerate}
Recall that if $A$ is a vector bundle over $M$ (or $X$) equipped with a metric $g$, $A^*$ the canonical dual bundle, then the fibre-Fourier transform on $\Smc(A)$ is given by 
$$\Fu_{\mathrm{fiber}}^{-1}\varphi(\xi)=(2\pi)^{-\dim(M)}\int_{\pi(\xi)=\pi(\zeta)} e^{i(\xi,\zeta)_g}\varphi(\zeta)\dd\zeta.$$  
This map is then extended to a map from certain symbol spaces $S(A^*)$, to be introduced in the next section, to the space of distributions on $A$ by
$$\langle \Fu_{\mathrm{fiber}}^{-1}a,\varphi\rangle := \langle a,\Fu_{\mathrm{fiber}}^{-1}\varphi\rangle \qquad a\in S(A^*),\, \varphi\in\Smc(A)$$ 
The quantization of a symbol $a$ in a suitable symbol space $S(C_\bullet)$ with $\bullet\in \{M,g,X,b,c\}$ is then $\mathrm{Op}(a)=(\Phi_\bullet)_*(\chi \Fu_{\mathrm{fiber}}^{-1}a)$, where $\chi$ is a cut-off supported in $U=(TM)_r$ with $\chi\equiv 1$ on $M$ (identified with the zero section). This, when the appropriate symbol classes are quantized (see Section \ref{section:SymbolCalc} below) yields five classes of operators,
%
with distributional kernels on $M\times M$, $M\times M$, $X\times M$, $M\times X$ and $X\times X$, respectively.
In adapted local coordinates, where $X=\{x^\prime,0\}$ and $TX=\{x^\prime,0,\xi^\prime,0\}$ and the metric takes the standard form on $\RR^d$, we can then write for the kernel of $A_\bullet$:
\begin{align*}
K_{A_{\Psi}}(x,y)&=(2\pi)^{-n}\int e^{i\xi(x-y)} \chi(x,x-y) a(x,\xi)\dd\xi,\\
K_{A_g}(x,y)&=(2\pi)^{-d-n}\int e^{i\xi^\prime(x^\prime-y^\prime)+i\xi^{\prime\prime}x^{\prime\prime}-i\eta^{\prime\prime} y^{\prime\prime}} \chi(x-y) a(x^\prime,\xi^\prime,\xi^{\prime\prime},\eta^{\prime\prime})\dd\xi^\prime\dd\xi^{\prime\prime}\dd \eta^{\prime\prime},\\
K_{A_b}(x^\prime,y)&=(2\pi)^{-n}\int e^{i\xi^\prime(x^\prime-y^\prime)-i\xi^{\prime\prime}y^{\prime\prime}} \chi(x^\prime,x^\prime-y^\prime) a(x^\prime,\xi)\dd\xi,\\
K_{A_c}(x,y^\prime)&=(2\pi)^{-n}\int e^{i\xi^\prime(x^\prime-y^\prime)+i\eta^{\prime\prime}x^{\prime\prime}} \chi(x,x^\prime-y^\prime) a(x^\prime,\xi)\dd\xi^\prime\dd\eta^{\prime\prime},\\
K_{A_{\partial}}(x^\prime,y^\prime)&=(2\pi)^{-d}\int e^{i\xi^\prime(x^\prime-y^\prime)} \chi(x^\prime,x^\prime-y^\prime) a(x^\prime,\xi^\prime)\dd\xi^\prime.\\
\end{align*}

The Operators associated to composable Lagrangians as in Table \ref{tab:Lagrangians} admit compositions when their symbols are compactly supported or rapidly decaying in all variables. For example, an operator $A_c$ composed with some $A_\partial$ yields an operator of type $A_g$. The symbol of the compositions can be computed again, modulo remainders of lower order, in terms of so called twisted products. We do not list these explicit integral formulae for these, which are contained in \cite{SS}, since we will carry out the compositions in greater generality in Section \ref{section:SymbolCalc}. However, we note that since the compositions are merely clean, a rapid decay condition needs to be put on certain subsets of variables in the symbols such that the products exist. \\

\subsection{Adapted symbol spaces}

We now introduce the symbol spaces adapted to the quantization of the previous section. As in \cite{NS} we search for smooth functions on the respective Lagrangians that in the adapted coordinates of \eqref{eq:lagcoord} satisfy the estimates:

\begin{equation}
\label{eq:lagcoord2}
\begin{aligned}
S^m(\Lambda_\Psi)&:\quad |\partial_x^\alpha\partial_\xi^\beta a(x,\xi)|\lesssim (1+|\xi|)^{m-|\beta|}\\
S^{k,l,m}(\Lambda_g)&:\quad |\partial_{x^\prime}^\alpha \partial_{\xi^\prime}^\beta \partial_{\xi^{\prime\prime}}^\gamma\partial_{\eta^{\prime\prime}}^\delta a(x^\prime,\xi^\prime,\xi^{\prime\prime},\eta^{\prime\prime})|\\
&\qquad\qquad\lesssim (1+|\xi^\prime|)^{k-|\beta|}(1+|\xi^\prime|+|\xi^{\prime\prime}|)^{m-|\gamma|}(1+|\xi^\prime|+|\eta^{\prime\prime}|)^{l-|\gamma|}\\
S^{k,l}(\Lambda_b)&:\quad |\partial_{x^\prime}^\alpha \partial_{\xi^\prime}^\beta \partial_{\eta^{\prime\prime}}^\gamma a(x^\prime,\xi^\prime,\eta^{\prime\prime})|\lesssim (1+|\xi^\prime|)^{k-|\beta|}(1+|\xi^\prime|+|\eta^{\prime\prime}|)^{l-|\gamma|}\\
S^{k,m}(\Lambda_c)&:\quad |\partial_{x^\prime}^\alpha \partial_{\xi^\prime}^\beta \partial_{\xi^{\prime\prime}}^\gamma a(x^\prime,\xi^\prime,\xi^{\prime\prime})|\lesssim (1+|\xi^\prime|)^{k-|\beta|}(1+|\xi^\prime|+|\xi^{\prime\prime}|)^{m-|\gamma|}\\
S^m(\Lambda_\partial)&:\quad |\partial_{x^\prime}^\alpha\partial_{\xi^\prime}^\beta a(x^\prime,\xi^\prime)|\lesssim (1+|\xi^\prime|)^{m-|\beta|}
\end{aligned}
\end{equation}

For the calculus, we also need a meaningful definition of a principal term of these amplitudes and an invariant definition. Both can be achieved by introducing the polyhomogeneous versions of these symbol spaces as pullbacks of spaces of smooth functions on certain blow-up spaces, which we sketch in Appendix \ref{sec:symbinv}.

In order to obtain a meaningful calculus of operators, these are then realized as bounded operators between appropriate $L^2$-based Sobolev spaces, see \cite{NS}. For our purposes, it is enough to consider the $L^2$-bounded case.\\

\begin{Prop}
Let $k_g > 0, \ k_c > 0, \ k_b > 0, \ m_g < -\frac{\nu}{2}, \ m_g + k_g > -\frac{\nu}{2}$. Then the matrix of operators
\begin{equation}
A=\left(\begin{array}{cc}
A_{\Psi}+A_g & A_c \\ 
A_b & A_{\partial}
\end{array}\right) 
\end{equation}
defines a continuous operator on $L^2(M)\oplus L^2(X)$.
\end{Prop}

We will give a proof of this fact in our more general setting in the main body of this paper, see Theorem \ref{Thm:algebra}. There is a way to define a global pseudodifferential calculus on manifolds with a Lie structure introduced by Ammann, Lauter, Nistor cf. \cite{ALN}. The definition of the calculus is based on the Kohn-Nirenberg quantization and essential features, in particular the closedness under composition, is proved using groupoid
techniques. Based on their approach the relative pseudodifferential calculus studied in Section \ref{section:KN} admits a natural generalization to embeddings of Lie manifolds. To properly define this however we will need a theory of Fourier integral operators on Lie groupoids, studied in Section \ref{FIO}. 

\section{Manifolds with corners}
\label{Section:3}

In this section we give a brief overview of the category of manifolds with corners. For a more detailed presentation we refer e.g. to \cite{Mel}. 
The topological structure of a manifold with corners is that of a space $M$ with a finite number of embedded (intersecting) codimension one hypersurfaces. 
As such $M$ is locally modelled around each point $x_0$ by relatively open subsets of $[-1,1]^k \times \Rr^{n-k}$, where the minimal $k = \mathrm{ord}(x_0)$ is the depth of $x_0$.
We require the transition maps to be smooth and obtain a smooth structure on $M$. A natural choice of structure preserving maps in this case are the \emph{$b$-maps} in \cite{Mel}. There are various other
inequivalent choices of structure preserving maps for manifolds with corners used in the literature. For an overview and more in depth discussion of such issues we refer to \cite{J}. In this way the manifolds with corners form a category, which we henceforth denote by $\Cinfty$, a subcategory of which are compact manifolds with boundary. 
In later sections we will study categories of groupoids, where the objects are groupoids which are \emph{internal to} the \emph{ambient category} $\Cinfty$. 
A compact manifold with corners can also be defined in an \emph{extrinsic} way with reference to an ambient manifold and boundary defining functions theron. 

\begin{Def}
A Hausdorff topological space $M$ is a \emph{manifold with embedded corners} if the following conditions hold.

\emph{i)} We have a topological isomorphism $i \colon M \to i(M) \subset \widetilde{M}$, where $\widetilde{M}$ is a compact manifold without boundary.
Then the smooth structure on $M$ is given by $C^{\infty}(M) = i^{\ast} C^{\infty}(\widetilde{M})$.

\emph{ii)} The boundary defining functions $\{\rho_i\}_{i \in I}$ are fixed as maps $\rho_i \in C^{\infty}(\widetilde{M}), \ i \in I$ with
\[
i(M) = \{y \in \widetilde{M} : \rho_i(y) \geq 0\} = \bigcap_{i \in I} \{\rho_i \geq 0\}. 
\]

\emph{iii)} For each $J \subset I, x \in \widetilde{M}$ with $\rho_j(x) = 0$ for each $j \in J$ it follows that $\{d \rho_j(x)\}_{j \in J}$ is linearly
independent.

\label{Def:MWC}
\end{Def}

We focus on the special case of Lie manifolds below. For manifolds with corners we consider a particular type of submersion.

\begin{Def}
A \emph{tame submersion} between two manifolds with corners $M$ and $N$ is given by a smooth map $f \colon M \to N$ such that $df$ is everywhere surjective and $v$ is an inward pointing tangent vector of $M$ if and only if $df(v)$ is an inward pointing tangent vector of $N$.
\label{Def:subm}
\end{Def}

\begin{Lem}
Let $f \colon M \to N$ be a tame submersion between manifolds with corners $M$ and $N$. Then
for each $y \in N$ the fibers $f^{-1}(y)$ are smooth manifolds without corners.
\label{Lem:subm}
\end{Lem}

\begin{proof}
See \cite{LN}, p.4. 
\end{proof}

For a given compact manifold with corners $M$ we fix the notation $\F(M)$ to denote the collection of boundary faces of $M$.

\begin{Def}
Let $M_i, \ i =1,2$ be compact manifolds with corners and $j \colon M_1 \hookrightarrow M_2$ a $C^{\infty}$-embedding. 

\emph{i)} The embedding is called \emph{closed} if for each boundary face $F$ of $M_1$ there is a boundary face $G$ of $M_2$ such that $F$ is the connected component of
$j^{-1}(G)$.

\emph{ii)} The embedding $j \colon M_1 \hookrightarrow M_2$ is \emph{transverse relative to $M_2$} if for any boundary face $F$ of $M_2$
and any $y \in F \cap j(M_1)$ we have that the space $T_y M$ is the non-direct sum of $T_y j(M_1)$ and $T_y F$. 

\emph{iii)} Denote by $\varphi(j) \colon \F(M_2) \to \F(M_1)$ the map of boundary faces given by $\F(M_2) \ni F \mapsto M_1 \cap F \in \F(M_1)$. 
The embedding $j \colon M_1 \hookrightarrow M_2$ is called \emph{admissible} if $\varphi(j)$ is bijective. 
\label{Def:transv}
\end{Def}

\begin{Exa}
A typical situation in which such an embedding arises is if one takes a submanifold of $\RRd$ that ``extends up to infinity''. If one then compactifies $\RRd$, the submanifold will hit the new-formed boundary at infinity. Figure \ref{fig:embex} is an example of a transverse relative embedding of compact manifolds with non-trivial corners. Here, $M$ is the ``torus with corners'' $\mathbb{S}^1\times Y$

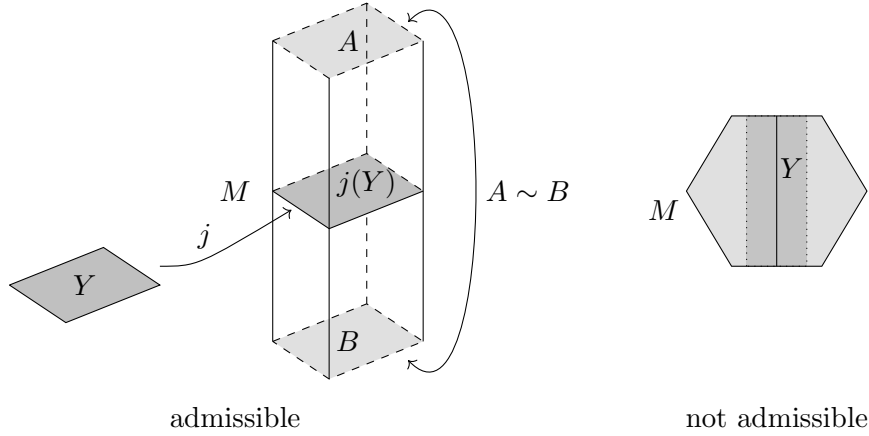
\begin{figure}[H]
\begin{center}
\begin{tikzpicture}
    \begin{scope}[shift={(-3.5,-1.25)}]	
    \fill[opacity=0.25] (0,0) -- (1.25,0.5)  -- (2,0) -- (0.75,-0.5) -- (0,0);
	\draw (0,0) -- (1.25,0.5)  -- (2,0) -- (0.75,-0.5) -- (0,0);
	\node  at (1,0) {$Y$};
	\end{scope}
	\node  at (-0.5,-3) {admissible};
	\draw[->] (-1.5,-1) .. controls (-1,-1) and (-1,-1) .. (0.25,-0.25) node [above, midway] {$j$}; 
    \fill[opacity=0.25] (0,0) -- (1.25,0.5)  -- (2,0) -- (0.75,-0.5) -- (0,0);
	\draw[dashed] (0,0) -- (1.25,0.5)  -- (2,0);
    \draw (2,0) -- (0.75,-0.5) -- (0,0);
    \node  at (-0.5,0) {$M$};
    \node  at (1.25,0.1) {$j(Y)$};
    \fill[opacity=0.125] (0,-2) -- (1.25,-1.5)  -- (2,-2) -- (0.75,-2.5) -- (0,-2);
    \fill[opacity=0.125] (0,2) -- (1.25,2.5)  -- (2,2) -- (0.75,1.5) -- (0,2);
    \draw[dashed] (0,-2) -- (1.25,-1.5)  -- (2,-2) -- (0.75,-2.5) -- (0,-2);
    \draw[dashed] (0,2) -- (1.25,2.5)  -- (2,2) -- (0.75,1.5) -- (0,2);
    \node  at (1,2) {$A$};
    \node  at (1,-2) {$B$};
    \draw[<->] (1.8,2.25) .. controls (3,3.5) and (3,-3.5) .. (1.8,-2.25) node [right, midway] {$A\sim B$}; 
    \draw (0,-2) -- (0,2);
    \draw (2,-2) -- (2,2);
    \draw (0.75,-2.5) -- (0.75,1.5);
    \draw[dashed] (1.25,2.5) -- (1.25,0.5);
    \draw[dashed] (1.25,-0.35) -- (1.25,-1.5);
    
    \begin{scope}[shift={(5.5,0)}]
        \node  at (1.4,0.3) {$Y$};
        \node  at (-0.3,-0.25) {$M$};
\draw (0,0) -- (0.6,1) -- (1.8,1) -- (2.4,0) -- (1.8,-1) -- (0.6,-1) -- (0,0);
\fill[opacity=0.125] (0,0) -- (0.6,1) -- (1.8,1) -- (2.4,0) -- (1.8,-1) -- (0.6,-1) -- (0,0);
\draw (1.2,-1) -- (1.2,1);
\draw[dotted] (0.8,-1) -- (0.8,1) -- (1.6,1) -- (1.6,-1) -- (0.8,-1);
\fill[opacity=0.125] (0.8,-1) -- (0.8,1) -- (1.6,1) -- (1.6,-1) -- (0.8,-1);    
\node  at (1.2,-3) {not admissible};
    \end{scope}
    
\end{tikzpicture}
\end{center}
\caption{Examples of transverse relative embeddings}
\label{fig:embex}
\end{figure}
\label{Exa:fig}
\end{Exa}

A frequently considered subclass of manifolds with corners is that of Lie manifolds, where the geometry at the boundary is encoded in a Lie algebra of vector fields.
We give several detailed examples of Lie manifolds at a later stage, see Section \ref{sec:lie}. The following definition of Lie manifolds is based on work by Ammann, Ionescu, Lauter and Nistor \cite{ALN}, \cite{AIN}. 

\begin{Def}
A Lie manifold $(M, \V)$ is a compact manifold with corners $M$ endowed with a so-called Lie structure of vector fields $\V \subset \Gamma(TM)$ which is assumed to be closed with regard to the Lie bracket of the set of smooth vector fields.

We call $\V$ \emph{Lie structure} if $\V$ is a finitely generated $C^{\infty}(M)$ module and all the vector fields in $\V$ are tangent to the boundary strata of $M$. 
\label{Def:Liemf}
\end{Def}

\begin{Rem}
Let $(M, \V)$ be a Lie manifold. Associated with $\V$ we find a vector bundle $\A \to M$ such that $\Gamma(\A) \cong \V$ by the Serre Swan theorem. We make the standing assumption that $\A_{|M_0}$ is the tangent bundle on the interior $TM_0$. This bundle $\A$ is canonically endowed with the structure of a Lie algebroid and under the previous assumption this Lie algebroid will be integrable (cf. \cite{ALN}). We will recall these notions of Lie groupoid and Lie algebroid in more detail in the following sections.
\label{Rem:Liemf}
\end{Rem}

We examine the class of embeddings compatible with this additional geometric structure.

\begin{Def}
An embedding of Lie manifolds $j \colon (Y, \B, \W) \hookrightarrow (M, \V, \A)$ is a transverse embedding of manifolds with corners such that inclusion $\B \hookrightarrow \A_{|Y}$ is a Lie subalgebroid, cf. \cite{Mc}. \\
We denote by $\EmbV$ the category of Lie manifolds with admissible embeddings as arrows. 
\label{Def:Liemf2}
\end{Def}

As a direct consequence of the previous discussion we obtain the following result.

\begin{Prop}
Let $j \colon Y \hookrightarrow M$ be an embedding of Lie manifolds. 

\begin{itemize}
\item[\emph{i)}] There is a tubular neighborhood
$Y \subset \U \subset M$ such that the embedding $\tilde{j} \colon Y \hookrightarrow \U$ is admissible. 

\item[\emph{ii)}] There is a direct sum decomposition of $\A_{|Y}$ into $\B \oplus \N$, where $\N := \frac{\A_{|Y}}{\B}$ is the $\A$-normal bundle and $\B \to Y$ is the Lie algebroid of the Lie-submanifold $(Y, \W)$, i.e. $\Gamma(\B) \cong \W$. 

\item[\emph{iii)}] The Lie structure $\W = \Gamma(\B)$ of vector fields of the Lie submanifold $Y$ takes the form
\[
\W = \{V_{|Y} : V \in \V, \ V_{|Y} \ \text{tangent to} \ Y\}. 
\]
\end{itemize}

\label{Prop:adm}
\end{Prop}

\begin{proof}
\emph{i)} This follows from the tubular neighborhood theorem given in \cite[Theorem 2.7]{AIN}.  

\emph{ii)} We have a short exact sequence of vector bundles
\[
\xymatrix{
\B \ar@{>->}[r] & \A_{|Y} \ar@{->>}[r]^q & \N. 
}
\]

This sequence splits and we can choose a splitting $\eta \colon \N \to \A_{|Y}$ by fixing a compatible Riemannian metric $g$ on $\A$. 

\emph{iii)} This follows by the condition that $\B$ is a Lie subalgebroid of $\A_{|Y}$. Denote by $\varrho$ the anchor of $\A$, then $\W = \{V \in \Gamma(\A_{|Y}) : \varrho(V) \in \Gamma(TY)\}$. 
\end{proof}

\begin{Rem}
The assumption of admissibility that we make throughout this document seems very restrictive at first. We mention however that in order to obtain a calculus for a relative transversal \emph{non-admissible} embedding it suffices to reduce ones attention to a tubular neighborhood $\tilde{M}$ of $j(Y)$ as illustrated in Figure \ref{fig:embex} where the embedding is admissible and then glue the resulting calculus with the corresponding pseudodifferential calculus on $M \setminus \tilde{M}$. This construction is carried out in \cite{B3}.
\label{Rem:adm}
\end{Rem}

\section{Non-Commutative completion}

\label{nccompl}

To motivate the definitions in this section we briefly return to the case where the base geometry has no singularity.
Let $M$ be a compact manifold (without corners) and $j \colon Y \hookrightarrow M$ an embedding.
Recall from section \ref{section:KN} that a microlocalization of $j$ is given by a calculus consisting of operators
of the form
\[
\begin{pmatrix} A_{\Psi} + A_{g} & A_c \\
A_b & A_{\partial} \end{pmatrix}.
\]

We restrict our attention to those operators of multiorder zero of the form
\[
\begin{pmatrix} A_g & A_c \\
A_b & A_{\partial} \end{pmatrix} : \begin{array}{c}
L^2(M) \\ 
\oplus\\
L^2(Y)
\end{array} \rightarrow \begin{array}{c}
L^2(M) \\ 
\oplus\\
L^2(Y)
\end{array}.
\]
We want to compare the class of all $A_g$ with the class of all $A_{\partial}$. The class of operators $\{A_g\}$ is an algebra (of singular Green operators), which we denote by $\U$. 
We also denote by $\Psi_{\partial}$ the class of operators $\{A_{\partial}\}$ which are pseudodifferential operators on $Y$.
The classes $\{A_b\}, \ \{A_c\}$ form bimodules $\mathcal{E}$ and $\mathcal{E}^{\ast}$, respectively, which implement a so-called \emph{bimodule correspondence} between $\U$ and $\Psi_{\partial}$. 
We can complete the operator classes $\U, \mathcal{E}, \mathcal{E}^{\ast}$ and $\Psi_{\partial}$ with regard to the relative $L^2$-based operator norms.
Then we obtain $C^{\ast}$-algebras $\overline{\U}$ and $\overline{\Psi}_{\partial}$.
It can be checked that the completed bimodules $\overline{\mathcal{E}}$ and $\overline{\mathcal{E}}^{\ast}$ are $C^{\ast}$-bimodules which implement
a \emph{Morita equivalence} of $C^{\ast}$-algebras $\overline{\Psi}_{\partial} \dashrightarrow \overline{\U}$. 
A bimodule correspondence / Morita equivalence of $C^{\ast}$-algebras yields an arrow / isomorphism in a suitable bimodule category of $C^{\ast}$-algebras.
We define this category $C_b^{\ast}$ below (the notation subscript $b$ being suggestive for \emph{bimodule}). \\
We now return to the case where the base geometry of the manifold may possess singularities. Our aim in this section is to give an abstract description of a microlocalization of an embedding of possibly singular manifolds.
As a preparation we introduce appropriate categories of Lie algebroids, Lie groupoids, symplectic groupoids and $C^{\ast}$-algebras. 
The main notion we need are the arrows in these categories which are implemented by a \emph{correspondence} and in particular the isomorphism classes of correspondences which we refer to as \emph{generalized morphisms}. 

\subsection{The category $\LG_b$}
\label{sec:groupoid}

We recall the definition of the category of Lie groupoids where the correct notion of morphism is a so-called generalized morphism. In fact the arrows in the category will consist of \emph{isomorphism classes} of bibundle correspondences, cf. \cite{L}. The Lie groupoids here are always assumed to be internal to the category of smooth manifolds, possibly with corners. We assume some familarity on the part of the reader with the theory of Lie groupoids and otherwise refer to \cite{Mc} for a more exhaustive presentation of the theory. 
We denote by $\G \rightrightarrows \Gop$ a Lie groupoid with space of objects which are also called the \emph{units} $\Gop$ and space of morphisms $\G$. Additionally, we fix the structural maps $r,s$ for the range and source maps, respectively, $u$ for the unit inclusion and $i$ the inversion.
The composable arrows are denoted by $\Gpull := \{(\gamma, \eta) \in \G \times \G : r(\gamma) = s(\eta)\}$. 
We denote by $m \colon \Gpull \to \G$ the multiplication in the Lie groupoid. As is costumary we denote by $\G_x = s^{-1}(x)$ the $s$-fibers, by $\G^x = r^{-1}(x)$ the $r$-fibers and by $\G_x^x$ the isotropy
group in $x \in \Gop$. The Lie algebroid $\A(\G) \to \Gop$ is the restriction of $\ker dr$ to the unit space $\Gop$ and equivalently defined as the normal bundle $N^{\G} \Gop$
to the unit inclusion $u \colon \Gop \hookrightarrow \G$. 
We give a brief summary of the categories of Lie groupoids and $C^{\ast}$-algebras under consideration in this work. This discussion is for the most part based on \cite{L}. See also the references contained therein. 
We also define so-called embeddings of Lie groupoids and what is meant by a non-commutative completion of an embedding. 

\begin{Def}
Let $\G \rightrightarrows \Gop$ be a Lie groupoid. A $C^{\infty}$-manifold $Z$ is called a \emph{right $\G$-space} if it carries a right $\G$-action in the following sense: there is a smooth map
$q \colon Z \to \Gop$ which is called \emph{charge map} and a smooth map $\alpha \colon Z \ast_r \G = \{(z, \gamma) : q(z) = r(\gamma)\} \to Z, (z, \gamma) \mapsto \alpha(z, \gamma) = z \cdot \gamma \in Z$
such that the following conditions hold

\hspace{0.5cm}\emph{i)} $q(z \cdot \gamma) = s(\gamma), \ (z, \gamma) \in Z \ast_r \G$. 

\hspace{0.5cm}\emph{ii)} $z \cdot (\gamma \cdot \eta) = (z \cdot \gamma) \cdot \eta$ for $(z, \gamma) \in Z \ast_r \G, \ (\gamma, \eta) \in \Gpull$, 

\hspace{0.5cm}\emph{iii)} $z \cdot \id_{q(z)} = z$,

A right $\G$-space $(Z, \alpha, q)$ is called \emph{$\G$-fibered} if $q$ is a surjective submersion. 

The right action is \emph{free} if $\forall_{z \in Z} \ z \cdot \gamma = z \Rightarrow \gamma = \id_{q(z)}$
and \emph{proper} if the map $Z \ast_r \G \to Z \times Z, \ (z, \gamma) \mapsto (z, z \cdot \gamma)$ is proper. 

A free and proper action is called \emph{principal}. 

A \emph{left $\G$-space} is a right $\G^{op}$-space where $\G^{op}$ denotes the opposite category of $\G$. 

\label{Def:act}
\end{Def}

\begin{Def}
Let $\G \rightrightarrows \Gop$ and $\H \rightrightarrows \Hop$ be Lie groupoids. A \emph{bibundle correspondence} from $\H$ to $\G$ 
is a triple $(Z, p, q)$ where $(Z, \alpha, p)$ is a left $\H$-space and $(Z, \beta, q)$ is a right $\G$-space such that 

\hspace{0.5cm}\emph{i)} the right action is principal and $Z$ is $\G$-fibered, 

\hspace{0.5cm}\emph{ii)} the map $p$ induces a diffeomorphism $Z / \G \iso \Hop$, 

\hspace{0.5cm}\emph{iii)} the actions of $\H  \ \rotatebox[origin=c]{-90}{$\circlearrowleft$}\ Z$ and $Z \ \rotatebox[origin=c]{90}{$\circlearrowright$}\ \G$ commute. 

An \emph{equivalence bibundle correspondence} of $\H$ and $\G$ (also called \emph{Morita equivalence} and denoted by $\H \simM \G$) is a triple $(Z, p, q)$ 
which is a bibundle correspondence from $\H$ to $\G$ such that

\hspace{0.5cm}\emph{iv)} the left action is principal and $Z$ is $\H$-fibered,

\hspace{0.5cm}\emph{v)} $q$ induces a diffeomorphism $\H / Z \iso \Gop$. 

\label{Def:morph}
\end{Def}

We now discuss how to compose morphisms in our category, hence we recall the definition of the \emph{generalized tensor product}. 
Let us fix two bibundle correspondence

\[
\begin{tikzcd}[every label/.append style={swap}]
\H \ar[d, shift left] \ar[d] \ar[symbol=\circlearrowleft]{r} & \ar{dl}{p_1} Z_1 \ar{dr}{q_1} & \ar[symbol=\circlearrowright]{l} \G \ar[d, shift left] \ar[d] \\
\Hop & & \Gop
\end{tikzcd}
\]

and
\[
\begin{tikzcd}[every label/.append style={swap}]
\G \ar[d, shift left] \ar[d] \ar[symbol=\circlearrowleft]{r} & \ar{dl}{p_2} Z_2 \ar{dr}{q_2} & \ar[symbol=\circlearrowright]{l} \K \ar[d, shift left] \ar[d] \\
\Gop & & \Kop
\end{tikzcd}
\]

Endow the fiber-product $Z_1 \times_{\Gop} Z_2$ with the right $\H$-action $\gamma \colon (z_1, z_2) \mapsto (z_1 \gamma, \gamma^{-1} z_2)$ and set 
$Z_1 \circledast Z_2 = (Z_1 \times_{\G} Z_2) / \G$, the orbit space of the action. 
We obtain a generalized morphism
\[
\begin{tikzcd}[every label/.append style={swap}]
\H \ar[d, shift left] \ar[d] \ar[symbol=\circlearrowleft]{r} & \ar{dl}{\tilde{p}} Z_1 \circledast Z_2 \ar{dr}{\tilde{q}} & \ar[symbol=\circlearrowright]{l} \K \ar[d, shift left] \ar[d] \\
\Hop & & \Kop
\end{tikzcd}
\]

with left-action $\eta [z_1, z_2]_{\G} = [\eta z_1, z_2]_{\G}$ and right-action $[z_1, z_2]_{\G} \kappa = [z_1, z_2 \kappa]_{\G}$. 
The charge maps are given by $\tilde{p}([z_1, z_2]) = p_1(z_1), \ \tilde{q}([z_1, z_2]_{\G}) = q_2(z_2)$. 

This composition of bibundle correspondences has a left and right unit, namely the trivial generalized morphism $\G  \ \rotatebox[origin=c]{-90}{$\circlearrowleft$}\ \G \ \rotatebox[origin=c]{90}{$\circlearrowright$}\ \G$. 

\begin{Def}
The category $\LG_b$ consists of Lie groupoids as objects with isomorphism classes of bibundle correspondences as arrows between objects and the generalized tensor product $\circledast$ as composition of arrows. 
We refer to isomorphism class of bibundle correspondences as \emph{generalized morphisms} of Lie groupoids. 
\label{Def:LGb}
\end{Def}

The isomorphisms in this category are precisely given by Morita equivalences of Lie groupoids. 

\begin{Prop}
A bibundle correspondence $\G  \ \rotatebox[origin=c]{-90}{$\circlearrowleft$}\ Z \ \rotatebox[origin=c]{90}{$\circlearrowright$}\ \H$ is a Morita equivalence if and only if the class $[Z]$ 
is invertible as an arrow in $\LG_b$. 
\label{Prop:Morita}
\end{Prop}

\begin{proof}
We refer to \cite{L} and \cite{MRW}. 
\end{proof}

\begin{Thm}
The category of Lie groupoid with strict morphisms is included in the category of Lie groupoids with generalized morphisms via a functor $j_b \colon \LG \to \LG_b$. 
\label{Thm:incl}
\end{Thm}

\begin{proof}
We need to study the map on morphisms.

Given a strict morphism
\[
\xymatrix{
\G \ar@<-.5ex>[d] \ar@<.5ex>[d] \ar[r]^{f} & \H \ar@<-.5ex>[d] \ar@<.5ex>[d] \\
\Gop \ar[r]^{\fop} & \Hop 
}
\]

we obtain a pullback diagram
\[
\xymatrix{
\ar[d]_{\pi_1} \Gop _{\fop}\times_{r} \H \ar[r]^-{\pi_2} & \H \ar[d]_{r} \\
\Gop \ar[r]^{\fop} & \Hop 
}
\]

This yields the generalized morphism
\[
\begin{tikzcd}[every label/.append style={swap}]
\G \ar[d, shift left] \ar[d] \ar[symbol=\circlearrowleft]{r} & \ar{dl}{\pi_1} \Gop _{\fop} \times_{r} \H \ar{dr}{r \circ \pi_2} & \ar[symbol=\circlearrowright]{l} \H \ar[d, shift left] \ar[d] \\
\Gop & & \Hop 
\end{tikzcd}
\]

The left action $\G \ \rotatebox[origin=c]{-90}{$\circlearrowleft$}\ \Gop _{\fop} \times_{r} \H$ is given by
\[
\gamma \cdot (s(\gamma), \eta) = (r(\gamma), f(\gamma) \cdot \eta), \ \gamma \in \G, \ \eta \in \H^{\fop(s(\gamma))}. 
\]

Finally, the right action $\Gop _{\fop} \times_{r} \H \ \rotatebox[origin=c]{90}{$\circlearrowright$} \ \H$ is given by
\[
(x, \eta) \cdot \tilde{\eta} = (x, \eta \cdot \tilde{\eta})
\]

and this action is principal. 
\end{proof}

\subsection{The category $\SG_b$}

\label{SGb}

Classically quantization is defined over symplectic manifolds. Therefore we are mainly interested in Lie groupoids with additional symplectic structure. We introduce the subcategory $\SG_b$ of $\LG_b$ which consists of symplectic Lie groupoids together with isomorphism classes of symplectic bibundle correspondences.
The objects in the category, the bibundles, composition of morphisms and the actions are defined as follows.

\begin{Def}
\emph{i)} A \emph{symplectic groupoid} $(\Gamma, \omega)$ is a Lie groupoid $\Gamma$ such that the space of morphisms $\Gamma = \Gamma^{(1)}$ is a symplectic manifold with symplectic $2$-form 
$\omega$ such that the graph of $\Gamma^{(2)}$ is a Lagrangian submanifold of $\Gamma \times \Gamma \times \Gamma^{-}$ with respect to $\omega \oplus \omega^{-}$ where $(\Gamma^{-}, \omega^{-}) = (\Gamma, -\omega)$.

\emph{ii)} An action $\alpha$ of a symplectic groupoid $(\Gamma, \omega)$ on a symplectic manifold $(S, \omega_S)$ is \emph{symplectic} if the graph of the action $\Gr(\alpha) \subset \Gamma \times S \times S^{-}$ is a Lagrangian submanifold with regard to $\omega_S \oplus \omega_{S^{-}}$ on $S \times S^{-}$ where $\omega_{S^{-}} = -\omega_S$. 
\label{Def:SGb}
\end{Def}

The above definition entails that the groupoid multiplication in a symplectic groupoid corresponds to a canonical relation. 
In the same way the symplectic action corresponds to a canonical relation as well. 

\emph{Morphisms in $\SG_b$:} Let $\Gamma$ and $\Sigma$ be symplectic groupoids. Then a \emph{symplectic bibundle} $S \in (\Gamma, \Sigma)$ in $\SG_b$ consists of two symplectic actions $\Gamma  \ \rotatebox[origin=c]{-90}{$\circlearrowleft$}\ S \ \rotatebox[origin=c]{90}{$\circlearrowright$}\ \Sigma$ on a given symplectic space $S$ where the right action is principal.

\emph{Composition:} Given two symplectic bibundles $\Gamma_1  \ \rotatebox[origin=c]{-90}{$\circlearrowleft$}\ S_1 \ \rotatebox[origin=c]{90}{$\circlearrowright$}\ \Sigma$ and $\Sigma  \ \rotatebox[origin=c]{-90}{$\circlearrowleft$}\ S_2 \ \rotatebox[origin=c]{90}{$\circlearrowright$}\ \Gamma_2$. Then there is a generalized tensor product $\Gamma_1  \ \rotatebox[origin=c]{-90}{$\circlearrowleft$}\ S_1 \circledast_{\Sigma} S_2 \ \rotatebox[origin=c]{90}{$\circlearrowright$}\ \Gamma_2$ which is a morphism in $\SG_b$ as can be checked, cf. \cite{L}. 

\emph{Morita equivalence ($\simM$):} A \emph{Morita equivalence} between two symplectic groupoids $\Sigma$ and $\Gamma$ is implemented by a symplectic bibundle $S \in (\Sigma, \Gamma)$ which is biprincipal. 

It is straightforward to define a suitable notion of isomorphism between two symplectic bibundles, i.e. a diffeomorphism
which is at the same time a symplectomorphism which is compatible with the actions. 
As stated previously the category $\SG_b$ therefore is defined to consist of symplectic groupoids as the objects and isomorphism classes of symplectic bibundles
as the arrows. The composition of the arrows is fascilitated by the generalized tensor product and the units are induced by the canonical symplectic bibundles $\Sigma  \ \rotatebox[origin=c]{-90}{$\circlearrowleft$}\ \Sigma \ \rotatebox[origin=c]{90}{$\circlearrowright$}\ \Sigma$,
where $\Sigma$ is a symplectic groupoid with the obvious left and right actions.
Note that this makes $\SG_b$ into a subcategory of $\LG_b$ consisting of Lie groupoids as the objects together with isomorphism classes of bibundles as
the arrows between objects. It also holds that two symplectic groupoids are isomorphic objects in the category $\SG_b$ if and only if they are Morita equivalent.
\begin{Prop}
A symplectic bibundle $S \in (\Gamma, \Sigma)$ is a Morita equivalence if and only if its isomorphism class $[S]$ is invertible
as an arrow in $\SG_b$. 
\label{Prop:Morita2}
\end{Prop}

\begin{proof}
We refer to \cite{L0} and \cite{L} for a proof. 
\end{proof}

\subsection{The category $C_b^{\ast}$}

\label{Cstarb}

The next goal is to consider $C^{\ast}$-algebras. In particular in this work we are interested in the functorial relationship between a suitable category of $C^{\ast}$-algebras and the category of Lie groupoids. 
For a more detailed presentation of the necessary background to this theory we first refer to the famous work of Muhly-Renault-Williams, \cite{MRW} and the general functoriality as can be found in Landsman \cite{L}. 
We define by $C_b^{\ast}$ the category with objects the (separable) $C^{\ast}$-algebras and arrows between objects given by isomorphism classes of bimodule correspondences.

\begin{Def}
Let $\E$ be a Banach space such that $\E$ is endowed with a right Hilbert $B$-module structure and a non-degenerate $\ast$-homomorphism $\pi \colon A \to \L_B(\E)$.
Then $\E$ is called a \emph{bimodule correspondence} and we write $\E \in (A, B)$. 
\label{Def:bimodule}
\end{Def}

Let $\E_1$ be an $(A, B)$-bimodule correspondence and let $\E_2$ be a $(B, C)$-bimodule correspondence. 
Then $\E_1 \otimes \E_2$ has a canonical $(A, C)$-bimodule structure, with the inner product $\scal{\cdot}{\cdot} \colon \E_1 \otimes \E_2 \to A$ given by
\[
\scal{\xi_1 \otimes \xi_2}{\eta_1 \otimes \eta_2}_A := \scal{\xi_1}{ \eta \scal{\eta_1}{\xi_2}_B}_A.
\]

Define the equivalence relation $\sim$ on $\E_1 \otimes \E_2$ via 
\[
\xi b \otimes \eta \sim \xi \otimes b \eta, \ \xi \in \E_1, \ \eta \in \E_2, \ b \in B.
\]

We complete the quotient by this equivalence relation with regard to the induced $A$-valued norm 
\begin{align}
\E_1 \hat{\otimes}_B \E_2 &= \overline{\E_1 \otimes \E_2 / \sim}^{\| \cdot\|}. \label{Rieffel}
\end{align}

This generalized (Rieffel) tensor product yields a Hilbert $(A, C)$-bimodule. 
 
Let $A$ be a $C^{\ast}$-algebra and let $\E$ be a left $A$ module, $\F$ be a right Hilbert $A$ module.
Define the maps $\theta_{x,y} \colon \E \to \F$ for given $x \in \F, \ y \in \E$ by 
$\theta_{x,y}(z) = x \scal{y}{z}_A$. 
We define the class of \emph{generalized compact operators} $\K(\E, \F)$ to be the closure of the span over the $\theta_{x,y}$, i.e. $\K(\E, \F) := \overline{\mathrm{span}\{\theta_{x,y} : x \in \F, y \in \E\}}$.
Note that $\K(\E, \F)$ is contained in the space of linear adjointable maps $\E \to \F$, i.e. $\K(\E, \F) \subset \L(\E, \F)$.
If $\F = A$ for $A$ a $C^{\ast}$-algebra we write $\K_{A}(\E) = \K(\E, A)$.

\begin{Rem} 
If $A = \Cc$ is the complex numbers and $\E = \H$ a complex Hilbert space we obtain that $\K_{\Cc}(\H) = \K(\H)$ are the compact operators on the Hilbert space $\H$.
In general the elements $\K(\E, \F)$ are not compact operators, hence the name generalized compact operators. 
\label{Rem:gencpt}
\end{Rem}

\begin{Def}
Let $A$ and $B$ be $C^{\ast}$-algebras. Then $A$ is \emph{Morita equivalent} to $B$ (written $A \simM B$) if there is 
a bimodule correspondence $\E \in (A, B)$ with the following properties:

\hspace{0.5cm} \emph{(i)} the linear span of the range of $\scal{\cdot}{\cdot}_B \colon \E \times \E \to B$ is dense in $B$.

\hspace{0.5cm} \emph{(ii)} The $\ast$-homomorphism $\pi \colon A \to \L_B(\E)$ is an isomorphism $A \cong \K_B(\E)$.  
\label{Def:MoritaCstar}
\end{Def}

The category $C_b^{\ast}$ consists of isomorphism classes of bimodule correspondences with composition given by the generalized (Rieffel) tensor product. 
We refer to an isomorphism class $[\E]$ of bimodule correspondence $\E \in (A, B)$ as a \emph{generalized morphism} of $C^{\ast}$-algebras, written $A \dashrightarrow B$. 
The next proposition shows that Morita equivalence is the same as isomorphy for $C^{\ast}$-algebras in the category $C_b^{\ast}$.

\begin{Prop}
A bimodule correspondence of $C^{\ast}$-algebras $\E \in (A, B)$ is a Morita equivalence if and only if its isomorphism
class $[\E]$ is invertible as an arrow in $C_b^{\ast}$. 
\label{Prop:MoritaCstar}
\end{Prop}

\begin{proof}
See \cite{L}, Prop. 3.7. 
\end{proof}

\begin{Prop}
There is a canonical covariant functor of inclusion \ $\widehat{}_b \ \colon C^{\ast} \hookrightarrow C_b^{\ast}$. 
\label{Prop:inclCstar}
\end{Prop}

\begin{proof}
We only have to describe the inclusion on morphisms. Let $f \colon A \to B$ be a non-degenerate strict morphism. Then we have
the assignment the module structure $B \times B \to B, \ (b_1, b_2) \mapsto b_1 b_2$. The scalar product 
$\scal{}{}_B \colon B \times B \to B$ given by $(b_1, b_2) \mapsto b_1^{\ast} b_2$. 
Finally, the non-degenerate $\ast$-homomorphism $A \to \L_B(B)$ is defined by $a \mapsto (b \mapsto f(a) \cdot b)$.  
\end{proof}

\subsection{The category $\LA_b$}

We recall the definition of a category of \emph{integrable Poisson manifolds} which is the classical analogue of
the previously introduced category of $C^{\ast}$-algebras with isomorphism classes of correspondence bimodules.
When we consider a bibundle correspondence between Lie groupoids, we may ask what a suitable notion
of correspondence of the associated Lie algebroids should be.
Specifically, we would like to define a notion of correspondence for Lie algebroids which makes the association
$\G \mapsto \A(\G)$ of a Lie groupoid to its Lie algebroid functorial.
We rely on the work of Xu who defined the notion of Morita equivalence of Poisson manifolds \cite{Xu2}. 
As well as on \cite{L0}, \cite{L} and \cite{Mc} for the definition of the category of integrable Poisson manifolds which
are duals of Lie algebroids. First we record the following well-known category equivalence between linear Poisson structures on manifolds and Lie algebroids.
By a linear Poisson structure we mean that the Poisson bracket of two fiberwise linear functions is again linear. 

\begin{Thm}
Given a smooth vector bundle $E \to M$ which is endowed with a linear Poisson structure, then $E^{\ast} \cong \A$ for 
a Lie algebroid $\A \to M$.
Conversely, the dual of a Lie algebroid has a canonical linear Poisson structure. In other words we have a category equivalence:
\[
\{\text{cat. of linear Poisson structures on vector bundles}\} \cong \{\text{cat. of Lie algebroids}\}. 
\]
\label{Thm:catequ}
\end{Thm}

Let $P, Q$ be Poisson manifolds. A \emph{Weinstein dual pair} $Q \leftarrow S \rightarrow P$ consists of 
a symplectic manifold $S$ and Poisson maps $q \colon S \to Q, \ p \colon S \to P^{-}$ such that 
$\{q^{\ast} f, p^{\ast} g\} = 0, \ f \in C^{\infty}(Q), \ g \in C^{\infty}(P)$. If $Q \leftarrow S_i \rightarrow P, \ i = 1,2$ are two Weinstein dual pairs, then they are defined to be \emph{isomorphic}
if there is a symplectomorphism $\varphi \colon S_1 \to S_2$ such that $q_2 \varphi = q_1, \ p_2 \varphi = p_1$. 
A \emph{regular} dual pair is a dual pair as above for which $q$ is a surjective submersion and $p,q$ are both complete Poisson maps.
The category $\LA_b$ consists of objects given by dual Lie algebroids $\A^{\ast}(\G)$ associated to arbitrary Lie groupoids
$\G \rightrightarrows \Gop$. The arrows are isomorphism classes of Weinstein dual pairs of the type $\A^{\ast}(\H) \leftarrow T^{\ast} Z \rightarrow \A^{\ast}(\G)$
induced by a correspondence bibundle $\H  \ \rotatebox[origin=c]{-90}{$\circlearrowleft$}\ Z \ \rotatebox[origin=c]{90}{$\circlearrowright$}\ \G$ of Lie groupoids (cf. Proposition \ref{Prop:mod}).

\subsection{Functoriality and the Muhly-Renault-Williams theorem}

The categories we have introduced thus far are related by functorial correspondences. In its original formulation the Muhly-Renault-Williams theorem
states that Morita equivalent Lie groupoids yield (strongly) Morita equivalent corresponding $C^{\ast}$-algebras. This statement can be generalized in the following Theorem,
based on the references \cite{L0}, \cite{L}, \cite{L2}, \cite{Xu}, \cite{Xu2} and \cite{MRW}.

\begin{Thm}[Functoriality]
\emph{i)} There is a functorial correspondence $\LG_b \to \LA_b$ from the category of $s$-connected Lie groupoids to the category of integrable Lie algebroids given by 
$\G \mapsto \A^{\ast}(\G)$ on objects and $[\H  \ \rotatebox[origin=c]{-90}{$\circlearrowleft$}\ Z \ \rotatebox[origin=c]{90}{$\circlearrowright$}\ \G] \mapsto [\A^{\ast}(\H)  \leftarrow T^{\ast} Z \rightarrow \A^{\ast}(\G)]$. 
In particular if $\H \simM \G$ in $\LG_b$ then $\A^{\ast}(\H) \simM \A^{\ast}(\G)$ in $\LA_b$, i.e. the functor preserves Morita equivalence. 

\emph{ii)} There is a functorial correspondence $\LG_b \ni \G \mapsto C^{\ast}(\G) \in C_b^{\ast}$ and $[\H  \ \rotatebox[origin=c]{-90}{$\circlearrowleft$}\ Z \ \rotatebox[origin=c]{90}{$\circlearrowright$}\ \G] \mapsto [C^{\ast}(\H)  \ \rotatebox[origin=c]{-90}{$\circlearrowleft$}\ \E_Z \ \rotatebox[origin=c]{90}{$\circlearrowright$}\ C^{\ast}(\G)]$. 
In particular if $\H \simM \G$ then $C^{\ast}(\H) \simM C^{\ast}(\G)$, i.e. the functor preserves Morita equivalence. 
\label{Thm:functoriality}
\end{Thm}

\begin{proof}
\emph{i)} See e.g. \cite{LR}. 

\emph{ii)} For a proof of the Muhly-Renault-Williams theorem for locally compact groupoids (Morita equivalent groupoids induce Morita equivalent $C^{\ast}$-algebras) we refer to \cite{MRW} and for the functoriality assertion we refer to \cite{L}, \cite{L2} as well as \cite{MO}.  
\end{proof}

\subsection{Completion functor}
\label{sec:noncommcomp}

In this section we finally establish the notion of a \emph{noncommutative completion} of a microlocal algebra on a base space relative to an embedded submanifold.
We first introduce operators that correspond to the restriction of a function on $M$ to $Y$ and its $L^2$-adjoint.

\begin{Def}
For a given embedding $Y \hookrightarrow M$ in $\EmbV$ we fix the following notion. We call two bounded operators 
$j^{\ast} \colon L_{\V}^2(M) \to L_{\W}^2(Y)$ and $j_{\ast} \colon L_{\W}^2(Y) \to L_{\V}^2(M)$ such that
$j_{\ast}$ is the $L^2$-adjoint of $j^{\ast}$ and $j^{\ast} j_{\ast} = \id_{L_{\W}^2(Y)}$ a \emph{generating (boundary, co-boundary) pair}.
\label{Def:genpair}
\end{Def}

We refer to the proof of the following Proposition for the construction of a generating pair. It is based on  \cite[Lem. 2]{AMMS}. 

\begin{Prop}
Let $j \colon Y \hookrightarrow M$ be an arrow in $\EmbV$. Then there is a generating pair $(j_{\ast}, j^{\ast})$ (canonically) associated to $j$.
\label{Prop:pair}
\end{Prop}

\begin{proof}
Let $j^{\ast}$ the restriction induced by pullback with $j$. By themselves these operators do not yet yield a generating pair. We modify them by use of a standard \textit{order reduction technique} as follows:\\
By \cite[Theorem 4.7]{AIN} this yields a continuous map
\[
j^{\ast} \colon H_{\V}^s(M) \to H_{\W}^{s - \frac{\nu}{2}}(Y), \qquad s>\frac{\nu}{2}
\]
where $H_{\V}^s(M)$ and $H_{\W}^{s'}(Y)$ are the $L^2$-based Sobolev spaces on $M$ and $Y$ respectively, defined by completion of $C_c^\infty$-functions.

We assume without loss of generality that the codimension of the embedding $j(Y)\subset M$ is constant $\nu = 1$ (otherwise we simply have to adapt the order reductions below accordingly).
Consider the bounded and invertible operator given by
\[
\begin{pmatrix} \Delta \\ j^{\ast} \end{pmatrix} \colon H_{\V}^2(M) \longrightarrow \begin{matrix} L_{\V}^2(M) \\ \oplus \\ H_{\W}^{\frac{3}{2}}(Y) \end{matrix}. 
\]

Here $\Delta$ is the Laplace operator associated to a fixed compatible metric $g$ on $M$. For the definition of the following order reductions in the Lie calculus we refer to \cite[Section 8]{B2}. 
Denote by $\lambda_{\partial}^{\frac{3}{2}}:H^s_\W(Y)\rightarrow H^{s-\frac{3}{2}}_\W(Y)$ an order reduction isomorphism of order $\frac{3}{2}$ on $Y$ and by $\lambda^{-2}$ an order reduction of order $-2$ on $M$. 
We obtain an isomorphism
\[
\begin{pmatrix} \Delta \lambda^{-2} \\ \lambda_{\partial}^{\frac{3}{2}} j^{\ast} \lambda^{-2} \end{pmatrix} \colon L_{\V}^2(M) \stackrel{\lambda^{-2}}{\longrightarrow} H^2_\V(M) \  \stackrel{\begin{pmatrix} \Delta \\ j^{\ast} \end{pmatrix}}{\longrightarrow}\ \begin{matrix} L_{\V}^2(M) \\ \oplus \\ H^{\frac{3}{2}}_{\W}(Y) \end{matrix} \ \stackrel{\id\times \lambda_\partial^{\frac{3}{2}}}{\longrightarrow} \ \begin{matrix} L_{\V}^2(M) \\ \oplus \\ L_{\W}^2(Y) \end{matrix}. 
\]

In particular the operator $B = \lambda_{\partial}^{\frac{3}{2}} j^{\ast} \lambda^{-2} \colon L_{\V}^2(M) \to L_{\W}^2(Y)$ has a right inverse which we denote by $C$. 
We check that the operator $C^{\ast} C$ is strictly positive: 
\[
\|v\|_{L_{\W}^2} = \|B C v\|_{L_{\W}^2} \leq \|B\|_{\L(L_{\V}^2, L_{\W}^2)} \|Cv\|_{L_{\V}^2}
\]

hence $\|Cv\| \geq c \|v\|$ for some $c > 0$. 
We set $S := (C^{\ast} C)^{-\frac{1}{2}}$ which is a $0$-order $\W$-pseudodifferential operator on $Y$. 
By an abuse of notation we now use the same symbol $j^{\ast}$ to denote the operator $S C^{\ast}$ and denote by $j_{\ast}$ the operator $CS$ which furnishes the desired generating pair. 
\end{proof}

\begin{Rem}
The proof shows in which sense $j^\ast$ is a microlocalization of the pullback by $j$. Taking the pullback results in a loss of Sobolev regularity, which can be formulated on a microlocal level by use of Sobolev wave front sets \cite{HHyp}. The order reductions, which are elliptic pseudo-differential operators, do not move singularities, and are used to restore this regularity. As such, $j^\ast$ is an operator with the same microlocally positioned singularities as the pullback by $j$, but of a different Sobolev strength.
\end{Rem}

\begin{Def}
Given an embedding $\j \in \EmbV$ we define the \emph{comparison algebra} $\Phi_{\V}(j)$ as the $C^{\ast}$-algebra generated by the set
\[
\left\{\begin{pmatrix} 0 & j_{\ast} \\ 0 & 0 \end{pmatrix}, \begin{pmatrix} 0 & 0 \\ 0 &  S \end{pmatrix} : S \in \Psi_{\W}(Y), \ (j^{\ast}, j_{\ast}) \ \text{generating pair}\right\}. 
\]
\label{Def:comparison}
\end{Def}

The association $\EmbV \ni j \mapsto \Phi_{\V}(j) \in C_b^{\ast}$ yields a covariant 
functor which is characterized, up to natural isomorphism, by certain universal properties. 
We specify next these minimal properties expected of a functor which yields a microlocalization of a given embedding in $\EmbV$. 

We fix the representation $\ast$-homomorphism $\varrho \colon \End\begin{pmatrix} C_c^{\infty}(\G) \\ \oplus \\ C_c^{\infty}(\H) \end{pmatrix} \to \End \begin{pmatrix} C^{\infty}(M) \\ \oplus \\ C^{\infty}(Y) \end{pmatrix}$ which is characterized by the property 
\begin{align}
& \left(\varrho(A) \circ \begin{pmatrix} r \\ r_{\partial} \end{pmatrix} \right) \begin{pmatrix} f \\ g \end{pmatrix} = A\begin{pmatrix} f \circ r \\ g \circ r_{\partial} \end{pmatrix}, \ f \in C^{\infty}(M), \ g \in C^{\infty}(Y) \label{repr}
\end{align}

where $r \colon \G \to M, \ r_{\partial} \colon \H \to Y$ denote the corresponding range maps. We denote by $\varrho_{|\H} \colon \End \ C_c^{\infty}(\H) \to \End \ C_c^{\infty}(Y)$ and by $\varrho_{|\G} \colon \End\ C_c^{\infty}(\G) \to \End \ C_c^{\infty}(M)$ the representation homomorphisms corresponding to $\H$ and $\G$ respectively.

\begin{Def}
A \emph{non-commutative completion} is a covariant functor from the category of admissible $C^{\infty}$-embeddings $\EmbV$ into the category of $C^{\ast}$-algebras and generalized morphisms, $C_b^{\ast}$.

Let $\Phi_{NC} \colon \EmbV \to C_b^{\ast}$ be a non-commutative completion. Then $\Phi_{NC}$ is called \emph{quantized} if it satisfies the following three properties:

\emph{i)} For each $j \colon (Y, \W) \hookrightarrow (M, \V)$ in $\EmbV$, the $\ast$-homomorphism $\varrho$, restricted to $\Phi_{NC}(j)$, induces a natural $C^{\ast}$-subalgebra of the bounded operators over $L_{\V}^2(M) \oplus L_{\W}^2(Y)$.
 
\emph{ii)} 
For each $j \colon (Y, \W) \hookrightarrow (M, \V)$ in $\EmbV$ there is a natural transformation $\lift_{qu} \colon \Emb_{\V} \to \LG_b$ associating to $Y$ a Lie groupoid $\H \rightrightarrows Y$ such that $\A_{\W} \cong \A(\H)$ and to $M$ a Lie groupoid $\G \rightrightarrows M$ such that
$\A_{\V} \cong \A(\G)$ as well as a natural transformation $\lift_{cl} \colon \Emb_{\V} \to \LA_b$ associating to $Y$ the Lie algebroid $\B \cong \A_{\W}$ and to $M$ the Lie algebroid $\A_{\V}$. In addition the diagram 
\begin{figure}[h]
\begin{tikzcd}
\EmbV\drar[bend right, "\lift_{qu}", swap]\rar[bend left, "\lift_{cl}"] 
 & \LA_b  \\
 & \LG_b \uar["\A^{\ast}"]
\end{tikzcd}
\end{figure}

commutes and consists of functors preserving Morita equivalence.

\emph{iii)} The representation $\varrho$ induces a surjective natural transformation $\varrho \colon \Phi_{NC} \to \Phi_{\V}$ such that for each $j \colon (Y, \W) \hookrightarrow (M, \V)$ and
corresponding groupoids $\G \rightrightarrows M, \ \H \rightrightarrows Y$ obtained via $\lift$ from \emph{ii)} the following diagram of generalized morphisms commutes:
\[
\xymatrix{
C(\H) \ar@{-->}[r]^{\Phi(j)} \ar@{-->}[d]_{\varrho_{|\H}} & C(\G) \ar@{-->}[d]_{\varrho_{|\G}} \\
C(Y, \W) \ar@{-->}[r]^{\Phi_{\V}(j)} & C(M, \V)
}
\]

where $C(\G), \ C(\H)$ denote the associated $C^{\ast}$-algebras.

\label{Def:nccompl}
\end{Def}

Let us explain the motivation for these three criteria and some implications:\\

The first property entails the continuity on the right class of Sobolev spaces, namely we require the operators in our algebra
to be $L^2$-bounded. 
The second condition is a factorization property over the category of Lie groupoids. This property is motivated
by the Muhly-Renault-Williams theorem, cf. \ref{SGb}, \cite{L}. 
The third property means that our algebra provides a solution to the Melrose quantization problem, cf. \cite{Mel2}. 
This problem is stated as follows: Given an arbitrary Lie structure $\V$ 
on a compact manifold with corners $M$, construct a pseudodifferential calculus naturally extending
the algebra of differential operators generated by $\V$, i.e. the universal enveloping algebra $\Diff_{\V}^{\ast}(M)$. 
Ammann, Lauter and Nistor in \cite{ALN} have constructed such a pseudodifferential calculus (the calculus $\Psi_{\V}(M)$ for a Lie manifold $(M, \V)$) and shown that there is a $\ast$-homomorphism from the pseudodifferential operators
$\Psi^{\ast}(\G)$ onto $\Psi_{\V}^{\ast}(M)$ which is surjective if $\G$ is a Lie groupoid integrating the Lie structure $\V$. 
Our condition three generalizes their approach by solving the following problem: Given an embedding of Lie manifolds (admissible, transverse) $j \colon (Y, \W) \hookrightarrow (M, \V)$ obtain a calculus of operators which yields a microlocalization of the embedding $j$ extending the algebra of differential
operators $\Diff_{\W}^{\ast}(Y)$ on the Lie manifold $(Y, \W)$. 
In the spirit of the Ammann, Lauter, Nistor representation theorem we request there to be a representation which in our case is
most easily expressed in terms of a generalized morphism of Lie groupoids and a calculus defined with regard to an embedding of Lie manifolds.
This condition is phrased succinctly in terms of axiom \emph{iii)} for a noncommutative completion. 

\begin{Rem}

The category $\EmbV$ of admissible embeddings of Lie manifolds is a category in two ways: The arrows consist of either the embeddings themselves $j \colon Y \hookrightarrow M$ or the corner correspondences $\varphi \colon \F(Y) \to \F(M)$,
which are bijective correspondences of all the faces of $Y$ and $M$. 
In the first case we obtain a semi-groupoid, i.e. a category internal to the category $\Cinfty$ of $C^{\infty}$-manifolds with corners.
This yields a bornological category or semi-groupoid with the objects given by isomorphy classes of Lie manifolds and the arrows 
given by transverse admissible embeddings of Lie manifolds. (Therefore this is a highly non-transitive category.) In the second case we obtain a groupoid structure with the objects given by 
$\mathcal{O}_M = \F(M)$, where $M$ runs over all isomorphy classes of Lie manifolds. Therefore in this second case we obtain a groupoid internal to the category of sets.
It is the second structure which is more interesting, since a corner correspondence is preserved by the functor $\lift \colon \EmbV \to \LG_b$ which associates
to a given corner correspondence the Morita equivalent Lie groupoids, integrating the corresponding Lie structures.
\label{Rem:nccompl}
\end{Rem}



% ...


\section{Examples: Lie structures}
\label{sec:lie}


Before constructing explicitly a non-commutative completion we first give examples of Lie structures and their integrating groupoids. 
There are several instances of manifolds with geometric singularities which can be modelled by non-compact manifolds endowed with a so-called Lie structure.
Let $(M, \A, \V)$ be a Lie manifold. By the results of Debord \cite{D} we can find an $s$-connected Lie groupoid $\G \rightrightarrows M$ such that $\A(\G) \cong \A$. Given a Lie groupoid $\G \rightrightarrows \Gop$ we denote by $\C_s \G$ the $s$-connected component of $\G$, i.e. the smallest groupoid with connected $s$-fibers which contains $\G$ as a subgroupoid. 
Alternatively, $\C_s \G$ is the union of the connected components of the $s$-fibers of $\G$. In general it is not clear for which Lie structure $\V$ it is possible to find an integrating groupoid with favorable properties, in particular Hausdorff and amenability properties.
It is therefore necessary to study particular cases of $\V$ and explicitely construct a suitable groupoid and orbit foliation associated to the Lie structure. This program is still underway, it has been solved for a few particular cases \cite{DLR}, \cite{LN}, \cite{M} and we also refer to \cite{DS2}
were a general framework is developed which includes constructions of minimal integrating Lie groupoids for many cases of Lie manifolds. We give below some examples for the case of embeddings of Lie manifolds which include the appropriate groupoid correspondences.

\begin{Exa}[The groupoid correspondence in the $b$-case]
Let $Y \hookrightarrow M$ be an admissible, transverse embedding of Lie manifolds with the Lie structure $\V_b$ of $b$-vector fields. 
Fix the boundary defining functions $\{p_i\}_{i \in I}$ of $M$ and $\{q_i\}_{i \in I}$ boundary defining functions of $Y$.
Note that by admissibility of the embedding $Y \hookrightarrow M$ we can assume that $I$ is a fixed (but finite) index set. 
Set $N = |I|$ and define
\begin{align}
\Gamma_b(M) &= \{(x, y, \lambda_1, \cdots, \lambda_N) \in M \times M \times (\Rr_{+})^N : p_i(x) = \lambda_i p_i(y), \ \forall i \in I\}. \label{bgrpd}
\end{align}
The structural maps are $s(x, y, \lambda) = y, \ r(x, y, \lambda) = x$, multiplication is given by $(x, y, \lambda) \circ (y, z, \mu) = (x, z, \lambda \cdot \mu)$ and inverse $(x, y, \lambda)^{-1} = (y, x, \lambda^{-1})$ with entrywise vector multiplication. 
Note that as a set $\Gamma_b(M)$ is 
\[
M_0 \times M_0 \cup \bigcup_{i \in I} F_i^2 \times (\Rr_{+})^{\codim(F_i)}
\]
The topology of the groupoid $\Gamma_b(M)$ can be described by specifying that interior sequences $M_0 \times M_0 \ni (x_n, y_n)$ converge towards a boundary element $(x, y, \lambda) \in \Gamma_b(M)$ if and only if
$x_n \to x, \ y_n \to y$ and $\frac{p_i(y_n)}{p_i(x_n)} \to \lambda_i, \ i \in I, \ n \to \infty$. 
The \emph{$b$-groupoid} is defined as the $s$-connected component of $\Gamma_b(M)$, i.e. $\G(M) := \C_s \Gamma_b(M)$. 

Fix the Lie algebroid $\A_M \to M$ such that $\Gamma(\A_M) \cong \V_b$. 
Then by \cite{M}, \cite{LN} we know that the $b$-groupoid $\G(M)$ integrates the Lie algebroid of the $b$-structure, i.e.
$\A(\G(M)) \cong \A_M$.

Define in the same way the corresponding $b$-groupoid $\G(Y) = \C_s \Gamma(Y)$ for the hypersurface $Y$ and set $Z := r^{-1}(Y) = \G_M^Y$, where $r$ is the range map of $\G(M)$. 
We describe in detail the generalized morphism $\G(Y) \dashrightarrow \G(M)$ implemented by $Z$. 
The right action of $\G(M)$ on $Z$ is given by right composition and the left action of $\G(Y)$ is given by left composition, i.e.
\[
\begin{tikzcd}[every label/.append style={swap}]
\G(Y) \ar[d, shift left] \ar[d] \ar[symbol=\circlearrowleft]{r} & \ar{dl}{p = r} Z \ar{dr}{q = s} & \ar[symbol=\circlearrowright]{l} \G(M) \ar[d, shift left] \ar[d] \\
Y & & M
\end{tikzcd}
\]

It is immediate to check that the actions commute and by definition the charge map $q$ induced by the source map of $\G(M)$ is a surjective submersion. 
The action of $\G(M)$ on its space of units is free and proper, hence the action of $\G(M)$ by right composition on $Z$ is free and proper.
Here properness means that the map $Z \ast \G(M) \to Z \times Z, \ (z, \gamma) \mapsto (z \gamma, z)$ is a homeomorphism onto
its image. Let $z \gamma = z$ then $q(z) = r(z) = r(\gamma)$ which implies composability and $s(\gamma) = q(z \gamma) = q(z)$.
Hence $\gamma = \id_{q(z)}$ which verifies that the action is free. The same can be proved for the left action under our assumption,
but we do not need this fact. Finally, we need to check that we have a diffeomorphism $Z / \G(M) \iso Y$ induced by the charge map $p$.
We first check this for the groupoids $\Gamma(M)$ and $\Gamma(Y)$, then we take the $s$-connected components which proves the assertion
for the groupoid $\G(M)$ and $\G(Y)$. We have to show that $p(z) = p(w)$ for $z, w \in Z$ if and only if there is a necessarily unique $\eta \in \G(M)$ such that $w = z \cdot \eta$.
Let $z = (x', y, (\lambda_i)_{i \in I}), \ w = (x', \tilde{y}, (\mu_i)_{i \in I})$ and set $\eta = \left(y, \tilde{y}, \left(\frac{\mu_i}{\lambda_i}\right)_{i \in I}\right)$. 
By definition of the topology of the groupoids fix the sequences $(x_n', y_n)$ such that $\frac{q_j(x_n')}{p_j(\tilde{y}_n)} \to \lambda_j, \ j \in I, \ n \to \infty$
and $(x_n', y_n)$ such that $\frac{q_j(x_n')}{p_j(y_n)} \to \mu_j, \ j \in I, \ n \to \infty$. Then $\eta \in \Gamma(M)$ since
\[
\frac{p_i(y_n)}{p_i(\tilde{y}_n)} = \frac{q_i(x_n')}{p_i(\tilde{y}_n)} \left(\frac{q_i(x_n')}{p_i(y_n)}\right)^{-1} \to \frac{\mu_i}{\lambda_i}, \ i \in I, \ n \to \infty.
\]
This concludes the proof of the isomorphism $Z / \G(M) \cong Y$. Hence we indeed obtain a generalized morphism $\G(Y) \dashrightarrow \G(M)$.
\label{Exa:bgrpd}
\end{Exa}

\begin{Exa}[Scattering Lie structure]
Let $M$ be a compact manifold with boundary endowed with the Lie structure of scattering vector fields, i.e. the module of vector fields $\V_{sc} = p \V_b$ where $p$ is the boundary defining function. In local coordinates where $x_1 = p$ the generators of these vector fields can be chosen as $\{x_1^2 \partial_{x_1}, x_1 \partial_{x_j}\}$, $j > 1$. 
An integrating groupoid in this case is written as a set 
\[
\G_{sc} = T_{\partial M} M\cup (M_0 \times M_0) \rightrightarrows M.
\]

Here the tangent bundle restricted to $\partial M$ is a viewed Lie groupoid which is glued to the pair groupoid on the interior.
If $M$ is a compact manifold with corners the scattering groupoid takes the form
\[
\G_{sc} = \left(\bigcup_{F \in \F_1(M)} T_{F} M \right) \cup (M_0 \times M_0) \rightrightarrows M.
\]

\label{Exa:scattering}
\end{Exa}

\begin{Exa}[Generalized cusp Lie structure]
On the compact manifold with corners $M$ we consider the Lie structure $\V_{c_n}$ of generalized cusps for $n \geq 2$, cf. \cite{LN}.
The local generators of vector fields of $\V_{c_n}$ are given by $\{x_1^n \partial_{x_1}, \partial_{x_2}, \cdots, \partial_{x_n}\}$.
Then an integrating groupoid is defined as
\[
\Gamma_{n}(M) = \{(x, y, \lambda) \in M \times M \times (\Rr_{+})^{I} : \lambda_i p_i(x)^n p_i(y)^n = p_i(x)^n - p_i(y)^n, \ \forall \ i \in I\}.
\]
With the structural maps defined in the same way as in the $b$-groupoid case. 
We set $\G_n(M) := \C_s \Gamma_n(M)$, the $s$-connected component of the groupoid $\Gamma_n(M)$.
Setting $Z := \G_n(M)_{M}^{Y}$ we define the right and left actions to obtain a generalized morphism $\G_{n}(Y) \dashrightarrow \G_n(M)$.
\label{Exa:cusp}
\end{Exa}

\begin{Exa}[Fibered cusp Lie structure]
Another interesting case is that of manifolds with iterated fibered boundary, in particular the following example is based on \cite{MM} and the definitions are from \cite{DLR}.
For the construction of the integrating Lie groupoid of a different type of fibered cusp Lie structure we refer to \cite{Guil}.  
We assume again that $M$ is a compact manifold with corners, but this time the boundary strata are assumed to be fibered in the following sense.
Let $\{F_i\}_{i \in I}$ the boundary hyperfaces of $M$ and denote by $\pi = (\pi_1, \cdots, \pi_N)$ an iterated boundary fibration
structure: There is a partial order defined on $\{F_i\}_{i \in I}$, $\pi \colon F_i \to B_i$ are fibrations where $B_i$ is the base,
a compact manifold with corners. 
Define the Lie structure
\[
\V_{\pi} := \{V \in \V_b : V_{|F_i} \ \text{tangent to fibers of} \ \pi_i \colon F_i \to B_i, \ V p_i \in p_i^2 C^{\infty}(M)\}
\]

where $\{p_i\}_{i \in I}$ denotes the boundary defining functions as usual.
Then $\V_{\pi}$ is a finitely generated $C^{\infty}(M)$-module and a Lie sub-algebra of $\Gamma^{\infty}(TM)$. 
The corresponding groupoid is amenable  \cite[Lemma 4.6]{DLR}; as a set it is defined as
\[
\G_{\pi}(M) := (M_0 \times M_0) \cup \left(\bigcup_{i = 1}^{N} (F_i \times_{\pi_i} T^{\pi} B_i \times_{\pi_i} F_i) \times \Rr\right), 
\]
where $T^{\pi} B_i$ denotes the algebroid of $B_i$.

One may then prove that $Z = \G_{\pi}(M)_{M}^{Y}$ implements a correspondence $\G_{\pi}(Y) \dashrightarrow \G_{\pi}(M)$. 
\label{Exa:phigrpd}
\end{Exa}

\section{Coste-Dazord-Weinstein groupoid}

In this section we recall the definition of the Coste-Dazord-Weinstein groupoid, in short CDW-groupoid, which is a groupoid structure on the cotangent bundle $T^{\ast} G$ of a given Lie groupoid $\G$. As such it is a natural space for microlocalization.  
For a detailed introduction to CDW-groupoids we refer to the work of Coste-Dazord-Weinstein, \cite{CDW} and
for a complementary definition to \cite{Mc}.


\begin{Def}
Let $\G \rightrightarrows \Gop$ be a Lie groupoid. Then $T^{\ast} \G \rightrightarrows \A^{\ast} \G$ is defined
as the groupoid with range and source maps given by
\[
\tilde{s}(\gamma, \xi) = (s(\gamma), dL_{\gamma}(\xi)), \ \tilde{r}(\gamma, \xi) = (r(\gamma), d R_{\gamma}(\xi)). 
\]

The inversion is given by
\[
(\gamma, \xi)^{-1} = (\gamma^{-1}, -(d i_{\gamma})^t \xi)
\]

and composition 
\[
(\gamma_1, \xi_1) \circ (\gamma_2, \xi_2) = (\gamma_1 \gamma_2, \xi) 
\]

with $\xi \in T_{\gamma_1 \gamma_2} \G$ such that
\[
\xi(dm(t_1, t_2)) = \xi_1(t_1) + \xi_2(t_2), \ (t_1, t_2) \in T \Gpull, \ \xi_1 \in T_{\gamma_1} \G, \ \xi_2 \in T_{\gamma_2} \G, \ (t_1, t_2) \in T_{(\gamma_1, \gamma_2)} \Gpull. 
\]

Whenever convenient we use from now on the notation $\Gamma := T^{\ast} \G$.

\label{Def:CDW}
\end{Def}

\begin{Rem}
We have a commuting diagram
\[
\xymatrix{
T^{\ast} \G \ar[d]_{\pi} \ar[r]^{\tilde{s}, \tilde{r}} & \A^{\ast}(\G) \ar[d]_{\overline{\pi}} \\
\G \ar[r]^{s, r} & \Gop 
}
\]

Additionally, the multiplication $m \colon \Gpull \to \G$ is compatible with the multiplication 
$\widetilde{m} \colon \Gamma^{(2)} \to \Gamma$, i.e. we have a commuting diagram
\[
\xymatrix{
\Gamma^{(2)} \ar[d]_{\pi \times \pi} \ar[r]^{\widetilde{m}} & \Gamma \ar[d]_{\pi} \\
\Gpull \ar[r]^{m} & \G
}
\]

In particular 
\[
N^{\ast} \Gpull \cong \ker \widetilde{m} \subset \Gamma^{(2)}
\]

cf. \cite{LMV}, Rem. 26. The groupoid $T^{\ast} \G = \Gamma$ has the structure of a vector bundle groupoid, in that
all structural maps are also vector bundle maps, see chapter 11 of \cite{Mc}. 
\label{Rem:CDW}
\end{Rem}

\begin{Thm}
The CDW-groupoid $T^{\ast} \G \rightrightarrows \A^{\ast}(\G)$ is a symplectic groupoid.
\label{Thm}
\end{Thm}

\begin{proof}
By the compatibility properties stated in Remark \ref{Rem:CDW} the CDW-groupoid is a Lie groupoid. Note that, as a cotangent bundle, the manifold $\Gamma = T^{\ast} \G$ carries a natural symplectic structure. This symplectic structure is compatible with the groupoid structure of $\Gamma$ which furnishes a symplectic groupoid, cf. \cite{CDW}.
Indeed, the graph $\mathrm{Gr}(\tilde{m})$ of $\tilde{m}$ is canonically isomorphic to the conormal space of the graph of $\mathrm{Gr}(m)$ of $m$ via 
$\mathrm{Gr}(\tilde{m}) \ni (\gamma, \xi, \gamma_1, \xi_1, \gamma_2, \xi_2) \mapsto (\gamma, -\xi, \gamma_1, \xi_1, \gamma_2, \xi_2) \in N^{\ast} \mathrm{Gr}(m)$. 
The symplectic structure on $T^{\ast} \G$ can be defined in such a way that $\mathrm{Gr}(\tilde{m})$ is a Lagrangian submanifold
of $\Gamma \times \Gamma \times \Gamma^{-}$, since $N^{\ast} \mathrm{Gr}(m)$ is a Lagrangian submanifold of $\Gamma \times \Gamma \times \Gamma$.  
\end{proof}

\begin{Exa}
\emph{(a)} Let $M$ denote a smooth manifold and $\G = M \times M \rightrightarrows M$ the pair groupoid.
Then $\A(M \times M) \cong TM$ and the CDW-groupoid $T^{\ast}(M \times M) \rightrightarrows T^{\ast} M$ has the following
structure 
\[
(T^{\ast} \G)^{(0)} = \A^{\ast}(\G) = \{(x, x, \xi, -\xi) \in T^{\ast} \G\} 
\]

with range and source maps given by
\[
\tilde{s}(x, y, \xi, \eta) = (y, y, -\eta, \eta), \ \tilde{r}(x, y, \xi, \eta) = (x, x, \xi, -\xi)
\]

composition
\[
(x, z, \xi, \zeta) \circ (z, y, -\zeta, \eta) = (x, y, \xi, \eta)
\]

and inverse given by
\[
(x, y, \xi, \eta)^{-1} = (y, x, -\eta, -\xi). 
\]

Note that the CDW-groupoid $T^{\ast}(M \times M)$ is canonically isomorphic to the pair groupoid $T^{\ast} M \times T^{\ast} M$
via the isomorphism $\pair \colon T^{\ast}(M \times M) \iso T^{\ast} M \times T^{\ast} M$ given by $(x, y, \xi, \eta) \mapsto (x, \xi, y, \eta)$. 

\emph{(b)} Let $M$ be a compact manifold with boundary. We denote by $\G := \G_b \rightrightarrows M$ the corresponding $b$-groupoid 
from Example \ref{Exa:bgrpd}. Recall that $\G$ as a set has the form 
\[
\G = M_0 \times M_0 \cup \partial M \times \partial M \times \Rr_{+}.
\]

Hence $\G$ is obtained by a glueing construction using a fixed boundary defining function $\rho \colon M \to \Rr_{+}$, 
where the pair groupoid on the interior is glued to the cylinder constructed out of the boundary $\partial M$ of $M$.
We have a pair groupoid structure on the interior and on the boundary we have the pair groupoid $\partial M \times \partial M$ and the
multiplicative group(oid) $(\Rr_{+}, \cdot)$. The corresponding CDW-groupoid $T^{\ast} \G \rightrightarrows \A^{\ast}(\G)$ can be constructed with the help of
the pair groupoid isomorphism defined in the previous example, applied to the interior and the boundary
separately. We denote by $\widetilde{\pair}$ the isomorphism of $T^{\ast} \G$ and the groupoid $\tilde{\Gamma} \rightrightarrows \A^{\ast}(\G)$
which is as a set given by 
\[
\tilde{\Gamma} = T^{\ast}(M_0) \times T^{\ast}(M_0) \cup T^{\ast}(\partial M) \times T^{\ast}(\partial M) \times T^{\ast}(\Rr_{+}).
\]
It is defined in the first set as the pair isomorphism applied to the interior cotangent bundle $T^{\ast}(M_0 \times M_0)$.
In the second set, it is the product of the pair isomorphism applied to $T^{\ast}(\partial M \times \partial M)$ with the 
identity group homomorphism $\id \colon (T^{\ast} \Rr_{+}, \cdot) \to (T^{\ast} \Rr_{+}, \cdot)$. \\
The CDW-groupoid $T^{\ast} \G \cong \tilde{\Gamma} \rightrightarrows \A^{\ast}(\G)$ has the following structure.
Note that the group $(\Rr_{+}, \cdot)$ has the associated Lie algebroid $A(\Rr_{+})$ which is the Lie algebra
$T_{1} \Rr_{+} \cong \Rr_{+}$.
One obtains the space of units $\A^{\ast}(\G) = \{(x, x, \xi, -\xi, 1, \lambda^{-1}) \in T^{\ast} \G\}$. The structural maps are given by 
\begin{align*}
& \tilde{s}(x_1, \xi_1, x_2, \xi_2, \lambda_1, \lambda_2) = (x_2, x_2, -\xi_2, \xi_2, 1, \lambda_2), \\
& \tilde{r}(x_1, \xi_1, x_2, \xi_2, \lambda_1, \lambda_2) = (x_1, x_1, \xi_1, - \xi_1, \lambda_1, 1),
\end{align*}
with composition
\[
(x_1, \xi_1, x_2, \xi_2, \lambda_1, \lambda_2) \circ (x_2, -\xi_2, x_3, \xi_3, \lambda_1^{-1}, \lambda_3) = (x_1, \xi_1, x_3, \xi_3, 1, \lambda_2 \lambda_3) 
\]
and inverse
\[
(x_1, \xi_1, x_2, \xi_2, \lambda_1, \lambda_2)^{-1} = (x_2, \xi_2, x_1, \xi_1, \lambda_1^{-1}, \lambda_2^{-1}).
\]

\label{Exa:CDW}
\end{Exa}

\subsection*{Induced symplectic actions}
We recall how any action of a Lie groupoid on a space induces canonically an action of the CDW-groupoid on the corresponding cotangent space. 
Fix a Lie groupoid $\G \rightrightarrows \Gop$ and a right principal $\G$-space $(Z, \alpha, q)$. 
We set $\Gamma := T^{\ast} \G \rightrightarrows \A^{\ast}(\G)$ for the Coste-Dazord-Weinstein groupoid
and set $\Gmod := T^{\ast} Z$. In the sequel we will endow $\Gmod$ with the structure of a $\Gamma$-space. 
To this end we define the charge map $\tilde{q} \colon \Gmod \to \A(\G)$ as $\tilde{q}(z, \xi) = (q(z), \overline{q}(\xi))$.
In order to define $\overline{q}$ we let $\xi \in T_z^{\ast} Z, \ x = q(z)$.
The extension of the linear form $\xi \circ d (L_z)_x \colon T_x \G_x \to \Rr$ by $0$ on the subspace $T_x \Gop$ of $T_x \G$
gives an element $\overline{q}(\xi)$ of $\A_x^{\ast}(\G) = (T_x \G / T_x \Gop)^{\ast}$. 
We also obtain the induced action $\tilde{\alpha} \colon \Gmod \ast_{\tilde{r}} \Gamma \to \Gmod$. 
The following Theorem concerns the functoriality of the association $\G \mapsto T^{\ast} \G$. A proof can be found in \cite{L2}. Nevertheless
we include some details to make the paper more self-contained.
\begin{Thm}
\emph{i)} The triple $(\Gmod, \tilde{\alpha}, \tilde{q})$ as specified above furnishes a principal right-$\Gamma$ space,
the charge map is given by $\tilde{q}(z, \xi) = (q(z), \overline{q}(\xi))$. 

\emph{ii)} We have a canonical isomorphism $(\Gmod)_{\tilde{\alpha}} \cong \N^{\ast} Z_{\alpha}$, i.e. $\Gmod$ 
has a symplectic structure as a $\Gamma$-space. 
\label{Prop:mod}
\end{Thm}

\begin{proof}
Consider the natural pairing $\scal{\cdot}{\cdot} \colon \A^{\ast}(\G) \times \A(\G) \to \Cc$ and rewrite the charge map $\tilde{q}$ 
as follows: $\scal{\tilde{q}(\theta)}{X} = \scal{\theta}{\xi_X^R}$ for $\theta \in T_{z}^{\ast} Z$, $X \in \A(\G)$. 
Here $\xi^R \colon \A(\G) \to TZ$ is defined as 
\[
\xi_X^R(z) := -z \frac{d \gamma(t)^{-1}}{dt|_{0}}, \ q(z) = x \in \Gop
\]

for $X$ given by $X = \frac{d \gamma(t)}{dt|_{0}} \in \pi^{-1}(x)$. Hence $\xi_X^R \in T_z Z$ is well-defined with
$\gamma(0) \in \Gop$ and $z \gamma(0) = z$, by definition of the $\G$-action on $Z$. 
The induced action $\widetilde{\alpha} \colon T^{\ast} Z \ast_{\tilde{r}} T^{\ast} \G \to T^{\ast} Z$ can be written as follows.
Let $\xi \in T_{\gamma}^{\ast} \G, \ \theta \in T_z^{\ast} Z$ such that $\tilde{r}(\xi) = \tilde{q}(\theta)$. Then by the definition of $\tilde{q}$
it follows that $r(\gamma) = q(z)$. Then $\widetilde{\alpha}(\theta, \xi) = \theta \cdot \xi \in T_{z \gamma}^{\ast} Z$ is given by
\[
\left \langle \theta \cdot \xi, \frac{d w(t)}{dt|_{0}} \right \rangle = \left \langle \theta, \frac{w(t) d g(t)^{-1}}{dt_{|0}}\right \rangle + \left \langle \xi, \frac{d g(t)}{dt_{|0}} \right \rangle 
\]

where $\frac{d w(t)}{dt_{|0}}$ is in $T_{z \gamma} Z$. Then $q(w(\cdot))$ is a curve in $\Gop$ such that $g(0) = \gamma, \ s(g(t)) = q(w(t))$. 
We can check that $\tilde{q}(\theta \cdot \xi) = \tilde{r}(\xi)$. Furthermore, the associativity property for actions holds.
To this end let $\xi \in T_{\gamma}^{\ast} \G, \ \widetilde{\xi} \in T_{\eta}^{\ast} \G$ and $\theta \in T_z^{\ast} Z$. 
We have that 
\begin{align*}
& \left \langle (\theta \cdot \xi) \cdot \widetilde{\xi}, \frac{d w(t)}{dt_{|0}}\right \rangle = \left \langle \theta \cdot \xi, \frac{w(t) d g(t)^{-1}}{dt_{|0}} \right \rangle + \left \langle \widetilde{\xi}, \frac{d g(t)}{dt_{|0}} \right \rangle \\ 
&= \left \langle \theta, \frac{w(t) g(t)^{-1} d h(t)^{-1}}{dt} \right \rangle + \left \langle \xi, \frac{dh(t)}{dt_{|0}} \right \rangle + \left \langle \widetilde{\xi}, \frac{d g(t)}{dt_{|0}} \right \rangle
\end{align*}

where the first line is obtained by definition of the multiplication in the CDW-groupoid and the second line by definition of the induced action. 
We have $h(0) = \eta, \ s(h(t)) = q(w(t) g(t)^{-1})$. Set $\tilde{g}(t) = g(t) h(t)$, then $\tilde{g}$ is the $s$-cover of $w(t)$,
i.e. $\tilde{g}(0) = \gamma \eta, \ s(\tilde{g}(t)) = q(w(t))$. In particular this yields the associativity, by the defining property
of the multiplication in the CDW-groupoid. 
Finally, we check that all structural maps of the CDW-action are vector bundle maps. In fact, fixing the projections $\pr \colon T^{\ast} Z \to Z$ and $\pr^{(0)} \colon \A^{\ast}(\G) \to \Gop$, 
we have the following commuting diagram
\[
\xymatrix{
N^{\ast}(Z \ast_r \G) \ar[d]_{\pr^2} \ar@{>->}[r] & T^{\ast} Z \ast_{\tilde{r}} T^{\ast} \G \ar[d]_{\pr^2} \ar@{->>}[r]^-{\tilde{\alpha}} & T^{\ast} Z \ar[d]_{\pr} \\
Z \ast_r \G \ar@{==}[r] & Z \ast_r \G \ar@{->>}[r]^{\alpha} & Z 
}
\]

Thus the action $\tilde{\alpha}$ extends the action $\alpha$ to a vector bundle map. If $Z = \G$ with the canonical self-action
we recover $\widetilde{m}$. Similarly, $\tilde{q}$ extends $q$ to a vector bundle map and the following diagram commutes
\[
\xymatrix{
\ker \tilde{q} \ar[d]_{\pr} \ar[r] & T^{\ast} Z \ar[d]_{\pr} \ar[r]^{\tilde{q}} & \A^{\ast}(\G) \ar[d]_{\prop} \\
Z \ar@{>->>}[r] & Z \ar@{->>}[r]^{q} & \Gop.
}
\]

Altogether, the induced action $\tilde{\alpha}$ makes the following diagram commute
\[
\xymatrix{
N^{\ast}(Z \ast_r \G) \ar@{>->}[r] & T_{|Z \ast_r \G}^{\ast}(Z \times \G) \ar@{->>}[r] & T^{\ast}(Z \ast_r \G) \\
\ker \tilde{\alpha} \ar[d]_{(\alpha, 0)} \ar@{>->>}[u] \ar@{>->}[r] & T^{\ast} Z \ast_{\tilde{r}} T^{\ast} \G \ar[d]_{\tilde{\alpha}} \ar@{>->}[u] \ar@{->>}[r] & (\ker d \alpha)^{\perp} \ar[u] \ar[d]_{(\alpha, (d^t \alpha)^{-1})} \\
Z \times \{0\} \ar@{>->}[r] & T^{\ast} Z \ar@{>->>}[r] & T^{\ast} Z 
}
\]

\emph{ii)} The graph of $(\theta, \xi) \mapsto \theta \cdot \xi$ is coisotropic in $T^{\ast} \G \times T^{\ast} Z \times T^{\ast} Z^{-}$
as can be checked by a local computation. By a dimension count one can check that we obtain a Lagrangian submanifold. Hence the induced action is in particular symplectic. 
\end{proof}

The association of a CDW groupoid to a Lie groupoid given above can therefore be implemented as a covariant functor from the
category of Lie groupoids $\LG_b$ to the category of symplectic groupoids $\SG_b$ together with the generalized tensor product
as composition, as introduced in section \ref{SGb}. 
This functor preserves Morita equivalence, in the sense that Morita equivalent Lie groupoids are sent to Morita
equivalent symplectic groupoids. 

\section{Fourier integral operators}

\label{FIO}

As a last preparatory step we now introduce the class of Fourier integral operators on Lie groupoids introduced by Lescure, Vassout and Manchon (\cite{LMV},\cite{LV}). These will provide an adequate framework for our microlocalization of an embedding. This is in complete analogy to the fact that the kernels of the operators in \cite{NS} are regarded as Fourier integral operators. In particular we present analogous notions of what is called $\G$ FIO's and $\G$ FFIO's in \cite{LV} in the case of principal $\G$-actions. 
We note that due to the availability of Kohn-Nirenberg quantization, many of the usual challenges of the theory can be circumvented by our choices of global phase functions ``of Fourier type''. However, we present the theory here in greater generality for sake of completeness and comparability with the results available in the literature.\\

Besides the introduction of the convolvability conditions for Lagrangians, we also recall the definition of the twisted product of
symbols from \cite{LV} and extend it to groupoid actions. We formulate the definition within the framework of the Maslov quantization
and the corresponding quantization of symbols. In the next section we specialize to the case of the Kohn-Nirenberg quantization
of symbol spaces which are contained in the standard symbol spaces used in the present section.

\subsection{Generalized morphisms}

\label{genmorph}

Let $M$ be a compact manifold with corners and let $Y \hookrightarrow M$ be a codimension $\nu$ embedding of manifolds with corners in the category $\EmbV$.
Fix a Lie groupoid $\G \rightrightarrows M$ such that $Y$ is transversal also with regard to $\G$ and set $\H := \G_Y^Y$. There is a groupoid correspondence $\H \dashrightarrow \G$ which is implemented by $Z = \G_M^Y = r^{-1}(Y)$. 
The corresponding actions are denoted as follows.

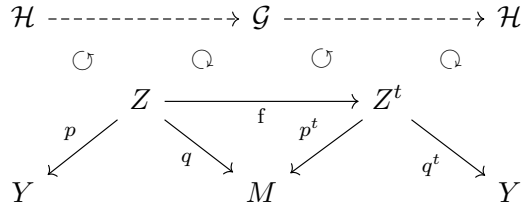
\begin{figure}[h]
\[
\begin{tikzcd}[every label/.append style={swap}]
\H \ar[dashed]{rr} \ar[symbol=\circlearrowleft]{rd} & & \G \ar[dashed]{rr} \ar[symbol=\circlearrowright]{ld} \ar[symbol=\circlearrowleft]{rd} & & \H \ar[symbol=\circlearrowright]{ld} \\[-2\jot]
& \ar{ld}{p} Z \ar{rr}{\flip} \ar{rd}{q} & & \ar{ld}{p^t} \Zop  \ar{rd}{q^t} & \\
Y & & M & & Y 
\end{tikzcd}
\]
\label{fig:main}
\caption{The Lie groupoid actions associated to an embedding.}
\end{figure}

The left and right actions are given respectively by left and right composition.
Note that the left hand side yields a morphism in $\LG_b$ and the right hand side a morphism in the opposite category
$\LG_b^{op}$.

We have induced actions of the corresponding CDW groupoids using Proposition \ref{Prop:mod}.

\begin{figure}[h]
\[
\begin{tikzcd}[every label/.append style={swap}]
T^{\ast} \H \ar[dashed]{rr} \ar[symbol=\circlearrowleft]{rd} & & T^{\ast} \G \ar[dashed]{rr} \ar[symbol=\circlearrowright]{ld} \ar[symbol=\circlearrowleft]{rd} & & T^{\ast} \H \ar[symbol=\circlearrowright]{ld} \\[-2\jot]
& \ar{ld}{\widetilde{p}} T^{\ast} Z \ar{rr}{d\flip} \ar{rd}{\widetilde{q}} & & \ar{ld}{\widetilde{p}^t} T^{\ast} \Zop  \ar{rd}{\widetilde{q}^t} & \\
\A^{\ast}(\H) & & \A^{\ast}(\G) & & \A^{\ast}(\H) 
\end{tikzcd}
\]
\label{fig:main2}
\caption{The induced CDW-groupoid actions associated to an embedding.}
\end{figure}

The left hand side is a morphism in $\SG_b$ and the right hand side a morphism in the opposite category $\SG_b^{op}$.

We denote by $\O = (\O_x)_{x \in M}$ the orbits of the Lie groupoid $\G \rightrightarrows M$ and by $\O^{\partial} = \{\O_y^{\partial}\}_{y \in Y}$ the orbits of $\H \rightrightarrows Y$. 
The fibers are defined by $\O_y^{\partial} := r(s^{-1}(y)) = s(r^{-1}(y)), \ y \in Y$ and $\O_x := r(s^{-1}(x)) = s(r^{-1}(x)), \ x \in M$.
We obtain a singular foliation $\F_M \to M$ by the orbits of the groupoid $\G$ and a singular foliation $\F_Y \to Y$ by the orbits of the
subgroupoid $\H$. 
We can lift these foliations by use of the range / source map to obtain a foliation $\F_{\G} \to \G$ and $\F_{\H} \to \H$, see also \cite{LV}.
Since $\H \dashrightarrow \G$ is a generalized morphism, we have the identification of the quotient $Z / \G \iso Y$ via the anchor map $p$.
Fix the quotient map $Q \colon Z \to Z / \G$. 
The orbits $\tilde{\O}_x := p(q^{-1}(x)) \subset Y, \ x \in M$ induce a foliation $\tilde{\F} \to Y$ which can be lifted
via the quotient map $Q$ to a foliation $\F_Z \to Z$ of $Z$. 
Given an admissible embedding $j \in \EmbV$ we have a natural lifting $\lift(j)$ which yields a generalized morphism $\H \dashrightarrow \G$ of the corresponding $s$-connected integrating Lie groupoids
$\H$ of $Y$ and $\G$ of $M$. 

\begin{Lem}
Given an admissible embedding $j \in \EmbV$, then $\lift(j)$ yields a Morita equivalence between $\H$ and $\G$. Furthermore, by functoriality we obtain the induced Morita equivalence between $T^{\ast} \H$ and $T^{\ast} \G$. 
\label{Lem:lift}
\end{Lem}

\begin{proof}
Since $j$ is admissible the corresponding faces of $Y$ and $M$ are in one-to-one correspondence. The faces of $Y$ and $M$ in turn correspond bijectively to the orbits of the groupoid $\H$ and $\G$ respectively.
Therefore the foliations $\F_{\Hop}$ and $\F_{\Gop}$ are in one-to-one correspondence with the faces of $Y$ and $M$ respectively. Also the lifted foliations $\F_{\G}$ and $\F_{\H}$ are in bijective correspondence with $\F_Z$.
Hence the generalized morphism $\lift(j)$ implemented by $Z$ is a Morita equivalence.  The combined structure is displayed in Figure \ref{fig:summary} where the last assertion follows since the functor $\CDW \circ \lift$ induces a Morita equivalence between $T^{\ast} \H$ and $T^{\ast} \G$. Here $\CDW$ associates to a given Lie groupoid $\G$ the corresponding CDW-groupoid $T^{\ast} \G$ on the level of objects of the categories
$\LG_b$ and $\SG_b$. The induced map on the level of arrows is described by Theorem \ref{Prop:mod}. In particular, we obtain induced symplectic foliations $\F_{T^{\ast} \G}$ and $\F_{T^{\ast} \H}$ that are 
in one-to-one correspondence with $\F_{T^{\ast} Z}$. 
\end{proof}

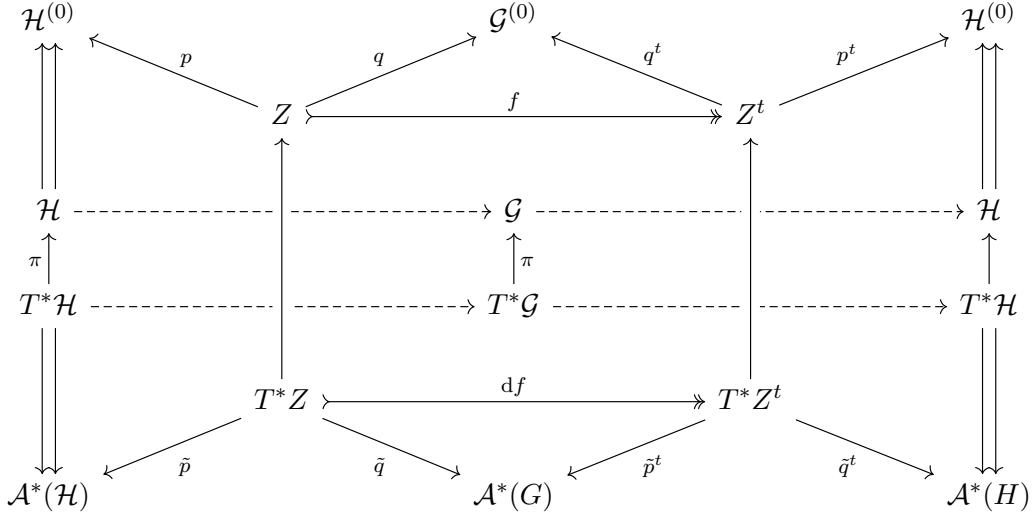
\begin{figure}[H]
      \centering
     \begin{tikzcd}[column sep=huge, text height=1.5ex, text depth=0.25ex]
 \H^{(0)}  & 																				& \G^{(0)} 	&	 	& \H^{(0)}		\\
  	& Z	\arrow[swap]{ul}{p} \arrow{ur}{q} \arrow[two heads, tail]{rr}{f}	& 		& Z^t \arrow[swap]{ul}{q^t} \arrow{ur}{p^t}	& 		\\
 \H \arrow[swap,shift right=2.5pt]{uu} \arrow[shift left=2.5pt]{uu} \arrow[dashed]{rr}	& 																				& \G	\arrow[dashed]{rr}	&		& \H	\arrow[swap,shift right=2.5pt]{uu} \arrow[shift left=2.5pt]{uu}	\\
 T^*\H	\arrow[swap,shift right=2.5pt]{dd} \arrow[shift left=2.5pt]{dd} \arrow[dashed]{rr} \arrow{u}{\pi}&																				& T^*\G	\arrow[dashed]{rr} \arrow[swap]{u}{\pi}	&		& T^*\H	\arrow[swap,shift right=2.5pt]{dd} \arrow[shift left=2.5pt]{dd} \arrow{u}	\\
 	& T^*Z	\arrow[crossing over]{uuu} \arrow{dl}{\tilde p} \arrow[swap]{dr}{\tilde q} \arrow[two heads, tail]{rr}{\dd f}																			&		& T^*Z^t	\arrow[crossing over]{uuu} \arrow{dl}{\tilde p^t} \arrow[swap]{dr}{\tilde q^t}	&		\\
 \A^*(\H)	&																				& \A^*(G)	&		& \A^*(H)
     \end{tikzcd}
\caption{Summarizing diagram.}
\label{fig:summary}
\end{figure}

\subsection{Convolution of Lagrangians}

Fourier integral operators on groupoids are equivariant families of operators whose kernels are associated to Lagrangian submanifolds. Their composability is ensured by geometric assumptions on the involved Lagrangians, see \cite{LV}. 

In order to obtain an algebra we therefore have to assure the composability of the involved Lagrangians as in the classical theory outlined in Section \ref{section:KN}. 
We start from the setup in subsection \ref{genmorph}, i.e. the bibundle correspondences of the Lie- and CDW-groupoids. 
We fix Lagrangian submanifolds $\Lambda_1 \subset T^{\ast} Z \setminus 0$ and $\Lambda_2 \subset T^{\ast} \G \setminus 0$. Note that in our case the Lagrangians are embedded which avoids the technicality of immersed (local) Lagrangians. Based on \cite{LV} a sufficient condition for convolution of $\Lambda_1$ and $\Lambda_2$ is that $\tilde{\Lambda} := \Lambda_1 \times \Lambda_2 \subset \Gmod \times \Gamma$ is a local Lagrangian submanifold which intersects $\Gmod \ast_{\tilde{r}} \Gamma$ cleanly.
Then $\Lambda := \tilde{\alpha}(\tilde{\Lambda} \cap \Gmod \ast_{\tilde{r}} \Gamma)$ is a local Lagrangian submanifold. This gives rise
to the following definition.

\begin{Def}
The Lagrangians $\Lambda_1, \Lambda_2$ are \emph{cleanly convolvable} if $\Lambda_1 \times \Lambda_2$ cleanly intersects $\Gmod \ast_{\tilde{r}} \Gamma$.
The convolution product in that case is defined by 
\[
\Lambda_1 \circ \Lambda_2 = \tilde{\alpha}(\Lambda_1 \times \Lambda_2 \cap \Gmod \ast_{\tilde{r}} \Gamma). 
\]
\label{Def:conv}
\end{Def}

The other condition which we impose is the following \emph{no-zero condition} and called \emph{admissibility} in \cite{LV}. 

\begin{Def}
\emph{i)} Let $\Lambda \subset T^{\ast} \G \setminus 0$ be a Lagrangian. Then $\Lambda$ is called \emph{admissible} if 
$\Lambda \cap \ker \tilde{s} = \Lambda \cap \ker \tilde{r} = \emptyset$. 

\emph{ii)} Let $\Lambda \subset T^{\ast} Z \setminus 0$ be a Lagrangian. Then $\Lambda$ is called \emph{admissible} if 
$\Lambda \cap \ker \tilde{q} = \emptyset$. 
\label{Def:adm}
\end{Def}

\begin{Exa}
We verify these conditions for the Lagrangians from Section \ref{section:KN}. 
Let $M$ be a compact manifold (without corners) and let $j \colon Y \hookrightarrow M$ be an embedding. 
Then, using the identifications from Example \ref{Exa:CDW} (a), we note that the natural action of the CDW groupoid $T^{\ast} (M \times M)$ on $T^{\ast} (Y \times M)$ yields that
$\Lambda_b$ and $\Lambda_g$ are cleanly convolvable. i.e. $\Lambda_b \times \Lambda_g \cap T^{\ast}(Y \times M) \ast_{\tilde{r}} T^{\ast}(M \times M)$ is clean. 
Also $\Lambda_b \cap \ker \tilde{q} = \emptyset$ and hence $\Lambda_b$ is admissible. Similarly, we check that $\Lambda_g \subset T^{\ast}(M \times M) \setminus 0$ is admissible. 
\label{Exa:adm}
\end{Exa}

\begin{Lem}
\emph{i)} $\Lambda_1$ is admissible if and only if $\alpha_x^{\ast} \Lambda_1 \subset (T^{\ast} Z_x \setminus 0) \times (T^{\ast} \G^x \setminus 0)$ is a Lagrangian
submanifold for each $x \in \Gop$.

\emph{ii)} $\Lambda_2$ is admissible if and only if $m_x^{\ast} \Lambda_2 \subset (T^{\ast} \G_x \setminus 0) \times (T^{\ast} \G^x \setminus 0)$ is 
a Lagrangian submanifold for each $x \in \Gop$.

\label{Lem:adm}
\end{Lem}

\begin{proof}
By definition 
\[
\alpha_x^{\ast} \Lambda_1 = \{(z, \xi, \gamma, \eta) \in T^{\ast} Z_x \times T^{\ast} \G^x : \exists_{(\tilde{z}, \tilde{\xi})} \ z \gamma = \tilde{z}, \ (d \alpha_x)_{z, \gamma}^t(\tilde{\xi}) = (\xi, \eta)\}.
\]

Where $d \alpha \colon T(Z \ast \G) \to TZ$ and since $\alpha \colon Z \ast \G \to Z$ is a submersion the assertion follows 
as in the proof of Proposition 13 of \cite{LV}. 
\end{proof}

We fix the projections $\tilde{\pr} \colon T^{\ast} (Z \ast \G) \to Z \ast \G$ and $\tilde{\pi} \colon Z \ast \G \to Y$.
In the second case we have the inclusion $Y \ni y \mapsto (\id_y, \id_y) \in \Hpull \subset Z \ast \G$ and hence $Y$ is an immersed submanifold
with regard to the identity inclusion. Similarly, we fix the projections $\pr \colon T^{\ast} \Gpull \to \Gpull$ and $\pi \colon \Gpull \to M$.
Where the second projection stems from the identity inclusion $M \ni x \mapsto (\id_x, \id_x) \in \Gpull$. 

\begin{Def}
\emph{i)} The Lagrangian $\Lambda_1$ is \emph{transverse} to $\tilde{\pi} \colon Z \ast \G \to Y$ if $\tilde{\pi} \circ \tilde{\pr}_{|\alpha^{\ast}(\Lambda_1)}$
is a submersion.

\emph{ii)} The Lagrangian $\Lambda_2$ is \emph{transverse} to $\pi \colon \Gpull \to M$ if $\pi \circ \pr_{|m^{\ast}(\Lambda_2)}$ is a submersion.
\label{Def:transverse}
\end{Def}

Define $c_{\gamma} \colon \G_x \times \G^x \to \G_y \times \G^y, \ \gamma \in \G_y^x$ by $c_{\gamma}(\gamma_1, \gamma_2) = (\gamma_1 \gamma^{-1}, \gamma^{-1} \gamma_2)$.
Let $(\Lambda_x)_{x \in \Gop}$ be a family of Lagrangians in $T^{\ast} \G_x \setminus 0$. 
Then this family is \emph{equivariant} if $(c_{\gamma})^{\ast} \Lambda_y = \Lambda_x$ we obtain the previous condition. Alternatively, this can be expressed in terms of the commuting diagram
\[
\xymatrix{
\G_x \times \G^x \ar[d]_{c_{\gamma}} \ar[r]^-{m_x} & \G \\
\G_y \times \G^y \ar[ur]^{m_y}
}
\]

Let now $(\Lambda_x)_{x \in \Gop}$ be a family of Lagrangians in $T^{\ast} Z_x \setminus 0$. 
Then define $\tilde{c}_{\gamma} \colon Z_x \times \G^x \to Z_y \times \G^y$ by $\tilde{c}_{\gamma}(z, \eta) = (z \gamma^{-1}, \gamma^{-1} \eta)$. 
In analogy we obtain the equivariance condition for the principal action of $\G$ on $Z$ by $(\tilde{c}_{\gamma})^{\ast} \Lambda_y = \Lambda_x$.
We can express the equivariance in terms of the diagram
\[
\xymatrix{
Z_x \times \G^x \ar[d]_{\tilde{c}_{\gamma}} \ar[r]^-{\alpha_x} & Z \\
Z_y \times \G^y \ar[ur]^{\alpha_y} 
}
\]

\begin{Def}
\emph{(1)} $\Lambda_1$ is a \emph{family $\G$-relation} if $\Lambda_1$ is transverse to $Z \ast \G \to Y$. 

\emph{(2)} $\Lambda_2$ is a \emph{family $\G$-relation} if $\Lambda_2$ is transverse to $\Gpull \to M$. 

\label{Def:famGrel}
\end{Def}

\begin{Thm}
\emph{(1)} $\Lambda_2$ is a family $\G$-relation if and only if there is an equivariant family of Lagrangians $(\Lambda_x)_{x \in \Gop}$
such that $m_x^{\ast} \Lambda = \Lambda_x, \ x \in \Gop$.

\emph{(2)} $\Lambda_1$ is a family $\G$-relation if and only if there is an equivariant family of Lagrangians $(\Lambda_x)_{x \in \Gop}$ 
such that $\alpha_x^{\ast} \Lambda = \Lambda_x, \ x \in \Gop$. 

\label{Thm:famGrel}
\end{Thm}

\begin{proof}
We refer to \cite[Theorem 14]{LV} for the proof of \emph{(1)}. The argument in the case \emph{(2)} for groupoid actions goes along the same lines.
We therefore omit the details.
\end{proof}

Given a local Lagrangian submanifold $\Lambda \subset T^{\ast} Z \setminus 0$ in the setup of subsection \ref{genmorph}, we define the \emph{transpose} by $\Lambda^t := d\flip(\Lambda)$. In this setting we have the following result, cf. \cite[Theorem 11]{LV} which guarantees the convolvability of $\Lambda_b$ and $\Lambda_c$ in the setup of Section \ref{section:KN}. 

\begin{Thm}
The Lagrangian $\Lambda \subset T^{\ast} Z \setminus 0$ is convolvable with its transpose $\Lambda^t$. 
\label{Thm:transpose}
\end{Thm}

\begin{proof}
Consider $\Lambda^t \subset T^{\ast} \Zop \setminus 0$, then $\tilde{p}^t \colon T^{\ast} \Zop \to \A^{\ast}(\G)$ and the diagram
\[
\xymatrix{
T^{\ast} \Zop \ar[d]_{\tilde{p}^t} \ar@{>->>}[r]^{d \flip} & \ar[dl]_{\tilde{q}} T^{\ast} Z \\
\A^{\ast}(\G) &
}
\]

commutes, i.e. $\tilde{p}^t = \tilde{q} \circ d\flip$. Since $\tilde{q}$ is
a local diffeomorphism, $\tilde{q} \times \tilde{p}_{|\Lambda \times \Lambda^t}^t \colon \Lambda \times \Lambda^t \to \A^{\ast}(\G) \times \A^{\ast}(\G)$
is a local diffeomorphism. It follows that
\begin{align*}
& (\Lambda \ast \Lambda^t) = (\tilde{q} \times \tilde{p}^t)_{|\Lambda \times \Lambda^t}^{-1}(\Delta_{\A^{\ast}(\G)}) = (\tilde{q} \times \tilde{p}^t)^{-1}(\Delta_{\A^{\ast}(\G)}) \cap (\Lambda \times \Lambda^t)
\end{align*}
is a submanifold of $T^{\ast} Z \times T^{\ast} \Zop$. We have that
\[
T(\Lambda \ast \Lambda^t) = T(T^{\ast} Z \ast T^{\ast} \Zop) \cap T(\Lambda \times \Lambda^t)
\]

and hence the intersection $(\Lambda \times \Lambda^t) \cap (T^{\ast} Z \ast T^{\ast} Z^t)$ is clean. 
Additionally, $\Lambda \times \Lambda^t \subset T^{\ast} Z \times T^{\ast} \Zop$ is also a local Lagrangian submanifold.
Altogether it follows that $\Lambda$ and $\Lambda^t$ are composable. 
\end{proof}

\subsection*{Lagrangian distributions}

We recall some properties of the class of Lagrangian distributions associated to a family $\G$-relation. As a particular case of Fourier integral operators we recover the definition of pseudodifferential operators on Lie groupoids and $\G$-equivariant pseudodifferential operators defined with regards to a Lie groupoid action. We note that as usual, Lagrangian distributions are actually defined as distributional half-densities, that is as distributional sections of the half-density bundle $\Omega^{\frac{1}{2}}$ over some manifold.
Given a smooth fibration $\pi \colon X \to Y$, denote by $\D_{\pi}'(X)$ the space of distributions transversal to the fibers of $\pi$, see p.5 of \cite{LMV}.
If $\pi_i \colon X \to Y_i$ for $i = 1,2$ are smooth fibrations, then $\D_{\pi_1, \pi_2}'(X)$ denotes the \emph{bi-transversal} distributions, i.e. transversal to the fibers of $\pi_1$ as well as to $\pi_2$.
The Lagrangian distributions $I^m(Z, \Lambda, \OmegaZh)$ for a given family $\G$-relation $\Lambda \subset T^{\ast} Z \setminus 0$ are subspaces of $\D_{q}'(Z, \Lambda, \OmegaZh)$.
In the special case where $Z = \G$ with the canonical action of $\G$ on itself, we recover the space $I^m(\G, \Lambda, \OmegaGh)$ which is a subspace of the 
bi-transversal distributions $\D_{r,s}'(\G, \Lambda, \OmegaGh)$. The space of transversal distributions, endowed with the convolution product induced by
the groupoid structure, is in one-to-one correspondence with the space of $\G$-operators. Also, the bi-transversal 
distributions endowed with the convolution product are in one-to-one correspondence with the \emph{adjointable} $\G$-operators, cf. \cite{LMV}. 
Note that the map $\flip$ in Fig. \ref{fig:main} induces the canonical isomorphism $\D_q'(Z, \Lambda; \OmegaZh) \cong \D_{p^t}'(\Zop, \Lambda; \OmegaZoph)$.
Therefore, we focus in what follows on the case of the action of $\G$ on $Z$, the other cases, i.e. $Z = \G$ or $\Zop$, follow
from appropriate identifications. 

\begin{Def}
Given a family $\G$-relation $\Lambda \subset T^{\ast} Z \setminus 0$ define $I_c^{m}(Z, \Lambda, \OmegaZh)$
as the class of Lagrangian distributions on $Z$. 
\label{Def:LMV4}
\end{Def}

Lagrangian distributions are locally given as (families of) oscillatory integrals. Taking the amplitude of the oscillatory integral, after a choice of phase function and some reductions, provides a local representation of the so-called \textit{principal symbol map} which is an isomorphism 
\[
I_c^m(Z,\Lambda,\Omega_Z^{\frac{1}{2}})\rightarrow S^{m+n/4}(\Lambda,I_\Lambda\otimes\Omega^{\frac{1}{2}}\otimes\Omega_Z^{-\frac{1}{2}}).
\]
The Maslov bundle $I_\Lambda$ therein reflects the (non-equivalent) choices of phase function.\\
The inverse of this map, associating an operator to a (symbol) section of the Maslov bundle is called \textit{Maslov quantization}.

\begin{Thm}
Let $\Lambda_1 \subset T^{\ast} Z \setminus 0$ and $\Lambda_2 \subset T^{\ast} \G \setminus 0$ be family $\G$-relations that
are convolvable with excess $e$. 
Then we have the following \emph{rules of composition}: Let $A_1 \in I_c^{m_1}(Z, \Lambda_1, \OmegaZh), \ A_2 \in I_c^{m_2}(\G, \Lambda_2, \OmegaGh)$ 
then
\[
A_1 \circ A_2 \in I_c^{m_1 + m_2 + \frac{e}{2}}(Z, \Lambda_1 \circ \Lambda_2, \OmegaZh).
\]
\label{Thm:LMV6}
\end{Thm}

On principal symbol level, this composition induces a symbol product which is covered in the next section.\\

We recall the definition of pseudodifferential operators defined over the category of principal $\G$-spaces.
Most of the material in this section is from \cite{NWX} for the calculus of pseudodifferential operators on Lie groupoids
as well as from \cite{P} for pseudodifferential operators defined on fibered spaces with a groupoid action. In the pseudodifferential setting, the symbol map is far simpler since the Maslov bundle of a conormal bundle (in particular the  diagonal) is canonically trivializable.

\begin{itemize}
\item Let $\Omega \subset \G$ be an open subset, $W \subset \Rr^n$ be open, set $V_s := s(\Omega)$ and fix a diffeomorphism $\psi_s \colon V_s\times W_s  \to \Omega$.
Then $\psi_s$ is called \emph{$s$-fiber preserving} if the diagram
\[
\xymatrix{
\Omega \ar[d]_{s} & \ar[l]_-{\psi_s} W \times s(\Omega) \ar[dl]_-{\pi_2} \\
s(\Omega)
}
\]

commutes. In other words for each $(w, x) \in W_s\times V_s$ we have $s(\psi_s(w, x)) = x$. 
Fix the notation $\Omega \sim W \times s(\Omega) \sim W \times V$ if there is such an $s$-fiber preserving diffeomorphism.
Denote by $\psi_s^{\ast} \colon C_c^{\infty}(W \times V) \to C_c^{\infty}(\Omega)$ the pullback over compactly supported
smooth functions. Then $P$ is called a \emph{smooth family} or \emph{$C^{\infty}$-family} if for each fiber preserving diffeomorphism $\psi_s$ as above
the operator $(\psi_s^{-1})^{\ast} \circ P \circ (\psi_s)^{\ast} \colon C_c^{\infty}(W \times V) \to C_c^{\infty}(W \times V)$
is a family of properly supported pseudodifferential operators, locally parametrized over $s(\Omega)$.

\item Let $R_{\gamma} \colon C_c^{\infty}(\G_{s(\gamma)}) \to C_c^{\infty}(\G_{r(\gamma)})$ denote the \emph{right-transform}
defined by $(R_{\gamma} f)(\eta) = f(\eta \gamma)$. 
If $P = (P_x)_{x \in \Gop}$ is a family of continuous linear operators $P_x \colon C_c^{\infty}(\G_x) \to C_c^{\infty}(\G_x)$,
then $P$ is called \emph{right-invariant} if $R_{\gamma} P_{s(\gamma)} = P_{r(\gamma)} R_{\gamma}$ holds for each $\gamma \in \G$. 
We call a right invariant operator $P$ a \emph{$\G$-operator}.

\item Support condition: Let $\mu \colon \G \ast_{s} \G \to \G, \ (\gamma, \eta) \mapsto \gamma \eta^{-1}$.
Where we set $\G \ast_s \G := \{(\gamma, \eta) \in \G \times \G : s(\gamma) = s(\eta)\}$.
Denote by $\supp(P) = \overline{\bigcup_{x \in \Gop} \supp(k_x)}$ the support of $P$, where $k_x$ denotes the 
Schwartz integral kernel of the operator $P_x$ in the family $P = (P_x)_{x \in \Gop}$. 
Then $\supp_{\mu}(P) := \mu(\supp(P))$ will be called \emph{reduced support}. 
If $\supp_{\mu}(P) \subset \G$ is a compact subset (with regard to the locally compact topology of $\G$), then we 
call $P$ \emph{uniformly supported}. 

\end{itemize}

\begin{Def}
Let $\G \rightrightarrows \Gop$ be a Lie groupoid.
We define by $\Psi_u^m(\G)$ the class of smooth families $P = (P_x)_{x \in \Gop}$ of pseudodifferential operators of order $m \in \Rr$ on the $s$-fibers $(\G_x)_{x \in \Gop}$ of $\G$ 
which are right invariant and uniformly supported.
\label{Def:Gpsdo}
\end{Def}

\begin{Rem}
\emph{i)} By \ref{Lem:subm} in a Lie groupoid $\G \rightrightarrows \Gop$ the source and range
map are tame submersions. Hence the fibers $\G_x$ are smooth manifolds without corners. The uniform support property implies that the operator is properly supported when considered as an operator
acting on $C_c^{\infty}(\G)$. We obtain by the right invariance property of the family $P = (P_x)_{x\in \Gop}$ a reduced kernel
$k_P$ which is a compactly supported distribution on $\G$ conormal to $\Gop$.
If $P$ is a $\G$-operator note that by right invariance $k_{r(\gamma)}(\gamma_1, \gamma_2) = k_{s(\gamma)}(\gamma_1 \gamma^{-1}, \gamma_2 \gamma^{-1})$
for $\gamma_1, \gamma_2 \in \G_{r(\gamma)}, \ \gamma \in \G$.
It follows that $k_{s(\gamma)}(\gamma, \eta) = k_{r(\gamma)}(\id_{r(\gamma)}, \eta \gamma^{-1}) =: k_P(\eta \gamma^{-1})$.
In particular the reduced kernel $k_P(\eta \gamma^{-1}) = k_{s(\gamma)}(\eta, \gamma)$ depends only on $\eta \gamma^{-1} \in \G$ for
each $(\eta, \gamma^{-1}) \in \Gpull$ as can be shown by the right-invariance of a given Haar system, cf. \cite{NWX}. 
Thus the $\G$-operator $P \colon C_c^{\infty}(\G) \to C_c^{\infty}(\G)$ is defined as
\[
P u(\gamma) = \int_{\G_{s(\gamma)}} k_P(\gamma \eta^{-1}) u(\eta) \,d\mu_{s(\gamma)}(\eta)
\]

for $u \in C_c^{\infty}(\G)$ and $(\mu_x)_{x \in \Gop}$ is a smooth right invariant system of Haar measures which is uniquely determined up to Morita equivalence.
\label{Rem:Gpsdo}
\end{Rem}

In the same vein we consider the right $\G$-space $Z$ and the left $\G$-space $Z^t$. 
For technical reasons we need to consider also families of pseudodifferential operators defined on the fibers of these spaces. 

\begin{itemize}
\item We denote by $\Psi^m(Z)$ families of operators $S = (S_x)_{x \in \Gop}$ with the following properties. 
For each $x \in \Gop$ we have $S_x \colon C_c^{\infty}(Z_x) \to C_c^{\infty}(Z_x)$ are
continuous linear operators contained in $\Psi_{prop}^{m}(Z_x)$. 

\item Additionally, for each $\Omega \subset Z$ open such that there is a diffeomorphism preserving the $q$-fibers $\Omega \sim q(\Omega) \times W$.
There is a function $a \colon q(\Omega) \to S^m(T^{\ast} W)$ into the H\"ormander symbols class which is smooth and
$S_{x|\Omega \cap Z_x} = a_x(y, D_y)$ on $\Omega \cap Z_x \cong W$. 

\item By abuse of notation we denote by $R_{\gamma} \colon C_c^{\infty}(Z_{s(\gamma)}) \to C_c^{\infty}(Z_{r(\gamma)})$ the
\emph{right-transform} acting on the fibers of $Z$, i.e. $(R_{\gamma} f)(z) = f(z \gamma), \ z \in Z$. 
Then $S \in \Psi^m(Z)$ is called \emph{right-invariant} (with regard to the $\G$-action) if $R_{\gamma^{-1}} S_{r(\gamma)} R_{\gamma} = S_{s(\gamma)}$
holds for each $\gamma \in \G$. 

\end{itemize}

\begin{Def}
Let $Z \ \rotatebox[origin=c]{90}{$\circlearrowright$}\ \G$ denote a principal right $\G$-action,
then we denote by $\PsiZ$ the subclass of $\Psi^{\ast}(Z)$ consisting of families which 
are right-invariant. The corresponding notion $\PsiZop$ for a principal left $\G$-action $\G  \ \rotatebox[origin=c]{-90}{$\circlearrowleft$}\ Z^t$  is defined in complete analogy.
\label{Def:Zpseudos}
\end{Def}

As shown in \cite{LV}, the $\G$-invariant pseudodifferential operators are a special case of Fourier integral operators for the Lagrangian $\Lambda = \A^{\ast}(\G)$.
It is not hard to see that the same is true for actions: $I^{L}(Z, (T^q Z)^{\ast}) \cong \PsiZm$, where
$L = m + \frac{1}{4} (\dim Z - 2 \dim(M))$. If $m = 0$ the operator $A \in I_c^L(Z, \Lambda; \OmegaZh)$ extends to an unbounded operator on the $C^{\ast}$-module $C_r^{\ast}(Z) := \overline{C_c^{\infty}(Z)}$.
The appropriate version of the Egorov theorem for groupoid actions reveals the module structure of the space of Lagrangian distributions in relation to the
pseudodifferential operators. 

\begin{Thm}[Egorov's theorem for actions] Let $\Lambda_1 \subset T^{\ast} Z \setminus 0, \ \Lambda_2 \subset T^{\ast} \G \setminus 0$ be convolvable,
closed (right-) $\G$-relations such that $\Lambda_1 \circ \Lambda_2 \subset (T^q Z)^{\ast}$ and $\Lambda_2 \circ \Lambda_1^t \subset (T^q \Zop)^{\ast}$
(where $\Lambda_1^t = \flip^{\ast} \Lambda_1$), then 
\[
I_c(\G, \Lambda_2, \OmegaGh) \ast \Psi(Z, \OmegaZh) \ast I(\G, \Lambda_2, \OmegaGh) \subset \Psi(Z, \OmegaZh). 
\]
\label{Thm:Egorov}
\end{Thm}

Note that a Haar system is unique only up to Morita equivalence of Lie groupoids. Haar systems are constructed via half densities, cf. \cite{LR}.
We implicitly fix a Haar system here and in the next section.  

\subsection{Twisted product}

We now recall the definition of the twisted product for standard symbol spaces from \cite{LV}. First we fix some notation. This twisted product is defined in such a way such that the symbol of a convolution of two distributions is the twisted product of their symbols.\\ 
For a given vector bundle $p \colon E \to \G$ or $p \colon E \to Z$ we set $\widehat{E} := (p_{|\Lambda})^{\ast}(E_{|\Lambda})$.
Set $\SigmaG := \OmegaG^{-\frac12} \otimes \Omega^{\frac12}$ and $\SigmaZ := \Omega_Z^{-\frac12} \otimes \Omega^{\frac12}$. 
We denote by $I_{\Lambda}$ the corresponding Maslov bundle. 

\begin{Prop}
The pairing 
\[
\ast \colon S^{m_1 + \frac{n}{4}}(\Lambda_1, \SigmaZh \otimes I_{\Lambda_1}) \times S^{m_2 + \frac{n}{4}}(\Lambda_2, \SigmaGh \otimes I_{\Lambda_2}) \to S^{m_1 + m_2 + \frac{e}{2} + \frac{n}{2}}(\Lambda, \SigmaZh \otimes I_{\Lambda})
\]

where $\Lambda = \Lambda_1 \circ \Lambda_2$ for $(z, \xi) \in \Lambda$
\[
a_1 \ast a_2(x, \xi) := \int_{\widetilde{\alpha}^{-1}(z, \xi) \cap \Lambda_1 \times \Lambda_2} a_1 \boxtimes a_2
\]

yields well-defined symbol referred to as the \emph{twisted product}.
\label{Prop:twisted} 
\end{Prop}

\begin{proof}
That this pairing is a well-defined symbol follows by the relation to the groupoid structure of $T^{\ast} \G$; the corresponding
identifications of Maslov bundles and densities:

Start with the natural map 
\begin{align}
& I_{\Lambda_1} \boxtimes I_{\Lambda_2} \to \widetilde{\alpha}^{\ast} I_{\Lambda} \otimes \Omega^{-\frac{1}{2}}(\ker d\alpha) \otimes \Omega(\ker d\alpha \cap T(\Lambda_1 \times \Lambda_2)) \label{t1}
\end{align}

We have the identifications (cf. \cite{LV})
\begin{align}
& (\SigmaZh \otimes I_{\Lambda_1}) \boxtimes (\SigmaGh \otimes I_{\Lambda_2}) \cong \Omega(\ker d \widetilde{\alpha}) \otimes (I_{\Lambda_1} \boxtimes I_{\Lambda_2}) \cong \widetilde{\alpha}^{\ast}(\widehat{\Sigma}_{Z}) \otimes (I_{\Lambda_1} \boxtimes I_{\Lambda_2}), \label{t2} \\
& \Omega(\ker d \widetilde{\alpha})^{\frac{1}{2}} \cong \widetilde{\alpha}^{\ast}(\SigmaZh). \label{t3}
\end{align}

Via \eqref{t1}, \eqref{t2} and \eqref{t3} we obtain that there is a natural homomorphism of vector bundles over $(\Lambda_1 \times \Lambda_2) \cap \Gamma_{mod} \ast \Gamma$:
\begin{align}
& (\SigmaZh \otimes I_{\Lambda_1}) \boxtimes (\SigmaGh \otimes I_{\Lambda_2}) \to \widetilde{\alpha}^{\ast}(I_{\Lambda} \otimes \SigmaZh) \otimes \Omega(\ker d \widetilde{\alpha} \cap T(\Lambda_1 \times \Lambda_2)). \label{t4}
\end{align}

Therefore by \eqref{t4} the twisted product is well-defined on the (standard) symbol spaces. 
\end{proof}

\section{Relative quantization and compositions}

\label{section:SymbolCalc}

We finally have assembled the machinery needed to define the microlocalization of the embedding $Y \hookrightarrow M$. As in Section \ref{section:KN} we realize our operators as matrices whose entries are Fourier integral operators associated to certain Lagrangians, obtained from Kohn-Nirenberg quantization of adapted symbol classes. 
We first define these symbol classes and the Lagrangians. We then perform the quantization. Lastly, we check that we have obtained a non-commutative completion. 

\subsection*{Langrangian submanifolds}

For $Y \hookrightarrow M$ we consider the embeddings of Section \ref{genmorph} which are inclusions of submanifolds (with corners) $\varrho \colon \H \hookrightarrow Z$ and $\sigma \colon Z \hookrightarrow \G$. 
Then we fix the following normal bundles associated to these inclusions.
\begin{itemize} 
\item The normal bundle $\N \to Y$ to the inclusion $Y \hookrightarrow M$. 

\item The normal bundle $\N^{Z} Y$ with regard to the inclusion $\varrho \circ u_{\partial} \colon Y \hookrightarrow Z$ where
$u_{\partial} \colon Y \hookrightarrow \H$ denotes the unit inclusion in the groupoid $\H$.

\item The normal bundle $\N^{\G} Y$ with regard to the inclusion $\sigma \circ \varrho \circ u_{\partial} \colon Y \hookrightarrow \G$. 

\end{itemize}

We consider the following Lagrangians
\begin{align*}
\Lambda_{\Psi} &:= \A^{\ast}(\G) \subset T^{\ast} \G \setminus 0, \\
\Lambda_{\partial} &:= \A^{\ast}(\H) \subset T^{\ast} \H \setminus 0, \\
\Lambda_g &:= (\N^{\G} \Delta_Y)^{\ast} \subset T^{\ast} \G \setminus 0, \\
\Lambda_b &:= (\N^{Z} \Delta_Y)^{\ast} \subset T^{\ast} Z \setminus 0, \\
\Lambda_c &:= (\N^{Z^t} \Delta_Y)^{\ast} \subset T^{\ast} Z^t \setminus 0. 
\end{align*}

\begin{Lem}
The Lagrangians $\Lambda_b, \Lambda_c = \Lambda_b^t$ and $\Lambda_g$ are local admissible family $\G$-relations.
\label{Grels}
\end{Lem}

\begin{proof}
Note that locality and admissibility are immediate by the definition of the Lagrangians as conormal bundles.
Consider the case $\Lambda_b = (\N^{Z} \Delta_Y)^{\ast}$, then $\Lambda_b$ is transverse to the diagonal $\Delta_Y \cong Y$ and hence 
$\tilde{\pi} \circ \tilde{\pr}_{|\alpha^{\ast}(\Lambda_b)}$, as defined in the previous section, is a submersion. Therefore $\Lambda_b$ is a family $\G$-relation by definition.
The case $\Lambda_c$ follows since $\Lambda_c = \Lambda_b^t$.  
Next consider $\Lambda_g = (\N^{\G} \Delta_Y)^{\ast}$ which is transverse to the diagonal $\Delta_Y \cong Y$ and hence
$\pi \circ \pr_{|m^{\ast}(\Lambda_g)}$, as defined in the previous section, is a submersion. Therefore $\Lambda_g$ is a family $\G$-relation. 
\end{proof}

We equip these Lagrangians with symbol spaces which are a not the usual H\"ormander classes, but behave slightly differently with respect to the natural groups of variables in these Lagrangians, see Section \ref{section:KN}.
In order to define these non-standard symbol spaces we first describe certain vector fields. 

Note that we have the non-canonical decompositions, cf. Proposition \ref{Prop:adm}
\begin{align}
& \N^{Z} \Delta_Y \cong \A_{\partial} \oplus \N, \label{decomp1} \\
& \N^{Z^t} \Delta_Y \cong \A_{|Y} \oplus \N \cong \A_{\partial} \oplus \N \oplus \N, \label{decomp2}. 
\end{align}

We call $\W = \Gamma(\A(\H))$ the \emph{tangent} vector fields, $\Gamma(\N)_{(1)}$ transversal \emph{of the first kind} and $\Gamma(\N)_{(2)}$ transversal \emph{of the second kind}. 

Fix $n := \dim(M), \ n-\nu = \dim(Y)$ and let $\pi_1 \colon \A^{\ast}(\H) \times \A^{\ast} \to \A^{\ast}(\H)$ be the first projection.

Consider on $(\N^{Z} \Delta_Y)^{\ast}$ vector fields $V$ such that $\pi_{1 \ast}(V) = 0$ which are called \emph{normal to $\A^{\ast}(\H)$} as well as vector fields $V$ such
that $\pi_{1\ast}(V)$ is well-defined vector field on $\A^{\ast}(\H)$ which are called \emph{lifted from $\A^{\ast}(\H)$}. 

\begin{Def}
Define the symbols space $S^{l,k}(\Lambda_b)$ for $k, l \in \Rr, \ a \in C^{\infty}(\Lambda_b)$ such that if $\{V_1, \cdots, V_j\}$ are vector fields on $\Lambda_b$, each $V_i$ homogenous of degree $0$ or $1$ with
\[
|V_1 \cdots V_j a(\alpha)| \leq C (1 + |\alpha|)^{l - j_1} (1 + |\pi_1(\alpha)|)^{k-j_2}. 
\]

Here $j_1$ is the number of first degree homogenous vector fields which are normal to $\A^{\ast}(\H)$ and $j_2$ denotes the number of first degree homogenous vector fields which are lifted
from $\A^{\ast}(\H)$. 

\label{Def:Sb}
\end{Def}

\begin{Def}
Define $S^{m,k,l}(\Lambda_g)$ for $m,k,l \in \Rr$ such that 
\[
|V_1 \cdots V_j a(\alpha)| \leq C(1 + |\pi_1(\alpha)|)^{m-j_1} (1 + |\pr(\alpha)|)^{k-j_2} (1 + |\pi_2(\alpha)|)^{l-j_3}. 
\]

Here $j_1$ denotes the number of first degree homogenous vector fields which are transversals of the first kind, $j_2$ denotes the number of first degree homogenous vector fields which 
are tangent and $j_3$ denotes the number of first degree homogenous vector fields which are transversals of the second kind. 

\label{Def:Sg}
\end{Def}

The \emph{classical} symbols in these classes, namely those that admit polyhomogeneous expansions, have an invariant description which uses a radial compactification of the Lagrangian manifolds. 
We use here the notation for the corresponding boundary defining functions and blow-down map $\beta$ as specified in the appendix \ref{sec:symbinv}.
\begin{Thm}
\emph{i)} The classical symbols space of boundary operators identifies via the radial blowup as
\[
S_{cl}^{k,l}(\Lambda_b) \cong \beta^*\left(\rho^{l-k}\rho_{ff}^{-l} C^{\infty}(\hat{\Lambda}_b)\right). 
\]

\emph{ii)} The classical symbols space of Green operators identifies 
\[
S_{cl}^{m,k,l}(\Lambda_g) \cong \beta^*\left(\rho_\tau^{l-k}\rho_{\tau^\prime}^{l-m}\rho_{ff}^{-l} C^{\infty}(\hat{\Lambda}_g)\right). 
\]

\label{Thm:Sinv}
\end{Thm}
\begin{Rem}
Since the embedding $Y \hookrightarrow M$ is assumed to be admissible, by the tubular neighborhood theorem \cite{AIN} we have $M \cong Y \times \Rr^{\nu}$ where $\nu$ is the codimension of the embedding. 
Hence $\G$ locally takes the form $\G = \G_Y^Y \times \Rr^{\nu} \times \Rr^{\nu}$
Similarly, $Z \cong \G_Y^Y \times \Rr^{\nu}$ and $\Zop \cong \Rr^{\nu} \times \G_Y^Y$. 
Under this assumption we can locally identify the calculus $\PsiZ$ with $\Psi(\G_Y^Y)$. 
\label{Rem:PsiZ}
\end{Rem}

We define the algebra $\Psi_{NC}(\G, \H)$ which depends on the generalized morphism $\H \dashrightarrow \G$. 

\begin{Def}
Let $\H \dashrightarrow \G$ be the generalized morphism of Lie groupoids which is implemented by the groupoid actions
on the fiberwise smooth spaces $Z \cong Z^t$, cf. Fig. \ref{fig:main}. Fix the orders $(m_g, k_g, k_c, k_b) \in \Rr^4$ and the codimension $\nu := \codim(Y)$.
We set $l_g = -m_g - k_g - k, \ l_b = -k_b - \frac{\nu}{2}, \ m_c = -k_c - \frac{\nu}{2}$ and we let 
$k_g > 0, \ k_c > 0, \ k_b > 0, \ m_g < -\frac{\nu}{2}, \ m_g + k_g > -\frac{\nu}{2}$.  

Denote by $\A(m_g, k_g, k_c, k_b)$ the set of matrices of the form
\[
\begin{pmatrix} U & C \\
B & S \end{pmatrix}, \ U \in \U_{cl}^{m_g, k_g, l_g}(\G), \ C \in \C_{cl}^{m_c, k_c}(\Zop), \ B \in \B_{cl}^{k_b, l_b}(Z), \ S \in \Psi_{cl}(\H).
\]
whose entries belong to the operator classes
\begin{align*}
(\textrm{singular Green})\qquad	& \U_{cl}^{m_g, k_g, l_g}(\G) := I_{cl}^{m_g, k_g, l_g}(\G, \Lambda_g) \\ 
(\textrm{boundary})\qquad		& \B_{cl}^{k_b,l_b}(Z) := I_{cl}^{k_b, l_b}(Z, \Lambda_b)\\ 
(\textrm{co-boundary})\qquad		& \C_{cl}^{m_c, k_c}(\Zop) := I_{cl}^{m_c, k_c}(\Zop, \Lambda_c)\\ 
(\textrm{pseudodifferential})\qquad		& \PsiZ := I_{cl}^{0}(Z, \widetilde{\Lambda}_{\partial}) 
\end{align*}
here $I_{cl}$ denote the spaces obtained from application of Kohn-Nirenberg quantization to the corresponding symbol classes.
We finally define the algebra of \textit{Green operators} as
\[
\Psi_{NC}(\G, \H) := \sum \A(m_g, k_g, k_c, k_b)
\]

where the sum is a non-direct sum over all tuples $(m_g, k_g, k_c, k_b)$.

\label{Def:operators}
\end{Def}

\begin{Thm}
Given the generalized morphism $\H \dashrightarrow \G$ then $\Psi_{NC}(\G, \H)$ is an associative $\ast$-algebra. 
\label{Thm:algebra}
\end{Thm}

\begin{proof}
The proof of this theorem is based on the convolution of the corresponding Lagrangians and the twisted product
of the symbols. However, we cannot directly use Theorem \ref{Thm:LMV6} because of the non-standard symbol classes employed.
We first check the conditions for convolvability, in particular Definition \ref{Def:conv} as well
as Theorem \ref{Thm:transpose}.

We refer to table \ref{tab:Lagrangians} in section \ref{section:KN} for the compositions of the local Lagrangian submanifolds as well as
table \ref{table:Symbols} for the twisted products of the symbols.

In order to study the composition more clearly, we express the distributional kernels of our operators using the Kohn-Nirenberg phase functions. These are obtained by use of the appropriate normal fibrations and uniquely determined by that choice.
Let $\Phi \colon \G \to \N^{\G} \Delta_Y$ be the singular Green normal fibration associated to the embedding $\Delta_Y \hookrightarrow \G$. 
Then $U \in \U(\G)$ using Remark \ref{Rem:PsiZ}, \emph{ii)} can be written in the form
\[
(U f)(\gamma', \gamma_n) = \int_{\G_{s(\gamma)}} \int_{(\N^{\G} \Delta_Y)_{r_{\partial}(\gamma')}^{\ast}} e^{i \scal{\Phi(\gamma \eta^{-1})}{\xi}} u(r_{\partial}(\gamma'), \xi) f(\eta) \,d\xi\,d\mu_{s(\gamma)}(\eta)
\]

where $f \in C_c^{\infty}(\G), \ \gamma = (\gamma', \gamma_n) \in \G = \H \times \Rr^{\nu} \times \Rr^{\nu}$ and with phase function $\varphi(\gamma, \theta) = \scal{\Phi(\gamma)}{\theta}$. 
Let $\Psi \colon Z \to \N^{Z} \Delta_Y$ be the normal fibration of the embedding $\Delta_Y \hookrightarrow Z$.
Then $B \in \B(Z)$ is given by 
\[
(B f)(z) = \int_{\G_{q(z)}} \int_{(\N^{Z} \Delta_Y)_{p(z)}^{\ast}} e^{i \scal{\Psi(z \gamma^{-1})}{\xi}} b(q(z), \xi) f(\gamma) \,d\xi \,d\mu_{q(z)}(\gamma)
\]
with phase function $\psi(z, \theta) = \scal{\Psi(z)}{\theta}$. 
Analogously, $C \in \C(\Zop)$ is given by 
\[
(Cf)(\gamma', \gamma_n) = \int_{Z_{s(\gamma)}} \int_{(\N^{\Zop} \Delta_Y)_{r_{\partial}(\gamma')}^{\ast}} e^{i\scal{\Psi^t(\gamma^{-1} z)}{\xi}} c(r_{\partial}(\gamma'), \xi) f(z) \,d\xi \, d \lambda_{s(\gamma)}^t(z)
\]
where $f \in C_c^{\infty}(Z), \ \gamma = (\gamma', \gamma_n) \in \G$.
We denote by $\Psi^t \colon \Zop \to \N^{\Zop} \Delta_Y$ the corresponding normal fibration with phase function
given by $\psi^t(z, \theta) = \scal{\Psi^t(z)}{\theta}$. 
We have that $\U(\G)$ is a $\Psi^{0}(\G)$-module, since $\Lambda_{\Psi} \circ \Lambda_g = \Lambda_g, \ \Lambda_g \circ \Lambda_{\Psi} = \Lambda_g$ and 
by the properties of the twisted products. By explicit calculation on principal symbol level using the operator representations above it is straight-forward to verify that the twisted symbol product formula is also valid for the non-standard symbol classes.\\
Each element of $\U(\G)$ can be put into the form $C \cdot B$ for a $B \in \B(\G, \H)$ and a $C \in \C(\G, \H)$.  
In order to see this, note that with $\Lambda_g = \Lambda_c \circ \Lambda_b$ we have
\[
S_{cl}^{m_1, k_1 + k_2, l_2}(\Lambda_g) = S_{cl}^{m_1, k_1}(\Lambda_c) \ast S_{cl}^{k_2, l_2}(\Lambda_b).
\]
Then by the explicit form of the phase functions and the quantization the same holds on the level of operators.
Hence
\[
\U_{cl}^{m_1, k_1 + k_2, l_2}(\G) \cong \C_{cl}^{m_1, k_1}(\Zop) \circ \B_{cl}^{k_2, l_2}(Z). 
\]
The topology on $\U_{cl}^{m_1, k_1 + k_2, l_2}(\G)$ is defined by the Fr\'echet topologies induced by the semi-norm
systems defined by the invariant estimates and the topology induced by the homogenous components in the classical expansions
of these operators. This shows that a composition of coboundary and boundary operators results in a singular Green operator. 
The other asssertion follows by observing that the assignment $(a, b) \mapsto a \ast b$ is a surjective map $S_{cl}^{m_1, k_1}(\Lambda_c) \times S_{cl}^{k_2, l_2}(\Lambda_b) \to S_{cl}^{m_1, k_1 + k_2, l_2}(\Lambda_g)$.
Again the same holds on the level of operators by the explicit quantization. 
Let $B \in \B(Z)$ and $C \in \C(\Zop)$, then $B \cdot C$ is a pseudodifferential operator on the boundary,
i.e. $B \cdot C \in \PsiZ$ which follows by direct verification using the explicit form of these operators and their phase functions. We also have $\Lambda_b \circ \Lambda_c = \Lambda_{\partial}$ and the twisted products on the symbol spaces.
The previous discussion yields in particular that $\U(\G)$ is closed under composition. 
For each $B \in \B(Z)$ there is a $C \in \C(\Zop)$ such that $B = C^{\ast}$ (formal adjoint) and vice versa.
In order to check that $\Lambda_b = \Lambda_c^{\ast}$ write the kernel of the operator $B$ in the form
\begin{align}
B(z) &= \int e^{-i \Psi(z, \theta)} b(z, \theta)\,d\theta. \label{Bdy}
\end{align}
The formal adjoint is written $C = B^{\ast}$
\begin{align}
C(z) &= \int e^{i \Psi^t(z, \theta)} c(z, \theta) \,d\theta. \label{CoBdy}
\end{align}
 
We note that $c(z, \theta) = \overline{b(\flip(z), d\flip_{\flip(z)} \theta)}$ is the corresponding symbol.
Additionally, $\Psi^t(z, \theta) = -\Psi(\flip(z), d\flip_{\flip(z)} \theta)$ defines a non-degenerate phase function.
The corresponding Lagrangian is given by
\[
\Lambda_{\Psi^t} = \{(z, \xi) \in T^{\ast} Z : (\flip(z), -(d^t \flip_z)(\xi)) \in \Lambda_{\Psi}\}. 
\]

Note that $(\U(\G), \ast)$ is an associative $\ast$-algebra.
This follows since $U^{\ast} = (CB)^{\ast} = B^{\ast} C^{\ast}$. 
The combination of these facts yields the closedness under composition. 
\end{proof}

Note that it is not hard to show more: The collection of operators of the type $\begin{pmatrix} P + U & C \\ B & S \end{pmatrix}$,
where $U, C, B, S$ are Green operators as defined in the previous Theorem and $P \in \Psi^{0}(\G)$ is a pseudodifferential operator on $\G$,
also forms an algebra.
To see this, observe that by the rules of composition $\U(\G)$, the class of singular Green operators, is a $\Psi^{0}(\G)$-module and
thus $\{P + U : P \in \Psi^{0}(\G), \ U \in \U(\G)\}$ forms an algebra. 

\begin{Lem}
\emph{1)} A singular Green operator $U \in \U(\G)$ is a bounded linear operator $U \colon L^2(\G) \to L^2(\G)$. 

\emph{2)} A boundary operator $B \in \B(\G, \H)$ is a bounded linear operator $B \colon L^2(\G) \to L^2(Z)$.

\emph{3)} A coboundary operator $C \in \C(\G, \H)$ is a bounded linear operator $C \colon L^2(Z) \to L^2(\G)$. 

\label{Lem:L2}
\end{Lem}

\begin{proof}
Since every element of $\U$ is a product of a boundary and coboundary operator, by the proof of Theorem \ref{Thm:algebra}, we see that \emph{2)} and \emph{3)} imply \emph{1)}.

Assertion \emph{2)} follows by an application of the proof of Theorem \ref{Thm:algebra} as follows.
Set $S = B \cdot B^{\ast}$ which is a pseudodifferential operator on the boundary and hence extends
to a bounded linear operator $S \colon L^2(Z) \to L^2(Z)$. Let $f \in L^2(Z)$ then
\[
\scal{B^{\ast} f}{B^{\ast} f}_{L^2(\G)} = \scal{f}{B B^{\ast} f}_{L^2(Z)} \leq C \|f\|^2.
\]

Hence $B^{\ast} \colon L^2(Z) \to L^2(\G)$ is bounded. Since every coboundary operator can be put in the form 
$C = B^{\ast}$ for a boundary operator $B$, by the proof of the Theorem, the continuity of coboundary operators follows.
The same argument proves the assertion for boundary operators. 
\end{proof}

\begin{table}[H]
\caption{Symbols}
\centering

\setlength\tabcolsep{1.5pt}

\begin{tabular}{| c || c | c | c | c | c |}
\hline \hline
$\ast$ & $S_{\partial}^{m_2}$ & $S_{\Psi}^{m_2}$ & $S_b^{k_2, l_2}$ & $S_c^{m_2, k_2}$ & $S_g^{m_2, k_2, l_2}$ \\ [0,5ex] \hline 
$S_{\partial}^{m_1}$ & $S_{\partial}^{m_1 + m_2}$ & $-$ & $S_b^{m_1 + k_2, l_2}$ & $S_c^{m_1 + m_2, k_2}$ & $-$ \\ [0,5ex] \hline
$S_{\Psi}^{m_1}$ & $-$ & $S_{\Psi}^{m_1 + m_2}$ & $-$ & $S_c^{m_1 + m_2, k_2}$ & $S_g^{m_1 + m_2, k_2, l_2}$ \\ [0,5ex] \hline
$S_b^{k_1, l_1}$ & $S_b^{k_1 + m_2, l_1}$ & $S_b^{k_1, l_1 + m_2}$ & $-$ & $S_{\partial}^{k_1 + l_1 + k_2 + m_2 + k}$ & $S_b^{k_1 + l_1 + k_2 + m_2 + k, l_2}$ \\ [0,5ex] \hline
$S_c^{m_1, k_1}$ & $S_c^{m_1, k_1 + m_2}$ & $-$ & $S_g^{m_1, k_1 + k_2, l_2}$ & $-$ & $-$ \\ [0,5ex] \hline
$S_g^{m_1, k_1, l_1}$ & $-$ & $S_g^{m_1, k_1, l_1 + m_2}$ & $-$ & $S_c^{m_1, k_1 + l_1 + k_2 + m_2 + k}$ & $S_g^{m_2, k_1 + l_1 + k_2 + m_2 + k, l_2}$ \\ [1ex]
\hline
\end{tabular}
\label{table:Symbols}
\end{table}

\section{Functoriality properties}

\subsection*{Quantization scheme}

The final task is to show that our construction furnishes a quantized non-commutative completion of $Y\hookrightarrow M$ as defined in Section \ref{sec:noncommcomp}.
We have defined the following functors which preserve Morita equivalence, cf. Theorem \ref{Thm:functoriality} and Theorem \ref{Prop:mod}.
\[
\begin{tikzcd}[row sep=huge, column sep=huge, text height=1.5ex, text depth=0.25ex]
      \displaystyle \LA_b \arrow[swap]{d}{\Gamma} & \arrow{l}{\A^{\ast}} \LG_b \arrow{d}{C^{\ast}} \arrow{ld}{\mathrm{CDW}} \\
\SG_b & C_b^{\ast},  
\end{tikzcd}
\]

given on the level of objects by the assignments
\begin{align}
& \LG_b \ni \G \mapsto \A^{\ast}(\G) \in \LA_b, \label{Astar} \\
& \LG_b \ni \G \mapsto C^{\ast}(\G) \in C_b^{\ast}, \label{Cstar} \\
& \LA_b \ni P \mapsto \Gamma(P) \in \SG_b, \label{Gamma} \\
& \LG_b \ni \G \mapsto T^{\ast} \G \in \SG_b. \label{CDW}
\end{align}

Using the notation of section \ref{nccompl} our scheme for the construction of the operators for relative elliptic problems
can be summarized in the following diagram, Figure \ref{Fig:4}
\begin{figure}[H]
\caption{Order $0$ case}
\centering
\begin{tikzcd}[row sep=huge, column sep=huge, text height=1.5ex, text depth=0.25ex]
    \displaystyle 
    \LA_b \arrow[swap]{dr}{\mathrm{KN}}
    \arrow[bend right=60,swap]{drr}{\Gamma} 
&   \arrow[swap]{l}{\widetilde{\lift}}
    \EmbV
    \arrow{d}{nc}
    \arrow{r}{\lift}
&   \arrow[bend right=30,swap]{ll}{\A^{\ast}}
    \LG_b
    \arrow{d}{\mathrm{CDW}}
\\
&   C_b^{\ast}
&   \arrow{l}{\mathrm{FIO}}
    \SG_b
\end{tikzcd}
\label{Fig:4}
\end{figure}

Where the designation \emph{KN} stands for Kohn-Nirenberg quantization, by which we mean the process
of obtaining an operator from a symbol contained in the symbol spaces defined below.
The right arrow denoted $\lift$ is the lifting of an embedding to a generalized morphism of appropriate Lie groupoids and the left arrow $\liftcl$ is
the corresponding \emph{classical lift}, i.e. the lifting of the embedding to a generalized morphism of dual Lie algebroids.
We have commutativity in the upper part of the diagram, i.e.: $\A^{\ast} \circ \lift = \widetilde{\lift}$.
The square on the right hand side refers to the construction of operators from the generalized morphism of Lie groupoids
from which we obtain a generalized morphism of symplectic (\emph{CDW}) groupoids.
The designation \emph{FIO} stands for the definition of Fourier integral operators as recalled in Section \ref{FIO}.
Finally, \emph{nc} stands for the process of obtaining a non-commutative completion of
$C^{\ast}$-algebras from an embedding of manifolds with corners.

For the pseudodifferential calculus on a Lie groupoid $\G$ we have $\Psi^{-\infty}(\G) \cong C_c^{\infty}(\G)$ and
the $\L(L^2)$-completion of the latter convolution algebra yields: $\overline{C_c^{\infty}(\G)} \cong C^{\ast}(\G)$. 
Hence for the case of residual operators (or smoothing operators) the functorial diagram simplifies to Figure \ref{Fig:5}:

\begin{figure}[H]
\caption{Smoothing operators}
      \centering
     \begin{tikzcd}[row sep=huge, column sep=huge, text height=1.5ex, text depth=0.25ex]
          \displaystyle \LA_b \arrow[swap]{dr}{DQ = KN} & \arrow[swap]{l}{\widetilde{\lift}} \EmbV \arrow{d}{nc} \arrow{r}{\lift} & \arrow[bend right=50,shorten >=10pt]{ll}{\A^{\ast}} \arrow{dl}{C^{\ast}} \LG_b \\
	& C_b^{\ast} &  
     \end{tikzcd}
\label{Fig:5}
\end{figure}
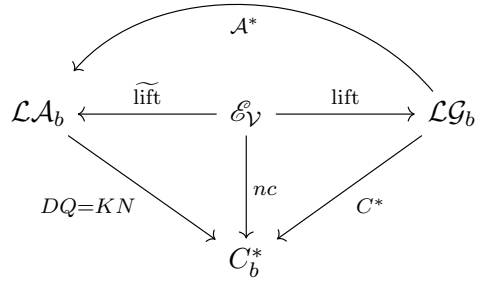

The diagram commutes: $nc = C^{\ast} \circ \lift$ and $\A^{\ast} \circ \lift = \widetilde{\lift}$. 
We recover the Kohn-Nirenberg or right quantization $DQ = KN$, where $DQ$ stands for deformation quantization as described
in \cite{LR}. 

\subsection*{Representation}

A \emph{representation} of our calculus in order $0$ is a $\ast$-homomorphism $\varrho_{\Phi} \colon \Phi(\G, \H) \to \End\begin{pmatrix} C^{\infty}(M) \\ \oplus \\ C^{\infty}(Y) \end{pmatrix}$.
The $\ast$-representation is defined in analogy to the $\ast$-representation for the calculus of pseudodifferential operators on Lie manifolds as constructed in \cite{ALN}. 
It is defined as the restriction of the homomorphism $\varrho$ defined \eqref{repr}. Recall that it is characterized by the \emph{defining property} 
\[
\left(\varrho_{\Phi}(A) \circ \begin{pmatrix} r \\ r_{\partial} \end{pmatrix} \right) \begin{pmatrix} f \\ g \end{pmatrix} = A\begin{pmatrix} f \circ r \\ g \circ r_{\partial} \end{pmatrix}, \ f \in C^{\infty}(M), \ g \in C^{\infty}(Y), \ A \in \Psi(\G, \H) 
\]

where $r \colon \G \to M, \ r_{\partial} \colon \H \to Y$ denote the corresponding range maps (local diffeomorphisms). 
Set $\Phi_{NC}(j) := \Psi_{NC}(\lift(j))$, where $\lift$ denotes the natural transformation associating to $j$ the generalized morphism of Lie groupoids $\H \dashrightarrow \G$. 
By the Sobolev continuity of the operator calculus we can define the $L^2$-completion $\Phi_{NC}(j) \hookrightarrow \overline{\Phi}(j) \subset \L\begin{pmatrix} L^2(\G) \\ \oplus \\ L^2(\H) \end{pmatrix}$.
We show next that we obtain a non-commutative completion functor $\overline{\Phi}$ from $\Emb_{\V}$ into the category of $\ast$-algebras.  

\begin{Thm}
The functor $\overline{\Phi}$ furnishes a quantized non-commutative completion. 
\label{Thm:nccompl}
\end{Thm}

\begin{proof}
We have to study the functor $\overline{\Phi} \colon \Emb_{\V} \to C_b^{\ast}$. 
The interesting point is the map on morphisms. Hence we have to show that for any given admissible embedding
$Y \hookrightarrow M$ the algebra $\overline{\Phi}(j)$ implements a generalized morphism of $C^{\ast}$-algebras.
Fix an admissible embedding $Y \hookrightarrow M$, i.e. a morphism in $\EmbV$. 

Consider the $L^2$-completed order-$0$ algebra $\overline{\Phi}(j)$ which can be written in terms of the matrix $\begin{pmatrix} \U & \C \\ \B & \Psi_{\partial} \end{pmatrix}$.
Here $\U$ denotes the $L^2$-completion of the algebra of singular Green operators and $\Psi_{\partial}$ denotes the $L^2$-completion
of the algebra of pseudodifferential operators on $Y$ (the Lie calculus).

By the rules of composition for $\Phi(j)$ we have $\U \cdot \C \subset \C$ and $\Psi_{\partial} \cdot \B \subset \B$.
We can check that $\C$ is a right-Hilbert $\U$-module and $\B$ is a right Hilbert $\Psi_{\partial}$-module.
Additionally, we need to show that there is a $\ast$-homomorphism $\varphi \colon \U \to \L_{\Psi_{\partial}}(\B)$ taking
values in the adjointable operators. There is a scalar product $_{\Psi_{\partial}}\scal{\cdot}{\cdot} \colon \B \times \B \to \Psi_{\partial}$ such that
$_{\Psi_{\partial}}\scal{\varphi}{\psi} = _{\Psi_{\partial}}\scal{\psi}{\varphi}$ and $_{\Psi_{\partial}}\scal{\varphi}{\varphi} \geq 0$. 
Additionally $\|\varphi\|^2 = \|_{\Psi_{\partial}}\scal{\varphi}{\varphi}\|$ defines a norm with regard to which $\B$ is complete.
In our case we define $_{\Psi_{\partial}}\scal{B_1}{B_2} = B_1 \circ B_2^{\ast} \in \Psi_{\partial}$. 
Also $\varphi \colon \U \to \L_{\Psi_{\partial}}(\B)$ is given by $\varphi(G) \colon \B \to \B$ where $G$ is a singular Green operator.
The latter is defined as $\varphi(G)(B) = (G \cdot B^{\ast})^{\ast} = B \cdot G^{\ast}$. 
Then check that $_{\Psi_{\partial}}\scal{B_1}{\varphi(G) B_2} = _{\Psi_{\partial}}\scal{\varphi(G)^{\ast} B_1}{B_2}$ which holds
since by definition
\begin{align*}
_{\Psi_{\partial}}\scal{B_1}{\varphi(G) B_2} &= B_1 (\varphi(G) B_2)^{\ast} = B_1 (B_2 G^{\ast})^{\ast} \\
&= B_1 G B_2^{\ast} = (\varphi(G)^{\ast} B_1) B_2^{\ast} \\
&= _{\Psi_{\partial}}\scal{\varphi(G)^{\ast} B_1}{B_2}. 
\end{align*}

Secondly, $\varphi$ is a homomorphism because $\varphi(G_1 G_2) B = B G_2^{\ast} G_1^{\ast} = \varphi(G_1) \varphi(G_2) B$.
Also we have that
\[
_{\Psi_{\partial}}\scal{\varphi(G)  B_1}{B_2} = _{\Psi_{\partial}}\scal{B_1}{\varphi(G^{\ast}) B_2}.
\]

Note that the left hand side equals $(B_1 G^{\ast}) B_2^{\ast}$ while the right hand side equals $B_1 (B_2 G)^{\ast} = B_1 G^{\ast} B_2^{\ast}$.
Hence $\varphi$ is a well-defined $\ast$-homomorphism.
This yields the desired generalized morphism in $C^{\ast}$ and we have shown that $\Phi$ is a noncommutative completion.
The condition \emph{i)} follows by Lemma \ref{Lem:L2}. 
Condition \emph{ii)} follows via the generalized morphism $T^{\ast} \H \dashrightarrow T^{\ast} \G$ obtained by functoriality and
the corresponding construction of the operator classes.
In order to check condition \emph{iii)} we fix the embedding functor \ $\widehat{}_b \ \colon C^{\ast} \hookrightarrow C_b^{\ast}$ from Proposition \ref{Prop:inclCstar}.
Let $j \in \Mor(\E_{\V})$ be the admissible embedding $j \colon Y \hookrightarrow M$ and denote by $\W := \{V_{|Y} : V \in \V, \ V_{|Y} \ \text{tangent to} \ Y\}$
the induced Lie structure of $Y$.  
Also fix the $\ast$-homomorphisms $\varrho \colon \U(\G) \to \U_{\V}(M, Y)$ and $\varrho_{\partial} \colon \Psi(\H) \to \Psi_{\W}(Y)$.
We then define the natural representation $\varrho_{\Phi} \colon \Phi \to \Phi_{\V}$ and check that it is surjective natural transformation.
The following diagram of generalized morphisms in $C^{\ast}$ commutes
\[
\xymatrix{
\Psi(\H) \ar@{-->}[d]_{\varrho_{\Phi|\H}} \ar@{-->}[r]^{\Phi(j)} & \U(\G) \ar@{-->}[d]_{\varrho_{\Phi|\G}} \\
\Psi_{\W}^{\ast}(Y) \ar@{-->}[r]^{\Phi_{\V}(j)} & \U_{\V}(M, Y).
}
\]

Where we define $\varrho_{\Phi|\G} := \widehat{}_b \ \circ \varrho$ and $\varrho_{\Phi|\H} := \widehat{}_b \ \circ \varrho_{\partial}$. 
In particular the surjective $\ast$-homomorphism $\varrho_{\partial} \colon \Psi(\H) \to \Psi_{\W}(Y)$ yields
a $\Psi_{\W}^{\ast}(Y) - \Psi(\H)$ bimodule.
Also the surjective $\ast$-homomorphism $\varrho \colon \U(\G) \to \U_{\V}(M)$ yields a $\U_{\V}(M) - \U(\G)$ bimodule.
By definition $\Phi_{\V}(j)$ is a $\U_{\V} - \Psi_{\W}$ bimodule and $\Phi(j)$ is a $\U - \Phi(\H)$ bimodule.
Then 
\begin{align*}
\Phi_{\V}(j) \circ \varrho_{\Phi|\H} &= \Phi_{\V}(j) \hat{\otimes}_{\Psi_{\W}} \widehat{\varrho_{\partial}}, \\
\varrho_{\Phi|\G} \circ \Phi(j) &= \widehat{\varrho} \hat{\otimes}_{\U(\G)} \Phi(j). 
\end{align*}

Here $\hat{\otimes}$ denotes the Rieffel tensor product, cf. section \ref{Cstarb} and \cite{L}. 
The surjectivity of $\varrho_{\Phi}$ follows from the surjectivity of the strict morphisms $\varrho_{\partial}, \ \varrho$. 
\end{proof}

\begin{Cor}
Let $\Phi \colon \EmbV \to C_b^{\ast}$ be a quantized non-commutative completion. Then for each embedding $j \in \Mor(\EmbV)$
the generalized morphism $\Phi(j) \colon \U \dashrightarrow \Psi_{\partial}$ induces a hereditary 
$\ast$-homomorphism $\Psi_{\partial} \to \U$.
\label{Cor:nccompl}
\end{Cor}

This follows by \cite[Lem. 3]{AMMS}. We provide details for completeness.

\begin{proof}
Let $j \in \Mor(\EmbV)$ be given. Then write $\Phi(j)$ in terms of the matrix of $C^{\ast}$-algebras 
$\begin{pmatrix} \U & \C \\ \B & \Psi_{\partial} \end{pmatrix}$. Let $(j_{\ast}, j^{\ast})$ be a generating pair.
Then if $0 \leq U \leq j_{\ast} S j^{\ast}$ for $S \in \Psi_{\partial}$ we have $U = j_{\ast} j^{\ast} U j_{\ast} j^{\ast}
= j_{\ast} (j^{\ast} U j_{\ast}) j^{\ast}$. Since $\Phi(j)$ is a bimodule correspondence it follows that 
the term inside the brackets is contained in $\Psi_{\partial}$. Hence $U$ has the form $j_{\ast} \tilde{S} j^{\ast}$ for some
$\tilde{S} \in \Psi_{\partial}$. This completes the proof that the $\ast$-homomorphism $\Psi_{\W} \to \U$ given by
$S \mapsto j_{\ast} S j^{\ast}$ is a hereditary $\ast$-homomorphism.
\end{proof}

\subsection*{Universality}

In order to address the universality property of non-commutative completions we now restrict attention to the subcategory $\Emb^H \subset \Emb_{\V}$ of those embeddings $j \colon Y \hookrightarrow M$ which
have the additional property that there are corresponding Lie groupoids, integrating the Lie structures of $Y$ and $M$ respectively, which are Hausdorff. Denote by $\LG_b^H$ the category consisting of Hausdorff Lie groupoids
as objects and generalized morphisms as arrows. Denote by $\LA_b^H$ the category consisting of Lie algebroids associated to Hausdorff Lie groupoids and generalized morphisms as arrows.  
\begin{Def}
A covariant functor $\Phi \colon \Emb^{H} \to \C_b^{\ast}$ is called \emph{$H$-quantized non-commutative completion} if it is the minimal covariant functor such that the conditions \emph{i)} and \emph{iii)} of Definition \ref{Def:nccompl} hold and $H-$\emph{ii)} if in addition
there are natural transformations $\lift_{qu} \colon \Emb^H \to \LG_b^H$ and $\lift_{cl} \colon \Emb^H \to \LA_b^H$ such that the functorial diagram
\begin{figure}[H]
\begin{tikzcd}
\Emb^H \drar[bend right, "\lift_{qu}", swap]\rar[bend left, "\lift_{cl}"] 
 & \LA_b^H  \\
 & \LG_b^H \uar["\A^{\ast}"]
\end{tikzcd}
\end{figure}
commutes and preserves Morita equivalences.
\label{Def:Hquantized}
\end{Def}

\begin{Rem}
By minimality in the above definition we mean that given $H$-quantized non-commutative completion $\Phi \colon \Emb^{H} \to C_b^{\ast}$,
if $\widetilde{\Phi} \colon \Emb^{H} \to C_b^{\ast}$ is another covariant functor such that properties \emph{i)}, $H$-\emph{ii)} and \emph{iii)} hold, then 
there is an essentially surjective natural transformation $F \colon \widetilde{\Phi} \to \Phi$. 
\label{Rem:Hquantized}
\end{Rem}

\begin{Thm}[Universality]
Let $\Phi \colon \Emb^H \to \C_b^{\ast}$, $\widetilde{\Phi} \colon \Emb^H \to C_b^{\ast}$ be $H$-quantized non-commutative completions. Then there is a natural isomorphism $\theta \colon \Phi \to \widetilde{\Phi}$. 
\label{Thm:univ}
\end{Thm}

\begin{proof}
Given an $H$-quantized non-commutative completion $\Phi$, let us compare $\Phi$ with $\Phi_{NC} := \Psi_{NC} \circ \lift$. 
By property \emph{iii)} and Theorem \ref{Thm:nccompl} there are surjective natural transformations $\varrho_{\Phi_{NC}} \colon \Phi_{NC} \to \Phi_{\V}$ and $\varrho_{\Phi} \colon \Phi \to \Phi_{\V}$, both compatible with the $\ast$-homomorphisms $\varrho$ and $\varrho_{\partial}$. 
We define next a natural transformation $\theta$ fitting into the commuting diagram 

\begin{tikzcd}
\Phi \ar{d}{\varrho_{\Phi}} \ar{r}{\theta} & \Phi_{NC} \ar{dl}{\varrho_{\Phi_{NC}}} \\
\Phi_{\V}
\end{tikzcd}

Let $j \in \Emb^H$ and using property $H-$\emph{ii)} consider the commuting diagram in $C_b^{\ast}$ with the associated $C^{\ast}$-algebras \\
\begin{tikzcd}[row sep=huge, column sep=huge, text height=1.5ex, text depth=0.25ex]
\C(\H) \ar[dashed]{d}{\theta_{|\H}} \ar[dashed]{r}{\Phi(j)} & \C(\G) \ar[dashed]{d}{\theta_{|\G}} \\
\Psi(\H) \ar[dashed]{r}{\Phi_{NC}(j)} & \U(\G). 
\end{tikzcd}

In particular we can assume that both $\H$ and $\G$ are Hausdorff Lie groupoids. By \cite{N} the $\ast$-homomorphism $\varrho_{\partial} \colon \Psi(\H) \iso \Psi_{\W}(Y)$ is an isomorphism. 
The same proof can be applied to the representation homomorphism $\varrho$ restricted to $\U(\G)$. Again, by the Hausdorff property this homomorphism is bijective. In particular we obtain a $\ast$-isomorphism
$\varrho \colon \U(\G) \iso \U_{\V}(M, Y)$. Note that here condition \emph{i)} of a non-commutative completion enters in the proof. Using again Theorem \ref{Thm:nccompl} we obtain the commuting diagram in $C_b^{\ast}$ \\
\begin{tikzcd}[row sep=huge, column sep=huge, text height=1.5ex, text depth=0.25ex]
\C(\H) \ar[bend right=80, dashed]{dd}{\theta_{|\H}} \ar[dashed]{d}{\varrho_{\Phi_{|\H}}} \ar[dashed]{r}{\Phi(j)} & \C(\G) \ar[dashed]{d}{\varrho_{\Phi_{|\G}}} \ar[bend left=80, dashed]{dd}{\theta_{|\G}} \\
\Psi_{\W}(Y) \ar[dashed]{r}{\Phi_{\V}(j)} & \U_{\V}(M, Y) \\
\Psi(\H) \ar[dashed]{u}{\Phi_{{NC}_{\H}}} \ar[dashed]{r}{\Phi_{NC}(j)} & \U(\G) \ar[dashed]{u}{\varrho_{\Phi_{{NC}_{|\G}}}}
\end{tikzcd}

The $\ast$-isomorphisms $\varrho$, $\varrho_{\partial}$ furnish by Proposition \ref{Prop:inclCstar} the bimodule isomorphisms $\varrho_{\Phi_{{NC}_{|\H}}} = \widehat{}_b \circ \varrho_{\partial}$ and 
$\varrho_{\Phi_{{NC}_{|\G}}} = \widehat{}_b \circ \varrho$ (apply the Proposition to the homomorphism and its inverse separately to obtain the bimodule correspondence and the inverse bimodule correspondence). 
Hence $\theta_{|\H}$ and $\theta_{|\G}$ are defined via the previous diagram by inverting these isomorphisms and composition with $\varrho_{\Phi_{|\G}}$ and $\varrho_{\Phi_{|\H}}$.
Therefore $\theta$ is a well-defined, essentially surjective natural transformation. This shows in particular that $\Phi_{NC}$ is an $H$-quantized non-commutative completion.
By the minimality of $\Phi$ we can now fix an essentially surjective natural transformation $F \colon \Phi_{NC} \to \Phi$. 
Consider the following diagram in $C_b^{\ast}$: \\
\begin{tikzcd}[row sep=huge, column sep=huge, text height=1.5ex, text depth=0.25ex]
\Psi_{\W}(Y) \ar[dashed]{r}{\Phi_{\V}(j)} & \U_{\V}(M, Y) \\
\C(\H) \ar[dashed]{u}{\varrho_{\Phi_{|\H}}}  \ar[dashed]{r}{\Phi(j)} & \C(\G) \ar[dashed]{u}{\varrho_{\Phi_{|\G}}} \\
\Psi(\H) \ar[dashed]{u}{F_{|\H}} \ar[bend left=80, dashed]{uu}{\varrho_{\Phi_{{NC}_{|\H}}}} \ar[dashed]{r}{\Phi_{NC}(j)} & \U(\G) \ar[dashed]{u}{F_{|\G}} \ar[bend right=80, dashed]{uu}{\varrho_{\Phi_{{NC}_{|\G}}}} 
\end{tikzcd}

Since both $\varrho_{\Phi_{NC}}$ and $\varrho_{\Phi}$ are defined via the $\ast$-homomorphisms $\varrho$ and $\varrho_{\partial}$ restricted to their respective domains,
we obtain by naturality of $F$ that the left- and rightmost sub-diagrams commute in $C_b^{\ast}$, i.e. $\varrho_{\Phi_{{NC}_{|\G}}} = \varrho_{\Phi_{|\G}} \circ F_{|\G}$
and $\varrho_{\Phi_{{NC}_{|\H}}} = \varrho_{\Phi_{|\H}} \circ F_{|\H}$. Since by definition $\theta = (\varrho_{\Phi_{NC}})^{-1} \circ (\varrho_{\Phi})$ we obtain that
$F$ is a right inverse natural transformation to $\theta$. Hence $\theta$ is natural isomorphism of $\Phi$ with $\Phi_{NC}$. 
\end{proof}

\section{Concluding remarks}

\emph{1.} Starting from a category of embeddings of a particular class of manifolds, in our case compact manifolds with corners which arise as compactifications of certain complete Riemannian manifolds, we have
in this work described a functorial process (non-commutative completion) to associate operator algebras to such embeddings. Then we have shown that the axioms of a non-commutative completion yield that
this functorial process is universal, i.e. there is - in a precise categorical sense - only one way to associate an operator algebra, fulfilling certain minimal properties.
This suggests generalizations of our scheme to other categories of embeddings. An obvious question is whether we can extend our scheme to the entire category of compactifications of complete Riemannian manifolds with bounded geometry. 

\emph{2.} In the case where the base geometry has no singularity consider a compact manifold with boundary $(X, \partial X)$ and take $M$ to be the double of $X$ at $\partial X$.
Then the operator algebra associated to the embedding $\partial X \hookrightarrow M$ is $\ast$-isomorphic to the Boutet de Monvel algebra of operators, if these operators are extended (using the Seeley extension operator
in the normal direction) to the double $M$. We refer to \cite{G}, Section 2.4. for the proof in this special case. In the case with singularities we consider a Lie manifold $(X, \A_{+}, \V_{+})$ with boundary $(Y, \B, \W)$ and take 
$(M, \V, \A)$ as the double. Lie manifolds with boundary and the doubling construction are described in \cite{AIN}. In this case the embedding of Lie manifolds $j \colon Y \hookrightarrow M$ is transverse by definition and without loss of generality assumed to be admissible. The associated non-commutative completion $\Phi_{NC}(j)$ is $\ast$-isomorphic to the Boutet de Monvel algebra (for order $m \leq 0$), extended to the double Lie manifold $(M, \A, \V)$, as shown in \cite{B2}. As a consequence it follows that Boutet de Monvel's algebra is universal for manifolds with singularities.  

\appendix

\section{Coordinate-free characterization of symbol spaces}
\label{sec:symbinv}

In this appendix, we give a coordinate-free description of the symbol spaces used in \cite{NS}. The symbol spaces are actually shown to correspond to spaces of smooth functions multiplied by certain weights. The construction is outlined in Melrose and Rochon \cite{MR} in the case of manifolds with fibred boundaries. We apply this general machinery in some detail to the symbol spaces on $\Lambda_\partial$, $\Lambda_\Psi$, $\Lambda_b$, $\Lambda_c$ and $\Lambda_g$.
\subsection{Radial compactification of $\RRd$ and classical symbols}
As a model example, we want to realize the space $S^m_\cl(\RRd)$ of symbols in one variable as a space of smooth functions.\\
We recall that the radial compactification of $\RRd$ may be constructed as follows: topologically, we identify $\overline{\RRd}$ with $\BBd=\{\eta\in\RRd\,|\,|\eta|\leq 1\}$, the unit disk. We write $\iota:\RRd\hookrightarrow (\BBd)^o$ for an isomorphism that is given, near the boundary, by radial inversion $\xi\mapsto \frac{\xi}{|\xi|}$. \\
We now equip $\BBd$ with a compatible smooth structure. In any open neighbourhood of the origin, we equip $(\BBd)^o$ with the $\Sm$-structure it inherits from $\BBd\subset \RRd$. Near the boundary $|\eta|=1$ we say a function $f$ is smooth if $f(\frac{\eta}{|\eta|^2})$ is smooth in polar coordinates, that is we choose $\rho=1-|\eta|$ and the angular variables $\varphi_j\in \SSSd$ as smooth coordinates. The function $\rho$ then also serves as a boundary defining function if smoothly continued to a positive function into the interior.\\
The previous procedure thus serves precisely to make the vector fields $r\partial_r$ and $\partial_{\varphi_j}$ into smooth vector fields on all of $\RRd$ ``up to infinity'' in the sense that their image under radial compactification, $-\rho\partial_\rho$ and $\partial_{\varphi_j}$ are smooth vector fields near the boundary $\partial\BBd$. Then we have the following equivalence:
\begin{Lem}
The space of symbols $S_\cl(\RRd)=\bigcup_m S_\cl^m(\RRd)$ equipped with its usual LF-topology is, as a filtered differential algebra, isomorphic to $\bigcup_{m\in\RR} \rho^{-m} \Sm(\BBd)$ with the equivalence given by pullback with the radial compactification map $\iota$.
\end{Lem}
\begin{proof}
We briefly recall the argument behind the equivalence: One first observes that the weight $\rho^{-m}$ is under radial inversion equivalent to $(1+|\xi|)^m$ and thus the image of smooth functions up to the boundary under $(\iota^{-1})^*$ are in particular bounded.\\
One then notices that the vector fields $\rho\partial_\rho$ and $\partial_{\varphi_j}$ preserve  $\rho^{-m} \Sm(\BBd)$. However, this implies for $f\in \rho^{-m} \Sm(\BBd)$ the estimate  
$$ (\frac{1}{\rho}\partial_{\varphi_j})^\beta \partial_\rho^\alpha f(\rho,\varphi)\in \rho^{-m-|\alpha|-|\beta|} \Sm(\BBd).$$
The coordinate vector fields $\partial_{\xi_i}$ in the standard coordinates on $\RRd$ may be rewritten in polar coordinates as $\partial_{\xi_i}=\frac{1}{r}r\partial_r+\sum c_{ij}\frac{1}{r}\partial_{\varphi_{j}}$. 
This implies the symbolic estimates
$$|\partial_{\xi}^{\alpha} f(\iota^{-1}(y))|\lesssim (1+|\xi|)^{m-|\alpha|}.$$
Finally, the classical expansion of a symbol $a=(\iota^{-1})^*f$ into homogeneous functions in $|\xi|$ is nothing but a Taylor-expansion in $\rho$ at $\rho=0$, since $\rho=|xi|^{-1}$ for $|\xi|$ suitably large.\\
Thus we may characterize $S^m_\cl(\RRd)$ as a space of smooth functions. By nuclearity we have for open sets $U$ that
$\Smc(U)\otimes S^m_\cl(\RRd)$
is isomorphic to the space of compactly supported classical symbols in $U$ and consequently $S^m_{\cl,c}(U\times\RRd)$ is isomorphic to $\bigcup_{m\in\RR} \rho^{-m} \Smc(\RRd\times\BBd))$. 
\end{proof}

Using a local trivialization by a partition of unity, we can thus characterize the classical space on a given vector bundle.
The previous construction thus yields a coordinate-free description of $S_\cl(\Lambda_{\partial})$ and $S_\cl(\Lambda_{\Psi})$, since these are just classical symbols on vector bundles.

\subsubsection{Blow-up of submanifolds}
We now want to mimic the preceding construction to obtain the symbol spaces of Section \ref{section:SymbolCalc}. For that, we note that $S^m_\cl(\Lambda_X)$ and $S^m_\cl(\Lambda_\partial)$ are simply given as the classical symbols on the respective bundles. In order to realize the remaining symbol spaces as above, we have to suitably compactify the fibres of $\Lambda_b$, $\Lambda_c$ and $\Lambda_g$. This is achieved by a blow-up construction, see \cite{Mel} and \cite[Chapter 18.3]{HIII}.\\
We recall some properties of the blow-up construction. Let $X$ a $\Sm$-manifold and $Y$ a $\Sm$ submanifold. Then the blow-up of $X$ at $Y$ is $[X,Y]=(X\setminus Y)\cup (T_YX\setminus TY)/(TY)$ as a set which is then equipped with a particular differential structure. The essential feature of this differential structure is that it coincides with the structure on $X$ outside of $Y$ and is a coordinate-free way to introduce polar coordinates ``around $Y$''. Furthermore, there is a natural map from $[X,Y]$ to $X$, the blow-down map $\beta$, which reverses the changes into singular coordinates.\\
In coordinates, the situation is as follows: assume $Y$ is locally given by $(x_1,\dots,x_k)=0$ where $X=(x_1,\dots,x_d)$. Then $[X,Y]$ can be modelled as $[0,\infty)\times\SSS^{k-1}\times \RR^{n-k}\ni (\rho,\phi_1\dots\phi_k,x_{k+1},\dots,x_{d})$. In a neighbourhood of a point where $\phi_j\neq 0$, a function on $[X,Y]$ is then called smooth if it is so in $\rho \phi_j$, $\phi_i/\phi_j$ ($j\neq i$), and in $(x_{k+1},\dots,x_d)$.\\ 
This essentially means that the vector fields given on $X\setminus Y$ which diverge as $1/\rho$ are turned into smooth vector fields on the blow-up.
\subsubsection{Coordinate-free description of $S(\Lambda_b)$ and $S(\Lambda_c)$}
We now perform the construction in the case of $\Lambda_b$, see \cite{MR}. In local coordinates in some neighbourhood $U$, $\Lambda_b$ is given by
\begin{align*}
\Lambda_b& =\{(x,p,x^\prime,t^\prime,p^\prime,\tau^\prime\,|\,x=x^\prime,\ t^\prime=0,\ p=p^\prime\}
\end{align*}
This means that $\Lambda_b$ can be locally parametrized by $\{(x,p^\prime,\tau^\prime\}\in U\times\RR^{d_1}\times\RR^{d_2}$.
We pass to the joint radial compactification of $\RR^{d_1}\times\RR^{d_2}\cong \RR^{d_1+d_2}$, that is $\BB^{d_1+d_2}$, meaning to the joint fibre-wise radial compactification $\overline{\Lambda_b}$ of $\Lambda_b$. Therein, we blow-up the boundary $B$ of the submanifold where $p^\prime=0$, i.e. we blow up at ``$|\tau^\prime|=\infty$ and $p^\prime=0$''.
%
%
\begin{Rem}
Note that the submanifold with $p^\prime=0$ is the submanifold $T^*Y\times N^*_Y M$. 
\end{Rem}
On the resulting space $[\overline{\Lambda_b},B]$, we have two smooth boundary defining functions $\rho$ and $\rho_{ff}$. In the old coordinates, these may be chosen as follows:
\begin{itemize}
\item $\rho$ equals $|p^\prime|^{-1}$ for $|p^\prime|>c$ and is extended to a positive function in the interior
\item $\rho_{ff}$ equals $\rho^{-1}(|p^\prime|^2+|\tau|^2)^{-2}$ for $|p^\prime|+|\tau|>c$ and smoothly extended.
\end{itemize} 
This construction achieves that the image of the following vector fields, given in coordinates on $\Lambda$ as
$$p^\prime_j\partial_{p^\prime_k}\quad \tau^\prime_l\partial_{\tau^\prime_i} \quad p^\prime_j\partial_{\tau^\prime_i}$$
with arbitrary indices $j,k,i,l$ are smooth ``up to infinity'' on the blow-up space. Notice that the first two collections of vector fields, $p^\prime_j\partial_{p^\prime_k}$ and $\tau^\prime_l\partial_{\tau^\prime_i}$ were already smooth on $\overline{\Lambda_b}$, whereas $p^\prime_j\partial_{p^\prime_i}$ is so because it may be written as $\frac{p^\prime_j}{\tau_l}(\tau_l\partial_{\tau_i})$.\\
The fact that these vector fields leave $\rho^{m-k}\rho_{ff}^{-m} C^\infty([\overline{\Lambda_b},B])$ invariant then implies the following estimates 
$$|\partial_x^\alpha \partial_{p^\prime}^\beta \partial_{\tau^\prime}^\gamma a(x,p^\prime,\tau^\prime)|\lesssim (1+|p^\prime|)^{k-|\beta|}(1+|p^\prime|+|\tau^\prime|)^{l-|\gamma|}.$$
as well as polyhomogeneous expansions. This means we can set
$$S^{k,l}_{\cl}(\Lambda_b)=\beta^*\left(\rho^{l-k}\rho_{ff}^{-l} C^\infty([\overline{\Lambda_b},B]\right).$$
The space $S^{k,l}_{\cl}(\Lambda_c)$ may be constructed in the same manner, starting from $\Lambda_c$. 
\subsubsection{Coordinate-free description of $S(\Lambda_g)$}
In the case of $S^{m,k,l}_\cl\Lambda_g,$ we only sketch the construction. Locally, $S^{m,k,l}_\cl\Lambda_g$ is given as
\begin{align*}
\Lambda_g& =\{(x,t,p,\tau,x^\prime,t^\prime,p^\prime,\tau^\prime\,|\,x=x^\prime,\ t=t^\prime=0,\ p=p^\prime\},
\end{align*}
the situation is similar. Here, we want to obtain the symbolic estimates
$$|\partial_x^\alpha \partial_{p^\prime}^\beta \partial_{\tau^\prime}^\delta \partial_{\tau}^\gamma a(x,p^\prime,\tau,\tau^\prime)|\lesssim (1+|p^\prime|)^{k-|\beta|}(1+|p^\prime|+|\tau|)^{m-|\gamma|}(1+|p^\prime|+|\tau^\prime|)^{l-|\delta|}.$$
Here, however, we first note that we have to treat $\tau$ and $\tau^\prime$ independently. This means we separately compactify the fibres $K_{\tau}=\RR_{p^\prime}^{d}\times\RR^{d_1}_\tau$ as well as in $K_{\tau^\prime}=\RR_{p^\prime}^{d}\times\RR^{d_1}_{\tau^\prime}$. This means that outside of suitable neighbourhoods of the origin we pass to polar coordinates in
$$\frac{1}{\sqrt{|\tau^\prime|^2+|p^\prime|^2}},\quad \frac{1}{\sqrt{|\tau|^2+|p^\prime|^2}},\quad + \text{angular variables.}$$
We obtain a manifold with corners, $\overline{\Lambda_g}$. In it, we then again blow up the submanifold given by $p^\prime=0$ at the respective boundaries $B$. Finally, we then obtain the characterization
$$S^{k,l,m}_{\cl}(\Lambda_g)=\beta^*\left(\rho_\tau^{l-k}\rho_{\tau^\prime}^{l-m}\rho_{ff}^{-l} C^\infty([\overline{\Lambda_g},B]\right).$$

\section*{Acknowledgements}

We thank the SFB1085 \emph{Higher Invariants} at University of Regensburg for fascilitating this collaboration. In addition we thank Bernd Ammann, Jean-Marie Lescure, Victor Nistor and Elmar Schrohe for useful discussions.


\small
\begin{thebibliography}{99}
\bibitem{ANS} J. Aastrup, R. Nest, E. Schrohe, \emph{Index Theory for Boundary Value Problems via Continuous Fields of $C^{\ast}$-algebras}, J. Funct. Anal., 257, 2645-2692 (2009). 
\bibitem{ANS2} J. Aastrup, R. Nest, E. Schrohe, \emph{A Continuous Field of $C^*$-algebras and the Tangent Groupoid for Manifolds with Boundary}, J. Funct. Anal. 237, 482-506 (2006). 
\bibitem{AMMS} J. Aastrup, S. T. Melo, B. Monthubert, E. Schrohe, \emph{Boutet de Monvel's calculus and groupoids I}, J. Noncommut. Geom.  4  (2010), no. 3, 313--329.
\bibitem{ACN} B. Ammann, C. Carvalho, V. Nistor, \emph{Regularity for Eigenfunctions of Schr\"odinger Operators}, Letters in Math. Physics, Vol. 101, Issue 1, pp. 49-84.
\bibitem{AIN} B. Ammann, A. Ionescu, V. Nistor, \emph{Sobolev spaces on Lie manifolds and regularity for polyhedral domains}, Doc. Math. 11, 161-206 (2006).
\bibitem{ALN} B. Ammann, R. Lauter, V. Nistor, \emph{Pseudodifferential operators on manifolds with a Lie structure at infinity}, Ann. of Math. 165, 717-747 (2007).
\bibitem{ALN2} B. Ammann, R. Lauter, V. Nistor, \emph{On the geometry of Riemannian manifolds with a Lie structure at infinity}, Int. J. Math. and Math. Sciences 4: 161-193.
\bibitem{ALNV} B. Ammann, R. Lauter, V. Nistor, A. Vasy, \emph{Complex powers and non-compact manifolds}, Comm. Part. Diff. Eq. 29, no. 5/6 671-705 (2004).
\bibitem{B} K. Bohlen, \emph{Boutet de Monvel's calculus on singular manifolds}, C. R. Acad. Sci. Paris, Ser. I 354 (2016), 239-243.
\bibitem{B2} K. Bohlen, \emph{Boutet de Monvel operators on Lie manifolds with boundary}, Adv. Math. 312 (2017), 234--285.
\bibitem{B3} K. Bohlen, \emph{Index formulae for mixed boundary conditions on manifolds with corners}, arXiv:1704.00535.
\bibitem{BM} Louis Boutet de Monvel, \emph{Boundary problems for pseudo-differential operators}, Acta Math., 126(1-2):11-51, 1971.
\bibitem{C} A. Connes, \emph{Noncommutative Geometry}, Academic Press, 1994.
\bibitem{CDW} A. Coste, P. Dazord, A. Weinstein, \emph{Groupo\"\i des symplectiques}, Publications du {D}\'epartement de {M}ath\'ematiques. {N}ouvelle {S}\'erie. {A}, {V}ol.\ 2, Publ. D\'ep. Math. Nouvelle S\'er. A, 87-2, Univ. Claude-Bernard, Lyon, 1987.
\bibitem{D} C. Debord, \emph{Holonomy groupoids of singular foliations}, J. Differential Geom. 58 (2001), no. 3, 467--500.
\bibitem{DLR} C. Debord, J.-M. Lescure, F. Rochon, \emph{Pseudodifferential operators on manifolds with fibred corners}, A para\^itre aux Annales de l'institut Fourier (2015).
\bibitem{DS} C. Debord, G. Skandalis, \emph{Adiabatic groupoid, crossed product by $\R_+^*$ and Pseudodifferential calculus}, Adv. in Math. 257 (2014), 66-91.  
\bibitem{DS2} C. Debord, G. Skandalis, \emph{Blowup constructions for Lie groupoids and a Boutet de Monvel type calculus}, arXiv:1705.09588.
\bibitem{G} G. Grubb, \emph{Functional calculus of pseudodifferential boundary problems}, Second edition, Progress in Mathematics, 65, Birkh\"auser Boston, Inc., Boston, MA, 1996.
\bibitem{Guil} L. Guillaume, \emph{G\'eom\'etrie non-commutative et calcul pseudodiff\'entiel sur les vari\'et\'es \`a coins fibr\'es}, Ph.D. thesis, Universit\'e Paul Sabatier Toulouse 3, 2012. 
\bibitem{H} E. Hawkins, \emph{A groupoid approach to quantization}, J. Symplectic Geom. 6  (2008), no. 1, 61--125.
\bibitem{HHyp} L.~H\"ormander, \emph{Lectures on nonlinear hyperbolic differential equations}, Math\'ematiques \& Applications (Berlin), 26. Springer-Verlag, Berlin, 1997.
\bibitem{HIII} L. H\"ormander, \emph{The analysis of linear partial differential operators. III. Pseudo-differential operators}, Reprint of the 1994 edition, Classics in Mathematics, Springer, Berlin, 2007. 
\bibitem{J} D. Joyce, \emph{On manifolds with corners}, Advances in geometric analysis, 225--258, Adv. Lect. Math. (ALM), 21, Int. Press, Somerville, MA, 2012. 
\bibitem{L0} N. P. Landsman, \emph{Operator algebras and Poisson manifolds associated to groupoids}, Comm. Math. Phys.  222  (2001), no. 1, 97--116.
\bibitem{L} N. P. Landsman, \emph{Quantized reduction as a tensor product}, Quantization of Singular Symplectic Quotients, Progress in Mathematics, 138-163, Birkh\"auser, 2001.
\bibitem{L2} N. P. Landsman, \emph{The Muhly-Renault-Williams Theorem for Lie Groupoids and its Classical Counterpart}, Letters in Mathematical Physics, 09/2000; 54(1). 
\bibitem{LR} N. P. Landsman, B. Ramazan, \emph{Quantization of Poisson algebras associated to Lie algebroids}, Contemp. Math., Amer. Math. Soc., Vol. 282, Providence, 2001. 
\bibitem{LMV} J.-M. Lescure, D. Manchon, S. Vassout, \emph{About the convolution of distributions on groupoids}, arXiv:1502.02002.
\bibitem{LN} R. Lauter, V. Nistor, \emph{Analysis of geometric operators on open manifolds: A groupoid approach}, Progress in Mathematics, Vol. 198, 2001, pp 181-229.
\bibitem{LV} J.-M. Lescure, S. Vassout, \emph{Fourier integrals operators on Lie groupoids}, arXiv:1601.00932.
\bibitem{MM} R. Mazzeo, R. Melrose, \emph{Pseudodifferential operators on manifolds with fibred boundary}, Asian Journal of Mathematics 2 No. 4 (1999) pp. 833-866.
\bibitem{Mc} K. C. H. Mackenzie, \emph{General Theory of Lie Groupoids and Lie Algebroids}, London Math. Soc., Lecture Note Series 213, 2005.
\bibitem{Mel} R. Melrose, \emph{The Atiyah-Patodi-Singer index theorem}, A. K. Peters, Ltd., Boston, Mass., 1993.
\bibitem{Mel2} R. B. Melrose \emph{Pseudodifferential operators, corners and singular limits}, talk at ICM in Kyoto, 1990. 
\bibitem{MR} R. Melrose, F. Rochon, \emph{Index in {$K$}-theory for families of fibred cusp operators}, $K$-Theory 37, no. 1-2, 25--104, 2006.
\bibitem{MSS} S. Melo, T. Schick, E. Schrohe, \emph{A $K$-Theoretic Proof of Boutet de Monvel's Index Theorem}, J. Reine Angew. Math. 599, 217-233 (2006). 
\bibitem{M} B. Monthubert, \emph{Groupoids and pseudodifferential calculus on manifolds with corners}, Journal of Functional Analysis 199(2003), 243-286.
\bibitem{MRW} P. Muhly, J. Renault, D. Williams, \emph{Equivalence and isomorphism for groupoid $C^{\ast}$-algebras}, J. Operator Th. 17 (1987), 3-22.
\bibitem{MO} M. Macho-Stadler, M. O'Ouchi, \emph{Correspondence of groupoid $C^{\ast}$-algebras}, J. Operator Theory, 42(1999) 103-119.
\bibitem{N} V. Nistor, \emph{Pseudodifferential operators on non-compact manifolds and analysis on polyhedral domains}, Contemp. Math. 366:307-328, Amer. Math. Soc., Providence, RI, 2005.  
\bibitem{NS} V. Nazaikinskii, B. Sternin, \emph{Relative elliptic theory}, Operator Theory: Adv. and Appl., Birkh\"auser Verlag, Basel / Switzerland, Vol. 151, 495-560 (2004).
\bibitem{NWX} V. Nistor, A. Weinstein, P. Xu, \emph{Pseudodifferential Operators on Differential Groupoids}, Pacific J. Math. 189 (1999), 117-152.
\bibitem{P} A. L. T. Paterson, \emph{The equivariant analytic index for proper groupoid actions}, K-theory, 32(2004), 198-230.
\bibitem{SSS} A. Savin, B.-W. Schulze, B. Sternin, \emph{On the Homotopy Classification of Elliptic Boundary Value Problems}, Vol. 126 Operator Theory: Advances and Applications pp 299-305.
\bibitem{SS} B. Sternin, V. E. Shatalov, \emph{Relative elliptic theory and the Sobolev problem. (Russian)}, translated from Mat. Sb. 187 (1996), no. 11, 115--144 Sb. Math. 187 (1996), no. 11, 1691--1720.
\bibitem{S} S. R. Simanca, \emph{Pseudo-differential Operators}, Pitman Research Notices 236, 1990. 
\bibitem{Xu} P. Xu, \emph{Morita Equivalent Symplectic Groupoids}, Symplectic Geometry, Groupoids, and Integrable Systems, Vol. 20, Mathematical Sciences Research Institute Publications, pp 291-311.
\bibitem{Xu2} P. Xu, \emph{Morita equivalence of Poisson manifolds}, Commun. Math. Phys. 142 (1991), 493-509.
\end{thebibliography}
\end{document}